\documentclass[a4paper,onepage,reqno,11pt]{amsart}
\usepackage[top=20mm,right=30mm,bottom=20mm,left=30mm]{geometry}

\usepackage{pdfsync}
\usepackage{graphicx}
\usepackage{amsmath}
\usepackage{amssymb}
\usepackage{amsfonts}
\usepackage{amsthm}
\usepackage{amstext}
\usepackage{amsbsy}
\usepackage{amsopn}
\usepackage{amscd}
\usepackage{enumerate}
\usepackage{color}
\usepackage[colorlinks=true,linkcolor=red,citecolor=blue]{hyperref}
\usepackage[hyperpageref]{backref}
\usepackage{multirow}
\usepackage{lipsum}
\usepackage{rotating}
\usepackage{longtable}
\usepackage{float}
\usepackage{lscape}
\usepackage{pdflscape}
\usepackage[]{algorithm2e}

\newtheorem{theorem}{Theorem}[section]
\newtheorem{corollary}[theorem]{Corollary}

\newtheorem{lemma}[theorem]{Lemma}
\newtheorem{proposition}[theorem]{Proposition}

\theoremstyle{definition}

\numberwithin{equation}{section}

\newcommand{\G}{\mathrm{G}}

\newcommand{\GL}{\mathrm{GL}}
\newcommand{\SL}{\mathrm{SL}}

\newcommand{\Sp}{\mathrm{Sp}}

\newcommand{\SU}{\mathrm{SU}}

\newcommand{\PSL}{\mathrm{PSL}}
\newcommand{\PSU}{\mathrm{PSU}}

\newcommand{\PSp}{\mathrm{PSp}}
\newcommand{\AGL}{\mathrm{AGL}}

\newcommand{\GaL}{\mathrm{\Gamma L}}

\newcommand{\GaO}{\mathrm{\Gamma O}}

\newcommand{\AGaL}{\mathrm{A}\Gamma\mathrm{L}}
\newcommand{\POm}{\mathrm{P \Omega}}
\newcommand{\J}{\mathrm{J}}
\newcommand{\HS}{\mathrm{HS}}
\newcommand{\McL}{\mathrm{McL}}
\newcommand{\Suz}{\mathrm{Suz}}
\newcommand{\He}{\mathrm{He}}
\newcommand{\Ru}{\mathrm{Ru}}
\newcommand{\ON}{\mathrm{O'N}}
\newcommand{\Co}{\mathrm{Co}}
\newcommand{\Fi}{\mathrm{Fi}}
\newcommand{\HN}{\mathrm{HN}}
\newcommand{\Th}{\mathrm{Th}}
\newcommand{\BM}{\mathrm{BM}}
\newcommand{\M}{\mathrm{M}}
\newcommand{\Ly}{\mathrm{Ly}}

\newcommand{\A}{\mathrm{A}}

\newcommand{\E}{\mathrm{E}}

\renewcommand{\S}{\mathrm{S}}

\newcommand{\C}{\mathrm{C}}
\newcommand{\D}{\mathrm{D}}

\newcommand{\Aut}{\mathrm{Aut}}

\newcommand{\PG}{\mathrm{PG}}

\renewcommand{\u}{\textsf{u}}
\renewcommand{\l}{\textsf{l}}

\newcommand{\B}{\mathrm{B}}
\newcommand{\N}{\mathrm{N}}
\newcommand{\F}{\mathrm{F}}
\newcommand{\Z}{\mathrm{Z}}
\newcommand{\dev}{\mathrm{dev}}

\newcommand{\Zbb}{\mathbb{Z}}

\newcommand{\Fbb}{\mathbb{F}}

\newcommand{\Dmc}{\mathcal{D}}
\newcommand{\Bmc}{\mathcal{B}}

\newcommand{\Pmc}{\mathcal{P}}
\newcommand{\Cmc}{\mathcal{C}}
\newcommand{\Smc}{\mathcal{S}}
\newcommand{\Nmc}{\mathcal{N}}

\newcommand{\e}{\epsilon}

\renewcommand{\leq}{\leqslant}
\renewcommand{\geq}{\geqslant}

\renewcommand{\mod}[1]{\ (\mathrm{mod}{\ #1})}
\newcommand{\imod}[1]{\allowbreak\mkern4mu({\operator@font mod}\,\,#1)}

\usepackage{tabularx}
\makeatletter
\def\hlinewd#1{%
\noalign{\ifnum0=`}\fi\hrule \@height #1 %
\futurelet\reserved@a\@xhline}
\makeatother

\begin{document}
\title[Flag-transitive symmetric designs]{Affine groups as  flag-transitive and point-primitive automorphism groups of symmetric designs}

\author[S.H. Alavi]{Seyed Hassan Alavi}%
\thanks{Corresponding author: S.H. Alavi}
\address{Seyed Hassan Alavi, Department of Mathematics, Faculty of Science, Bu-Ali Sina University, Hamedan, Iran.}%
\email{alavi.s.hassan@basu.ac.ir; alavi.s.hassan@gmail.com}

\author[M. Bayat]{Mohsen bayat}%
\address{Mohsen Bayat, School of Mathematics, Institute for Research in Fundamental Sciences (IPM), Tehran, Iran.}%
\email{mbayat@ipm.ir;  mohsenbayat3989@gmail.com}

\author[A. Daneshkhah]{Ashraf Daneshkhah}%
\address{Ashraf Daneshkhah, Department of Mathematics, Faculty of Science, Bu-Ali Sina University, Hamedan, Iran.}%
\email{adanesh@basu.ac.ir}

\author[A. Montinaro]{Alessandro Montinaro}%
\address{Alessandro Montinaro, Dipartimento di Matematica e Fisica ``E. De Giorgi'', University of Salento, Lecce, Italy. }%
\email{alessandro.montinaro@unisalento.it}

\subjclass[]{20B25, 05B05, 05B25}%
\keywords{affine group, flag-transitive,  point-primitive, automorphism group, symmetric design.}
\date{\today}%

\begin{abstract}
In this article, we investigate symmetric designs admitting a flag-transitive and point-primitive affine automorphism group. We prove that if an automorphism group $G$ of a symmetric $(v,k,\lambda)$ design with $\lambda$ prime is point-primitive of affine type, then $G=2^{6}{:}\S_{6}$ and $(v,k,\lambda)=(16,6,2)$, or $G$ is a subgroup of $\AGaL_{1}(q)$ for some odd prime power $q$. In conclusion, we present a classification of flag-transitive and point-primitive symmetric designs with $\lambda$ prime, which says that such an incidence structure is a projective space $\PG(n,q)$, it has parameter set $(15,7,3)$, $(7, 4, 2)$, $(11, 5, 2)$, $(11, 6, 2)$, $(16,6,2)$ or $(45, 12, 3)$, or $v=p^d$ where $p$ is an odd prime and the automorphism group is a subgroup of $\AGaL_{1}(q)$.
\end{abstract}

\maketitle

\section{Introduction}\label{sec:intro}

Symmetric designs admitting flag-transitive automorphism groups are of most interest.
In 1987, Kantor \cite{a:Kantor-87-Odd} classified flag-transitive symmetric $(v,k,1)$ designs known as projective planes. More generally, a celebrated result on the classification of linear spaces was announced in 1990 \cite{a:BDDKLS90}, and the proof of this result was given in a series of papers with noting that some important special flag-transitive linear spaces namely projective and affine planes have drawn much attention, see for example, Foulser \cite{a:Foulser-64,a:Foulser-64-sol}, Kantor \cite{a:Kantor-85-Homgeneous,a:Kantor-87-Odd} and Wagner \cite{a:Wagner-1965-affine}.
The approach to the classification of linear spaces starts with a result of Higman and McLaughlin \cite{a:HigMcL-61}, according to which any flag-transitive group of a linear space must act primitively on the points of the space.
Buekenhout, Delandtsheer and Doyen \cite{a:BDD-1988} using O'Nan-Scott theorem for finite primitive permutation groups proved that
the socle of a flag-transitive group of linear space is either an elementary abelian group (affine type), or a non-abelian simple group (almost simple type), and the classification is proved by analyzing each of these possibilities.
The proof in the affine case is due to Liebeck \cite{a:Liebeck-98-Affine}, and the almost simple case has been treated by several authors namely Buekenhout, Delandtsheer, Doyen,  Kleidman and Saxl \cite{a:Delandtsheer-86,a:Delandsheer-Lin-An-2001,a:Kleidman-Exp,a:Saxl2002}.

The situation for a larger $\lambda$ is rather different as there are examples of flag-transitive and point-imprimitive symmetric designs which we briefly discuss later. Focusing on the point-primitive symmetric designs, for biplanes (symmetric designs with $\lambda=2$) and triplanes (symmetric design with $\lambda=3$), O'Reilly Regueiro~\cite{a:Regueiro-reduction} proved a similar reduction result that a flag-transitive and point-primitive automorphism group of a biplane or a triplane must be of almost simple or affine type, and then in a series of papers (including point-imprimitive examples), she gave a classification of biplanes excluding $1$-dimensional affine automorphism groups \cite{a:Regueiro-alt-spor,a:Regueiro-reduction,a:Regueiro-classical,a:Regueiro-Exp}. Dong, Fang and Zhou studied flag-transitive automorphism groups of triplanes, and in conclusion, they determined all such possible symmetric designs excluding $1$-dimensional affine automorphism groups  \cite{a:Zhou-lam3-affine,a:Zhou-lam3-spor,a:Zhou-lam3-alt,a:Zhou-lam3-excep,a:Zhou-lam3-classical}. Recently, Z. Zhang, Y. Zhang and Zhou \cite{a:Zhou-sym-lam-prime} generalized O'Reilly Regueiro's result to prime $\lambda$, and proved that a flag-transitive and point-primitive automorphism group of a symmetric design with $\lambda$ prime must be of almost simple or affine type.
In order to classify all such designs, we need to analyze these two types of primitive automorphism groups in details. The almost simple case has been handled in    \cite{a:ABD-PrimeLam,a:ABD-Exp,a:ADM-PrimeLam-An}. In this paper, we deal with the affine type automorphism groups, and our main result is

\begin{theorem}\label{thm:main}
	Let $\Dmc$ be a nontrivial symmetric design with $\lambda$ prime admitting a flag-transitive and point-primitive automorphism group $G$ of affine type. Then  $G\leq\AGaL_1(q)$, or $\Dmc$ is a symmetric $(16,6,2)$ design with full automorphism group $2^{4}:\S_{6}$ and point-stabilizer $\S_{6}$.
\end{theorem}

As an immediate consequence of Theorem \ref{thm:main}, by the main results of \cite{a:ABD-PrimeLam,a:ABD-Exp,a:ADM-PrimeLam-An,a:Biliotti-CP-sym-affine,a:Zhou-lam3-affine,a:Regueiro-reduction}, we are able to present a classification result on flag-transitive and point-primitive symmetric designs with $\lambda$ prime.

\begin{corollary}\label{cor:main}
	Let $\Dmc$ be a nontrivial symmetric $(v,k,\lambda)$ design with $\lambda$ prime admitting a flag-transitive and point-primitive automorphism group $G$. Then one of the following holds:
	\begin{enumerate}[\rm \quad (a)]
		\item  $\Dmc$ is a projective spaces $\PG(n-1,q)$ with $\lambda = (q^{n-2}-1)/(q-1)$ prime and $\PSL_n(q)\unlhd G\leq \Aut(\PSL_n(q))$, or $\Dmc$ and $G$ are as in  lines {\rm 1-5} of {\rm Table~\ref{tbl:main}};
		\item  $\Dmc$ is  the point-hyperplane design of the projective space $\PG_{2}(3,2)$ with parameters $(15,7,3)$ and $G=\A_{7}$ with the point-stabilizer $\PSL_{3}(2)$;
		\item $\Dmc$ has parameters $(16,6,2)$ and its full automorphism group is $2^{4}\S_{6}$ with point-stabilizer $\S_{6}$;
		\item $\Dmc$ has $q=p^{d}$ points with $p\neq \lambda$ an odd prime and $G$ is a subgroup of the group $\AGaL_{1}(q)$ of $1$-dimensional semilinear affine transformations.
	\end{enumerate}
\end{corollary}

\begin{table}
	\centering
	\small
	\caption{Some flag-transitive and point-primitive symmetric designs with $\lambda$ prime.}\label{tbl:main}
	\begin{tabular}{llllllll}
		\hline\noalign{\smallskip}
		Line &
		\multicolumn{1}{l}{$v$} &
		\multicolumn{1}{l}{$k$} &
		\multicolumn{1}{l}{$\lambda$} &
		\multicolumn{1}{l}{$X$} &
		\multicolumn{1}{l}{$G_{\alpha}$}  &
		\multicolumn{1}{l}{$G$}  &
		\multicolumn{1}{l}{References} \\
		\noalign{\smallskip}\hline\noalign{\smallskip}
		$1$ &
		$7$ &
		$4$ &
		$2$ &
		$\PSL_{2}(7)$ &
		$\S_{4}$ &
		$\PSL_{2}(7)$ &
		\cite{a:ABD-PSL2, b:Handbook} \\
		$2$ &
		$11$ &
		$5$ &
		$2$ &
		$\PSL_{2}(11)$&
		$\A_{5}$&
		$\PSL_{2}(11)$ &
		\cite{a:ABD-PSL2, b:Handbook} \\
		$3$ &
		$11$ &
		$6$ &
		$3$ &
		$\PSL_{2}(11)$&
		$\A_{5}$&
		$\PSL_{2}(11)$ &
		\cite{a:ABD-PSL2, b:Handbook} \\
		$4$ &
		$45$ &
		$12$ &
		$3$ &
		$\PSU_{4}(2)$  &
		$2^{.}(\A_{4}\times \A_{4}).2$ &
		$\PSU_{4}(2)$ &
		\cite{a:Braic-2500-nopower,a:Dempwolff2001, a:Praeger-45-12-3}\\
		$5$ &
		$45$ &
		$12$ &
		$3$ &
		$\PSU_{4}(2)$  &
		$2^{.}(\A_{4}\times \A_{4}).2:2$ &
		$\PSU_{4}(2):2$
		&\cite{a:Braic-2500-nopower,a:Dempwolff2001, a:Praeger-45-12-3}\\
		%
		%
		\noalign{\smallskip}\hline
	\end{tabular}
\end{table}
%


We briefly give some information in Section~\ref{sec:example} below about the symmetric designs that occur in our study. In order to prove Theorem~\ref{thm:main}, we include all possible symmetric designs for $\lambda = 2, 3$ obtained in \cite{a:Zhou-lam3-affine,a:Regueiro-reduction} and therein references, and so we can assume that $\lambda\geq 5$. Since $\lambda(v-1)=k(k-1)$, and $\lambda$ is prime, $\lambda$ divides $k$ or $k-1$. In the latter case, $\lambda$ is coprime to $k$, and hence the main result of \cite{a:Biliotti-CP-sym-affine} implies that $v=p^d$ is odd and $G\leq \AGaL_1(p^d)$ is point-primitive and block-primitive. Note that $v=p^d$ is odd as $\lambda$ is an odd prime. So we only need to focus on the case where $\lambda$ divides $k$ and $p>2$.  In Proposition \ref{prop:lam-p}, we first threat the case where $\lambda=p$, and show that there is no symmetric design with parameter set $(p^{d},k,p)$. This allows us to assume that $\lambda\neq p$. Since $G$ is point=primitive, $G\leq \AGL_d(p)$ and $H:=G_{0}$ is an irreducible subgroup of $\GL_d(p)$. Then by Aschbacher's Theorem \cite{a:Aschbacher-84}, the subgroup $H$ belongs to one of the eight geometric families $\Cmc_{i}$ ($i=1,\ldots ,8$) of subgroups of $G$, or to the family where $H^{(\infty)}$ (the last term in the derived series of $H$) is quasisimple.
A rough description of the $\Cmc_i$ families is given in Table \ref{tbl:max}. We follow the description of these subgroups as in \cite[Chapter 4]{b:KL-90}, and analyze each of these possibilities in detail in Section \ref{sec:proof}. In quasisimple case, our proofs essentially rely on the machinery developed by Liebeck in~\cite{a:Liebeck-98-Affine}.

As noted above, in this paper, we deal with point-primitive automorphism groups of flag-transitive symmetric designs. On the other hand, when an automorphism group of a symmetric design is point-imprimitive,
inspired by the result of O'Reilly Regueiro \cite[Theorem 1]{a:Regueiro-reduction}, for a fixed $\lambda$, Praeger and Zhou \cite{a:Praeger-imprimitive} proved that either $k\leq\lambda(\lambda-3)/2$, or the parameter set of the design falls into three cases in which all other parameters given in terms of $\lambda$. Considering the fact that all flag-transitive point-imprimitive symmetric designs with either $\lambda \leq 10$ or $k> \lambda(\lambda-3)/2$ are classified \cite{a:Mandic-Sym-Imp-lam10, a:Montinaro-sym-imp-ellgeq3,a:Montinaro-sym-imp-ell2}, and this leaves an open problem of investigating the flag-transitive and point-imprimitive symmetric designs with $\lambda>10$ prime and $k\leq \lambda(\lambda-3)/2$.

\begin{table}
	\centering
	\small
	\caption{The geometric subgroup collections.}\label{tbl:max}
		\begin{tabular}{clllll}
			\noalign{\smallskip}\hline\noalign{\smallskip}
			Class & Rough description\\
			\noalign{\smallskip}\hline\noalign{\smallskip}
			$\Cmc_1$ & stabilizers of subspaces of $V$\\
			$\Cmc_2$ & stabilizers of decompositions $V=\bigoplus_{i=1}^{t}V_i$, where $\dim V_i = a$\\
			$\Cmc_3$ & stabilizers of prime index extension fields of $\mathbb{F}$\\
			$\Cmc_4$ & stabilizers of decompositions $V=V_1 \otimes V_2$\\
			$\Cmc_5$ & stabilizers of prime index subfields of $\mathbb{F}$\\
			$\Cmc_6$ & Normalizers of symplectic-type $r$-groups in absolutely irreducible representations\\
			$\Cmc_7$ & stabilizers of decompositions $V=\bigotimes_{i=1}^{t}V_i$, where $\dim V_i = a$\\
			$\Cmc_8$ & stabilizers of non-degenerate forms on $V$\\ \noalign{\smallskip}\hline\noalign{\smallskip}
		\end{tabular}
\end{table}

\subsection{Examples}\label{sec:example}

For the symmetric designs in parts (b) and (c) of Corollary~\ref{cor:main}, Hussain \cite{a:Hussain-1945} proved that there exist exactly three non-isomorphic symmetric designs with parameter set $(16,6,2)$. O'Reilly Regueiro \cite[ p. 139]{a:Regueiro-reduction} described the three designs and showed that exactly two of them are flag-transitive. 
The design in Corollary~\ref{cor:main}(c) or Theorem \ref{thm:main} is the point-primitive one with automorphism group $2^4:\S_{6}$ as a subgroup of $\AGL_{4}(2)$, see also \cite{a:ADP-bi,a:Regueiro-reduction}. The unique symmetric $(15,7,3)$ design in Corollary~\ref{cor:main}(b) with automorphism group $\A_{7}$ is viewed as a projective space \cite{a:BDD-1988} arose from the study of symmetric designs with $2$-transitive automorphism groups, see \cite{a:Kantor-85-2-trans}. 
For the symmetric designs in Table~\ref{tbl:main}, the unique symmetric $(7,4,2)$ design is indeed the complement of Fano plane with $G=\PSL(2,7)$ as its flag-transitive and point-primitive automorphism group, see \cite{a:ABD-PSL2,b:Handbook}.  There exists a unique symmetric design with parameter set $(11,5,2)$ which is a Hadamard design and it can also be constructed as a Paley difference set. The full automorphism group of this design is $\PSL(2,11)$ which is flag-transitive and point-primitive with $\A_{5}$ as its point-stabilizer \cite{a:Kantor-85-2-trans,a:Regueiro-reduction}). This group contains the Frobenius group $11:5$ as a subgroup of $\AGL_1(11)$ which is an affine automorphism group of this Hadamard design.  
The complement of this design is the unique symmetric $(11,6,3)$ design whose full automorphism group $\PSL(2,11)$ is also flag-transitive and point-primitive with $\A_{5}$ point-stabilizer, see \cite{a:Zhou-lam3-classical}. 
For the symmetric $(45, 12, 3)$ design, Praeger ~\cite[Theorem 3.3]{a:Praeger-45-12-3} proved that, up to isomorphism, there is only one symmetric $(45, 12, 3)$ design with flag-transitive and point-primitive  automorphism group $\PSU_{4}(2):2$. 

Note for the case where $G\leq \AGaL_{1}(q)$, we know several examples with $\lambda$ prime, however, a feasible classification in this case, seems to be out of reach. An exceptional example is the unique symmetric $(37,9,2)$ design admitting flag-transitive automorphism group $37 :9<\AGL_{1}(37)$ with point-stabilizer $9$. Indeed, there are four symmetric designs on $37$ points \cite{a:Salwach-Mezzaroba-1978}, only one of which has a flag-transitive automorphism group of rank $5$ that we mentioned.
Excluding above examples, any symmetric design with $\lambda$ a prime not dividing $k$ admitting $G\leq \AGaL_{1}(q)$ as a flag-transitive automorphism group has parameter set $(p^{d},(p^{d}-1)/i,(p^{d}-1-i)/i^{2})$ for some divisor $i$ of $p^{d}-1$, a base block is $\left \langle \omega^{i} \right\rangle$, where $\omega$ is a primitive element of $\Fbb_{p^{d}}$, and if $i=2$, then the designs are of Paley type, see  \cite[Theorem 2]{a:Biliotti-CP-sym-affine}.

\subsection{Definitions and notation}\label{sec:defn}
All groups and incidence structures in this paper are finite.  We denote by $\Fbb_{q}$ the Galois field of size $q$. Symmetric and alternating groups on $n$ letters are denoted by $\S_{n}$ and $\A_{n}$, respectively. We write ``$\Zbb_n$'' for the cyclic group of order $n$. A finite simple group is (isomorphic to) a cyclic group of prime order, an alternating group $\A_{n}$ for $n\geq 5$, a simple group of Lie type or a sporadic simple group, see \cite{b:Atlas} or \cite[Tables~5.1A-C]{b:KL-90}.
We use that standard notation as in \cite{b:Atlas} for finite simple groups, and we consider the finite simple groups of Lie type as listed in Table~\ref{tbl:simplegroups} by excluding all isomorphic groups. We note  that $\PSL_{2}(2)\cong \A_{3}$ and $\PSL_{2}(3)\cong \A_{4}$ are not simple. All isomorphisms amongst the classical groups and between the classical groups and the alternating groups are given in \cite[Proposition~2.9.1]{b:KL-90}. Since
$\PSL_{2}(4)\cong \PSL_{2}(5)\cong \A_{5}$, $\PSL_{2}(9)\cong \PSp_{4}(2)'\cong \A_{6}$ and $\PSL_{4}(2)\cong \A_{8}$,
we view all these simple groups of Lie type as their corresponding alternating groups, and since $\PSL_{2}(7)\cong \PSL_{3}(2)$, we  exclude $\PSL_{2}(7)$. For the finite simple exceptional groups, we know that $^{2}\B_{2}(2)\cong 5:4$ is not simple, and $\G_{2}(2)'\cong \PSU_{3}(3)$ and
$^{2}\G_{2}(3)'\cong \PSL_{2}(8)$, and so both groups are viewed as classical groups. Further, ${}^2\F_4(2)'$ is the Tits simple group and we treat this group as a sporadic simple group. Note in passing that, we sometimes use the Lie notation to denote the corresponding finite simple groups of Lie type.
For example, we may write $\A_{l}(q)$ and $\A_{l}^{-}(q)$ in place of $\PSL_{l+1}(q)$ and $\PSU_{l+1}(q)$, respectively, $\D_l^{-}(q)$ instead of  $\POm_{2l}^{-}(q)$. We also write $\E_6^{-}(q)$ for ${}^2\E_6(q)$.

\begin{table}
	\small
	\caption{Finite simple groups of Lie type.}\label{tbl:simplegroups}
	\begin{tabular}{lllllll}
		\hline\noalign{\smallskip}
		Type &
		Simple group &
		Lie rank &
		Condition \\
		\noalign{\smallskip}\hline\noalign{\smallskip}
		$\A_{l}$ &
		$\PSL_{l+1}(q)$ &
		$l$ &
		$l\geq 2$, $(l,q)\neq (1,2), (1,3), (1,4), (1,5), (1,7), (1,9), (3,2)$ \\
		$\A_{l}^{-}$ &
		$\PSU_{l+1}(q)$ &
		$[\frac{l+1}{2}]$ &
		$l\geq 3$, $(l,q)\neq (2,2)$ \\
		$\C_{l}$ &
		$\PSp_{2l}(q)$ &
		$l$ &
		$l\geq 2$, $(l,q)\neq (2,2), (2,3)$ \\
		$\B_{l}$ &
		$\POm_{2l+1}(q)$ &
		$l$ &
		$l\geq 3$, $q$ odd \\
		$\D_{l}^{\e}$ &
		$\POm_{2l}^{\e}(q)$ &
		$l$ &
		$l\geq 4$, $\e=\pm$ \\
		$\G_{2}$ &
		$\G_{2}(q)$ &
		$2$ &
		$q\geq 3$\\
		$\F_{4}$ &
		$\F_{4}(q)$ &
		$4$ &
		$q\geq 2$\\
		$\E_{6}$ &
		$\E_{6}(q)$ &
		$6$ &
		$q\geq 2$ &\\
		$\E_{6}^{-}$ &
		$\E_{6}^{-}(q)$ &
		$4$ &
		$q\geq 2$
		\\
		$\E_{7}$ &
		$\E_{7}(q)$ &
		$7$ &
		$q\geq 2$ &\\
		$\E_{8}$ &
		$\E_{8}(q)$ &
		$8$ &
		$q\geq 2$ &\\
		${}^{2}\!\B_{2}$ &
		${}^{2}\!\B_{2}(q)$ &
		$1$ &
		$q=2^{2m+1}\geq 8$\\
		${}^{2}\!\G_{2}$ &
		${}^{2}\!\G_{2}(q)$ &
		$1$ &
		$q=3^{2m+1}\geq 9$\\
		${}^{2}\!\F_{4}$ &
		${}^{2}\!\F_{4}(q)$ &
		$2$ &
		$q=2^{2m+1}\geq 8$\\
		${}^3\!\D_{4}$ &
		${}^3\!\D_{4}(q)$ &
		$2$ &
		$q\geq 2$
		\\
		\noalign{\smallskip}\hline\noalign{\smallskip}&
	\end{tabular}
\end{table}

A \emph{symmetric $(v,k,\lambda)$ design} $\Dmc$ is a pair $(\Pmc,\Bmc)$ with a set $\Pmc$ of $v$ points and a set $\Bmc$ of $v$ blocks such that each block is a $k$-subset of $\Pmc$ and every pair of points is incident in $\lambda$ blocks. An \emph{automorphism} of $\Dmc$ is a permutation on $\Pmc$ which maps blocks to blocks and preserves the incidence. The \emph{full automorphism} group $\Aut(\Dmc)$ of $\Dmc$ is the group consisting of all automorphisms of $\Dmc$. A \emph{flag} of $\Dmc$ is a point-block pair $(\alpha, B)$ such that $\alpha \in B$. For $G\leq \Aut(\Dmc)$, $G$ is called \emph{flag-transitive} if $G$ acts transitively on the set of flags. The group $G$ is said to be \emph{point-primitive} if $G$ acts primitively on $\Pmc$.
A \emph{$(v, k, \lambda)$-difference} set $D$ is a $k$-subset of an additively written group $T$ of order $v$ such that the list of differences $\Delta D=\{\alpha-\beta \mid \alpha,\beta\in D, \alpha\neq \beta\}$ contains each nonzero element in $T$ precisely $\lambda$ times. If $T$ is elementary abelian, $D$ is called \emph{elementary abelian difference set}. If $D$ is a $(v, k, \lambda)$-difference set, the \emph{development} of $D$ defined by $\dev(D)=(T,\Bmc_{D})$, where $\Bmc_{D}=\{D+x\mid x\in T\}$, is a symmetric $(v,k,\lambda)$ design with $T$ as its point-regular automorphism group \cite[Theorem VI.1.6]{b:Beth-I}. Conversely, every symmetric $(v,k,\lambda)$ design admitting a point-regular automorphism group arises in this way. The automorphisms of $T$, which are also automorphisms of $\dev(D)$, are called \emph{multipliers}. The multipliers of the form $\alpha \mapsto m\alpha$, with $m$ integer, are called \emph{numerical}.  For a given positive integer $n$ and a prime divisor $p$ of $n$, we denote the $p$-part of $n$ by $n_{p}$, that is to say, $n_{p}=p^{t}$ with $p^{t}\mid n$ but $p^{t+1}\nmid n$.
Further notation and definitions in both design theory and group theory are standard and can be found, for example in \cite{b:Beth-I,b:Dixon,b:Lander}.  We use \textsf{GAP} \cite{GAP4} for computational arguments.

\section{Preliminaries}\label{sec:pre}
In this section, we state some useful facts in both design theory and group theory. If a group $G$ acts on a set $\Pmc$ and $\alpha\in \Pmc$, the \emph{subdegrees} of $G$ are the length of orbits of the action of the point-stabilizer $G_\alpha$ on $\Pmc$.

\begin{lemma}\label{lem:six} {\rm \cite[Lemma 2.1]{a:ABD-PSL2}}
	Let $\Dmc$ be a symmetric $(v,k,\lambda)$ design admitting a flag-transitive automorphism $G$ with a point-stabilizer $H$. Then
	\begin{enumerate}[\rm \quad (a)]
		\item $k(k-1)=\lambda(v-1)$;
		\item $4 \lambda(v-1)+1$ is a square;
		\item $k$ divides $|H|$, and $\lambda v<k^2$;
		\item $k\mid \lambda d$, for all nontrivial subdegrees $d$ of $G$.
	\end{enumerate}
\end{lemma}

\begin{lemma}\label{lem:dio}{\rm \cite[Section 2]{a:Skinner-dioph}}
	Let $p$ be an odd prime. Then the solutions of the Diophantine equation $x^2=4p^n-4p+1$ are
	\begin{enumerate}[\rm \quad (a)]
		\item $p=3$ and $(n,x)=(1,1), (2,5),(5,31)$;
		\item $p=5$ and $(n,x)=(1,1), (2,9), (7,559)$;
		\item $p\geq 7$ and $(n,x)=(1,1)$, $(2,2p-1)$.
	\end{enumerate}
\end{lemma}

\begin{lemma}\label{lem:diff}{\rm \cite[Theorems~VI.14.39(a) and VI.4.16(b)]{b:Beth-I}}
	Let $D$ be an abelian $(v,k,\lambda)$-difference set. If $D$ admits the multiplier  $-1$, then $v$ and $\lambda$ are even, and $k-\lambda$ is a square.
\end{lemma}

\section{Proof of the main result}\label{sec:proof}

Suppose that $\Dmc$ is a nontrivial symmetric $(v,k,\lambda)$ design with $\lambda$ prime admitting a flag-transitive and point-primitive automorphism group $G$ of affine type, that is to say, the socle $T$ of $G$ is an elementary abelian $p$-group of order $p^d$ with $d \geq 1$.
The points of $\Dmc$ can be identified with the vectors of the $d$-dimensional vector space $V:=V_d(p)$ over the finite field $\mathbb{F}_p$ of size $p$, and
$G=TH \leq \AGL_d(p)=\AGL(V)$, where $T \cong (\mathbb{Z}_p)^d$ is the translation group, and $H:=G_{0}$ (the point-stabilizer of the point $0$) is an irreducible subgroup of $\GL_d( p)$. For each divisor $n$ of $d$, the group $\GaL_n(p^{d/n})$ has a natural irreducible action on $V$. Choose $n$ to be minimal such that $H \leq \GaL_n(p^{d/n})$ in this action, and write $q=p^{d/n}$. Thus $H \leq \GaL_n(q )$ and $v=p^d=q^n$. In this section, we aim to prove Theorem~\ref{thm:main}. Recall from the introduction that  the symmetric $(v,k,\lambda)$ designs with $\lambda = 2,3$ or $\gcd(k,\lambda)=1$ have been studied in  \cite{a:Zhou-lam3-affine,a:Regueiro-reduction,a:Biliotti-CP-sym-affine}. We start by considering the case where $\lambda=p$, and in Proposition~\ref{prop:lam-p}, we show that this case leads to no possible parameter sets unless $\lambda=p=2$ in which case $v=16$. Therefore, in the rest of paper, we only need to deal with the case where $p$ is odd and $\lambda$ divides $k$ with $p\neq \lambda\geq 5$.

\begin{proposition}\label{prop:lam-p}
	If $\Dmc$ is a nontrivial symmetric $(p^{d},k,\lambda)$ design with $\lambda$ prime admitting a flag-transitive affine automorphism group $G\leq \AGL_d(p)$, then one of the following holds:
	\begin{enumerate}[\rm \quad (a)]
		\item $\lambda=p=2$ and $\Dmc$ has parameters $(16,6,2)$.
		\item $\lambda \neq p$, and the following hold:
		\begin{enumerate}[\rm (1)]
			\item  $\lambda$ and $1-4\lambda$ are squares in $\Fbb_{p}$;
			\item $k-\lambda$ is not a square;
			\item $G_{0}=G_{B}$ for some block $B$ of $\Dmc$ not containing $0$.
		\end{enumerate}
	\end{enumerate}
\end{proposition}
\begin{proof}
	We may assume by \cite[Theorem 1]{a:Biliotti-CP-sym-affine} that $\lambda$ is a prime divisor of $k$. Let first $\lambda=p$. Then we show that $\Dmc$ has parameters $(16,6,2)$.
	If $p=2$, then it follows from  \cite{a:Regueiro-reduction} that $\Dmc$ is the symmetric design with parameters $(16,6,2)$. If $p$ is an odd prime, then since $v=p^{d}$ and $\Dmc$ is a nontrivial design, we have that $d\geq 2$. Moreover, by Lemma \ref{lem:six}(b) we know  that $4\lambda(v-1)+1=4p^{d+1}-4p+1$ is a square. We now apply Lemma \ref{lem:dio} and conclude that $(p,d)=(3,4)$ or $(5,6)$, equivalently, $\Dmc$ is a symmetric design with parameter set $(81,16,3)$ or $(15625,280,5)$, respectively. By \cite[Lemma 2.4]{a:Zhou-lam3-affine}, we have no symmetric design with parameter set $(81,16,3)$. In the latter case, $n=1,2,3$ or $6$, and by Lemma~\ref{lem:diff}(a), we have that $-1 \notin G_{0} \cap Z$, where $Z$ denotes the center of $\GL_{n}(5^{6/n})$. Clearly, $n \neq 1$. If $n = 2$, then $|G_{0} \cap \SL_{2}(5^{3})|$ is even as $|\GaL_{2}(5^{3}):\SL_{2}(5^{3})|=2^{2}\cdot 3 \cdot 31$, and since the Sylow $2$-subgroups of $\SL_{2}(5^{3})$ are quaternions of order $8$ with $-1$ as their unique involution, $-1 \in G_{0} \cap Z$ , which is a contradiction. If $n=3$, then $|G_{0} \cap Z|=1$ or $3$ by Lemma~\ref{lem:diff} since $-1 \notin G_{0}$. Therefore, $\SU_{3}(5) \unlhd G_{0}$ by \cite[Tables 8.3--8.4]{b:BHR-Max-Low}, and hence $G_{0}=G_{B}$ for some block $B$ by \cite[Theorem 2.14(i)]{a:Kantor-rank3}. Further, $0 \notin B$ since $G_{B}$ acts transitively on $B$. Thus $B$ is a $G_{0}$-orbit of length $280$. However, this is impossible since the $G_{0}$-orbits on $V^{\ast}$ have length $3024$ or a multiple of $3150$ since $\SU_{3}(5) \unlhd G_{0}$. Thus, $n=6$ and $G_{0}\leq \GL_{6}(5)$ with $G_{0} \cap Z=1$ and $280 \mid | G_{0}|$. However, $\GL_{6}(5)$ does not have such subgroups by \cite[Theorem 3.1]{a:BambergPenttila-2008} or \cite[Tables 8.24-8.25]{b:BHR-Max-Low}. Therefore, part (a) holds.
	
	Let now $\lambda \neq p$. Then by Euler's Theorem, $\lambda^{\frac{p-1}{2}}\equiv \left(\lambda/p\right) \pmod{p}$, where $\left(\lambda/p\right)$ denotes the Legendre Symbol. If $\left(\lambda/p\right)=-1$, then $\lambda$ does not divide $k-\lambda$ by \cite[Theorem 4.5]{b:Lander}, but this is not the case as $\lambda$ is assumed to be a divisor of $k$. Therefore, $\left(\lambda/p\right)=1$, and hence $\lambda$ is a square in $\Fbb_{p}$. Further, $1-4\lambda$ is a square in $\Fbb_{p}$ by Lemma \ref{lem:six}(b). If $k-\lambda$ is a square, then $\lambda^{2}$ divides $k-\lambda$, and so \cite[Lemma VI.6.1]{b:Beth-I} implies that $\lambda^{2}$ divides $p^{d}-k-(k-\lambda)$. Thus $\lambda$ divides $p^{d}-k$, and since $\lambda$ is a divisor of $k$, it follows that $\lambda$ divides $p^{d}$, whereas $\lambda \neq p$. Finally, $k=(p^{d}-1)\lambda/(k-1)$ is prime to $p$ since $\lambda \neq p$, and hence $G_{0}=G_{B}$ for some block $B$ of $\Dmc$ by \cite[Lemma VI.2.5]{b:Beth-I}. Since $G_{B}$ acts transitively on $B$, the block $B$ does not contain $0$.
\end{proof}

\begin{corollary}\label{cor:irr}
	Let $\Dmc$ be a nontrivial symmetric $(v, k, \lambda)$ design with $\lambda$ prime and that $G$ be a flag-transitive and point-primitive group of automorphisms of $\Dmc$ of affine type. Let also $G=TG_{0} \leq \AGL (V)$, where $T$ is the translation group. Then $G_{0}$ is an irreducible subgroup of $\GL(V)$.
\end{corollary}

\begin{proof}
	Suppose that $G_{0}$ preserves a proper subspace $W$ of $V$, then $G_{0}$ is a proper subgroup of $X_W G_{0}$,  that is to say $G_{0}<T_WG_{0}<G$.
	This yields $G_{0}$ is not maximal in $G$, or equivalently $G$ is not point-primitive, which is a contradiction to our assumption.
\end{proof}

\begin{table}
	\centering
	\small
	\caption{Size of the subsets $\mathcal{S}$ and $\mathcal{N}$ in Lemma \ref{lem:C8}.}\label{tbl:SingNonSing}
	\resizebox{\textwidth}{!}{
		\begin{tabular}{llll}
			\noalign{\smallskip}\hline\noalign{\smallskip}
			$X$ & $|\Smc|$ & $|\Nmc|$ & $\gcd (|\Smc|,|\Nmc)|$ \\
			\noalign{\smallskip}\hline\noalign{\smallskip}
			$\SU_{n}(q^{1/2})$, $n\geq 3$ & $(q^{n/2}-(-1)^{n})(q^{n/2-1}-(-1)^{n-1})$ & $%
			q^{n/2-1}(q^{n/2}-(-1)^{n})(q^{1/2}-1)$ & $(q^{n/2}+1)(q^{1/2}-1)$ \\
			$\Omega _{n}^{\circ }(q)$, $n\geq 3$ & $q^{n-1}-1$ & $q^{n-1}(q-1)$ & $q-1$
			\\
			$\Omega _{n}^{+}(q)$, $n\geq 4$ & $(q^{n/2}-1)(q^{n/2-1}+1)$ & $%
			q^{n/2-1}(q^{n/2}-1)(q-1)$ & $2(q^{n/2}-1)$ \\
			$\Omega _{n}^{-}(q)$, $n\geq 4$ & $(q^{n/2}+1)(q^{n/2-1}-1)$ & $%
			q^{n/2-1}(q^{n/2}+1)(q-1)$ & $(q^{n/2}+1)(q-1)$\\
			\noalign{\smallskip}\hline\noalign{\smallskip}
		\end{tabular}
	}
\end{table}
	\begin{lemma}\label{lem:C8}
		Let $X$ be one of the classical groups $\SL_n(q)$, $\Sp_n(q)$,  $\SU_n(q^{1/2})$ and $\Omega^{\e}_n(q)$ with $\e \in \{\pm, \circ\}$, and let $V=V_n(q)$ be the underlying vector space of minimal dimension $n$ such that $G_0 \leq \N_{\GaL_n(q)}(X)$. Then $G_0$ does not contain $X$.
	\end{lemma}
	\begin{proof}
		Assume to the contrary that $G_{0}$ contains $X$. If $X$ is one of the groups $\SL_n(q)$ and $\Sp_n(q)$, then it follows from \cite[Theorem and Appendix 1]{a:Liebeck-HA-rank3} that $G$ acts $2$-transitively on $V$. Note by our assumption that $p$ is odd and $G_{0} \not \leq\GaL_1(p^d )$. Then the groups $\SL_n(q)$ and $\Sp_n(q)$ are ruled out by \cite{a:Kantor-85-2-trans}, as there is no symmetric design with $\lambda \geq 5$ prime. Hence, $X$ is one of the groups $\SU_n(q^{1/2})$ and $\Omega^{\e}_n(q)$ with $\e \in \{\pm, \circ\}$. \smallskip

		\noindent Let $B$ be a block of $\Dmc$ preserved by $G_{0}$ by Proposition \ref{prop:lam-p}(b3). Hence, $B$ is a non-trivial $G_{0}$-orbit since $G$ acts flag-transitively on $\Dmc$.
		Assume that $n>2$. Let $\mathcal{S}$ and $\mathcal{N}$ be the set of non-zero singular points and the set of non-singular points for the $X$-invariant hermitian form in the unitary
		case or the $X$-invariant quadratic form in the orthogonal one, respectively. Then the size of $\mathcal{S}$ and $\mathcal{N}$ recorded in Table \ref{tbl:SingNonSing} can be determined by \cite[Lemma 10.4 and Theorem 11.5]{b:Taylor92}. Then both $\mathcal{S}$ and $\mathcal{N}$ are a union of $G_{0}$-orbits, and so Lemma~\ref{lem:six}(d) implies that $k/\lambda$ divides $\gcd(|\mathcal{S}|,|\mathcal{N}|)$, where $k$ is the length of a suitable $G_{0}$-orbit.
		If $X=\SU_{n}(q^{1/2})$, then $n$ is odd by Lemma~\ref{lem:diff}. Since the elements of $X$ are isometries, the length of $X$-orbits on $V\setminus \{0\}$ is either $(q^{n/2}+1)(q^{n/2-1}-1)$, or $q^{n/2-1}(q^{n/2}+1)(q^{1/2}-1)/(q-1)$, and since $X\unlhd G_{0}$, the $G_{0}$-orbits on $V\setminus \{0\}$ have length $(q^{n/2}+1)(q^{n/2-1}-1)$ or $c q^{n/2-1}(q^{n/2}+1)(q^{1/2}-1)/(q-1)$ for some divisor $c$ of $q-1$. Note that $\lambda \neq p$ by Proposition 3.1. Then the facts that $k/\lambda$ divides $\gcd (|\mathcal{S}|,|\mathcal{N}|)=(q^{n/2}+1)(q^{1/2}-1)$ and that $k$ is the length of a $G_{0}$-orbit force $k=|\mathcal{S}|=(q^{n/2}+1)(q^{n/2-1}-1)$. Thus $(q^{n/2}+1)(q^{n/2-1}-1)/\lambda$ divides $(q^{n/2}+1)(q^{1/2}-1)$, and so $\lambda =(q^{n/2-1}-1)/(q^{1/2}-1)$. Since $n-1\geq 2$ is even and $q$ is odd, we conclude that $\lambda=2$, which is not the case. Therefore, $G_{0}$ does not contain  $\SU_{n}(q^{1/2})$.
		A similar argument can be applied to the case where $X=\Omega _{n}^{\e }(q)$ with $\e\in\{\pm,\circ\}$. In this case, we obtain $k=|\mathcal{S}|$, where $|\Smc|$ is given in Table \ref{tbl:SingNonSing}, and so
		$\lambda $ is $(q^{n/2-1}+1)/2$, $(q^{n/2-1}-1)/(q-1)$ or $(q^{n-1}-1)/(q-1)$ when $\e $ is $+$, $-$ or $\circ $, respectively. If $\e=\circ$, then since $n-1\geq 2$ is even and $q$ is odd, we obtain $\lambda=2$, which is not the case. If $\e= +$ or $-$, then $k/\lambda$ is $2(q^{n/2}-1)$ or $(q^{n/2}+1)(q-1)$, respectively, but both possibilities violate Lemma \ref{lem:six}(c).
		Therefore, $n=2$. The case where $X= \SU_{2}(q^{1/2})$ is ruled out by Lemma~\ref{lem:diff}.
		Hence, $X= \Omega_{2}^{\e }(q)\cong \mathbb{Z}_{(q-\e 1)/2}$ with $\e =\pm$. If $\e =-$, then $G_{0}\leq \N_{\GaL_{2}(q)}(X)\leq \GaL_{1}(q^{2})$, which violates our assumptions. Thus $\e =+$, and hence $X\cong \mathbb{Z}_{(q-1)/2}$.
		Let $\omega$ be a primitive element of $\mathbb{F}_{q}^{\ast}$. Then, we may assume that $X=\langle x^{2}\rangle $ and $X\unlhd
		G_{0}\leq \langle x,y,\phi \rangle =\GaO_{2}^{+}(q)$, where $\phi:(t_{1},t_{2})\mapsto (t_{1}^{p},t_{2}^{p})$  and
		\[
		x=\left(
		\begin{array}{cc}
			\omega  & 0 \\
			0 & \omega ^{-1}%
		\end{array}%
		\right) \text{ and }
		y=\left(
		\begin{array}{cc}
			0 & 1 \\
			1 & 0%
		\end{array}%
		\right),
		\]%
		see \cite[Theorem 11.4]{{b:Taylor92}}.
		Then each non-trivial $G_{0}$-orbit is of length divisible by $(q-1)/2$, and since $k$ is the length of a $G_{0}$-orbit, it follows that that $k$ is a multiple of $(q-1)/2$. Further, the $G_{0}$-orbit containing $(1,1)$ is of length $(q-1)/2$ or $q-1$ according as $x$ does not lie or does lie in $G_{0}$. Then $k/\lambda$ divides $q-1$ by Lemma \ref{lem:six}(d), and so $k=
		(q-1)\lambda/t$ for some $t\geq 1$. Since $(q-1)/2$ divides $k$, it implies that $(q-1)/2$ divides  $(q-1)\lambda/t$, and the fact that $\lambda$ is prime forces $t=2$, $\lambda$ or $2\lambda$. In this case, $v=q^2$. Then by Lemma \ref{lem:six}(c), we have that $t<\lambda$, an so $t=2$. By Lemma \ref{lem:six}(a), we have that $\lambda=(4q+6)/(q-1)$, and since $\lambda$ is prime, we obtain $q=11$ for which $(v,k,\lambda)=(121,25,5)$, which is ruled out in \cite{a:Braic-255}.
	\end{proof}

\begin{proposition}\label{pro:Aschbacher}
	Suppose that $\Dmc$ is a nontrivial symmetric $(q^{n},k,\lambda)$ design and that $G$ is a flag-transitive and point-primitive automorphism group of $\Dmc$ of affine type. Then one of the following holds:
	\begin{enumerate}[\rm \quad (a)]
		\item $G_{0}$ lies in a member of one of the geometric families $\Cmc_{i}$, $i=2,4,5,6,7$, of subgroups of $\N_{\GaL_n(q)}(T)$; or
		\item $G_{0}^{(\infty)}$, the last term in the derived series of $G_{0}$, is quasisimple, and its action on $V$ is absolutely irreducible and not realizable over any proper subfield of $\Fbb_q$.
	\end{enumerate}
\end{proposition}

\begin{proof}
	Note by Lemma~\ref{lem:C8} that $G_{0}$ does contain one of the classical groups $\SL_n(q)$, $\Sp_n(q)$,  $\SU_n(q^{1/2})$ and $\Omega^{\e}_n(q)$ with $\e \in \{\pm, \circ\}$. Since $G$ is flag-transitive and point-primitive,
	by Aschbacher's Theorem~\cite{a:Aschbacher-84}, there are families $\Cmc_i$, $1\leq i \leq 7$, of subgroups of $\N_{\GaL_n(q)}(X)$, such that either $G_{0}$ is contained in a member of $\Cmc_i$, $1\leq i \leq 7$, or $G_{0}^{(\infty)}$ is quasisimple and irreducible. If $G_{0}$ is a $\Cmc_1$-subgroup, then $G_{0}$ is reducible, which is not the case by Corollary~\ref{cor:irr}. Moreover, by the definition of $q$, $G_{0}$ cannot be a member of $\Cmc_3$. Thus the assertion follows.
\end{proof}

In the following sections, we analyze each possible case in Proposition~\ref{pro:Aschbacher}.
In some cases, we use the method presented in \cite{a:Liebeck-98-Affine}, and we also use the method described in \cite[Section 6.1]{a:ABD-Exp} which we briefly mention here.
We first note by Lemma~\ref{lem:six}(c) that $k$ divides $|H|$, where $H:=G_{0}$ is the  point-stabilizer of $0$ in $G$. We know that $v=p^d=q^n$. Since $\lambda \neq p$ is an odd prime and $k$ divides $\lambda(v-1)=\lambda(p^d-1)$, it follows that $k$ divides $\lambda(v-1, |H|_{p'})$. Since also $|H|$ is mainly a polynomial in terms of $q$, we conclude that $k$ divides $\lambda f(q)$, where $f(q)$ is a polynomial which is multiple of $\gcd(v-1, |H|_{p'})$. In the case where there are some suitable subdegrees of $G$, we also apply Lemma~\ref{lem:six}(e) and obtain $f(q)$ by the greatest common divisor of $v-1$ and these subdegrees. Therefore,
\begin{align}\label{eq:k-f}
	u\cdot k=\lambda f(q),
\end{align}
for some positive integer $u$. Since $\lambda v<k^{2}$, it follows that
\begin{align}\label{eq:u}
	u^2\cdot v <\lambda \cdot f(q)^2.
\end{align}
Again, by Lemma~\ref{lem:six}(a) and the fact that $u\cdot k=\lambda f(q)$, we find the parameters $k$ and $\lambda$ in terms of $u$ and $q$ as below:
\begin{align}\label{eq:k-lam}
	k = \frac{u\cdot (v-1)}{f(q)}+1 \quad \text{ and } \quad  \lambda=\frac{u^{2}\cdot (v-1)+u\cdot f(q)}{f(q)^{2}}.
\end{align}

\subsection{Geometric subgroups}

In this section,  we prove in Lemma~\ref{lem:affine-geom} below that the point-stabilizer $H=G_{0}$ cannot lie in one of the families of geometric $\Cmc_{i}$-subgroups of $G$ recorded as in Proposition~\ref{pro:Aschbacher}(a).

\begin{lemma}\label{lem:affine-geom}
	Let $\Dmc$ be a nontrivial symmetric $(p^{d},k,\lambda)$ design with $p\neq\lambda \geq 5$ a prime divisor of $k$. If $G$ is a flag-transitive and point-primitive automorphism group of $\Dmc$ of affine type, then the point-stabilizer $H$ cannot be a geometric subgroup of $\GL_{d}(p)$.
\end{lemma}

\begin{proof}
	By Proposition~\ref{pro:Aschbacher}, we only need to consider the cases where $H$ lies in one of the families $\Cmc_{i}$ with $i=2,4,5,6,7$. We now analyze each of these possible cases separately:\smallskip
	
	\noindent \textbf{(1)}
	Suppose first that $H$ is contained in a member of $\Cmc_{2}$. Then $H$ preserves a partition $V=V_{1} \bigoplus \ldots \bigoplus V_{t}$, where each $V_{j}$ of the same dimension $i$, and $n=it$ with $t\geq2$. By considering $V$ as a vector space over the field $\mathbb{F}_{p}$ of dimension $d$, we observe that $H$ is a subgroup of $N:=\GL_{d/t}(p) \wr \S_{t}$. By Proposition~\ref{prop:lam-p}, we can assume that $\lambda\neq p$. Then by the fact that $\lambda$ is an odd prime divisor of $k$ implies that $\lambda$ divides $t!$ or $p^{j}-1$ for some $1\leq j\leq d/t$. Set $\ell:=d/t$. Then $\lambda \leq \max \{t,p,p^{j}-1\}$, for some $1\leq j\leq \ell$, and so
	\begin{equation}\label{eq:C2-lam}
		\lambda <t\cdot(p^{\ell}-1).
	\end{equation}
	Note also that each orbit of $N$ on $V \backslash \{0\}$ is a union of orbits of $H$. Then $\bigcup_{i=1}^{t}V_{i} \backslash \{0\}$ is an $N$-orbit of length $t\cdot (p^{\ell}-1)$, and so Lemma~\ref{lem:six}(d) yields
	\begin{equation}\label{eq:C2-k}
		k \quad \text{divides} \quad \lambda t\cdot(p^{\ell}-1).
	\end{equation}
	Since $v=p^{\ell t}$, it follows from \eqref{eq:C2-k} and Lemma~\ref{lem:six}(c) that $p^{\ell t}<\lambda t^{2}\cdot (p^{\ell}-1)^{2}$. Thus
	$p^{\ell t}<\lambda t^{2}p^{2\ell}$. Then by \eqref{eq:C2-lam}, we have $p^{\ell \cdot(t-3)}<t^{3}$. Thus $\ell\cdot(t-3)\cdot \log_2p <3\log_2t$. The last inequality holds only when
	\begin{enumerate}[\rm (i)]
		\item $t=2$,
		\item $t=3$,
		\item $t\in\{4,5,6,7,8\}$ and $\ell=1$, or
		\item $(\ell,t,p)\in \{(2,4,3),(2,4,5),(2,4,7),(2,5,3),(3,4,3)\}$.
	\end{enumerate}
	We now discuss these four possible cases. Note that $\lambda=p$, or $\lambda$ divides $t!$ or $p^{j}-1$ for some $1\leq j\leq \ell$.
	
	Suppose first that $\lambda$ divides $t!$. Then since by the assumption $\lambda\geq 5$ is an odd prime, by considering the cases (i)-(iv), we need to focus on the possibilities where $5\leq t\leq 8$ and $\ell=1$, or $(\ell,t,q)=(2,5,3)$. In the latter case, we have that $\lambda=5$, and since $v=p^{\ell t}=3^{10}$, we observe that $4\lambda(v-1)+1=1180961$ which is not a square and this violates Lemma~\ref{lem:six}(b). In the former case, we have $\lambda=5,7$. Then by Lemma~\ref{lem:six}(c), we conclude that $p^{2t}<\lambda t^2 (p^{2}-1)^{2}$. Hence $p=3$ and $(t,\lambda)\in \{(5,5),(5,7),(6,7)\}$. We now apply Lemma~\ref{lem:six}, and conclude that $k$ divides $\lambda \gcd(t(q-1),q^t-1)$. Then $k$ is a divisor of $20$ or $28$. But for each of these values of $k$, the condition $k(k-1)=\lambda(v-1)$ does not hold, which is a contradiction.
	
	Therefore, $\lambda$ divides $p^{j}-1$ for some $1\leq j\leq \ell$. In what follows, we analyze each of possible cases (i)-(iv) under this condition.\smallskip
	
	\noindent \textbf{(i)} Let $t=2$. Then by \eqref{eq:C2-k}, we know that $k$ divides $2\lambda (p^{\ell}-1)$. Set $f_{\ell}(p):=p^{\ell}-1$, and let $u$ be a positive integer such that $uk=2\lambda f_{\ell}(p)$. Since $v-1=p^{2\ell}-1$, by \eqref{eq:k-lam}, we have that
	\begin{align}\label{eq:C2-i2-k-lam}
		2k=2+u\cdot (p^{\ell}+1) \quad \text{ and }  \quad 4\lambda=u^2+\frac{2u^2+2u}{f_{\ell}(p)}.
	\end{align}
	Then
	\begin{align}\label{eq:C2-lam-t2}
		u< 2\sqrt{\lambda}.
	\end{align}
	Let $m$ be a positive integer such that $m\lambda=p^j-1$. Since $\lambda$ is an odd prime, we conclude that $m$ is even. Then by \eqref{eq:C2-i2-k-lam}, we conclude that
	\begin{align*}
		u^2+\frac{2u^2+2u}{f_{\ell}(p)}=\frac{4(p^j-1)}{m}.
	\end{align*}
	Therefore, $4(p^j-1)-mu^2>0$, and so
	\begin{align}\label{eq:C2-u}
		u< 2\sqrt{\frac{p^j-1}{m}}.
	\end{align}
	By \eqref{eq:C2-i2-k-lam} and the fact that $k$ divides $2\lambda f_{\ell}(p)$, we conclude that
	\begin{align}\label{eq:C2-i2-Divide}
		u\cdot (p^{\ell}+1)+2  \text{ divides } \frac{4 (p^{j}-1)(p^{\ell}-1)}{m}.
	\end{align}
	
	If $j=\ell$, then since $\gcd(u\cdot (p^{\ell}+1)+2, p^{\ell}-1)$ divides $2u+2$, we conclude by \eqref{eq:C2-i2-Divide} that $u\cdot (p^{\ell}+1)+2$ must divide $16m^{-1}\cdot(u+1)^2$, and so $m\cdot p^{\ell}<64u$. Then by \eqref{eq:C2-u}, we conclude that $m^3\cdot p^{2\ell}<2^{14}\cdot (p^{\ell}-1)$. Then since $m$ is even, we must have $p^{\ell}<2^{11}$. This inequality holds only for $(p,\ell)$ as in below:\smallskip
	
	\begin{table}[h]
		\centering
		\small
		\begin{tabular}{lccccc}
			\hline\noalign{\smallskip}
			$p$ & $3$ & $5$ & $7,11$ & $13, \ldots, 43$ & $47, \ldots, 2039$ \\
			\noalign{\smallskip}\hline\noalign{\smallskip}
			$\ell \leq $ & $6$ & $4$ & $3$ & $2$ & $1$\\
			\noalign{\smallskip}\hline\noalign{\smallskip}
		\end{tabular}
	\end{table}

	\noindent For these values of $(p,\ell)$, we can find the value of $u$ by \eqref{eq:C2-u}. For these values of $(p,\ell, u)$, by \eqref{eq:C2-i2-Divide}, we conclude that $(p,\ell,u)\in\{(3, 2, 3), (5, 1, 1), (7,1,2),(11,1,4),$ $(13,1,1),(13,1,5),(17,1,7),$ $(19,1,8),(29,1,1),(41,1,3)\}$. Since $t=2$, it follows that $v=p^{2\ell}\leq 41^2$, however, by   \cite{a:Braic-2500-power}, we obtain no possible symmetric design.
	
	If $j=\ell-1$ with $p\in\{3,5\}$ and $\ell\geq 2$, then since $\gcd(u\cdot (p^{\ell}+1)+2, p^{\ell}-1)$ divides $2u+2$ and $\gcd(u\cdot (p^{\ell}+1)+2, p^{\ell-1}-1)$ divides $u\cdot (p+1)+2$ by \eqref{eq:C2-i2-Divide}, we conclude that $u\cdot (p^{\ell}+1)+2\leq 4m^{-1}\cdot(2u+2)[u\cdot (p+1)+2]$. Therefore,  $mu{\cdot} p^{\ell}< 16(u+1)(3u+1)$. Hence  $m\cdot p^{\ell}<2^7u$. Combining this with \eqref{eq:C2-u} implies that $m^3\cdot p^{2\ell}<2^{15}\cdot p^{\ell-1}$. As $m$ is even, we conclude that $p^{\ell+1}<2^{12}$ with $p=3,5$, and so $\ell\leq 7$ if $p=3$, and $\ell \leq 4$ if $p=5$. For these values of $(p,\ell)$, we can find the  parameter $u$ by \eqref{eq:C2-u}, and then,  \eqref{eq:C2-i2-Divide} holds only when $(p,\ell,u)=(5, 3, 1)$, which gives no possible parameter by \eqref{eq:C2-i2-k-lam}.\smallskip
	
	\noindent \textbf{(ii)} Let $t=3$. Then by \eqref{eq:C2-k}, the parameter $k$ divides $3\lambda (p^{\ell}-1)$. Set $f_{\ell}(p):=p^{\ell}-1$, and let $u$ be a positive integer such that $uk=3\lambda f_{\ell}(p)$, where $f_{\ell}(p)=p^{\ell}-1$. Since $v-1=p^{3\ell}-1$, by \eqref{eq:k-lam}, we have that
	\begin{align}\label{eq:C2-i3-k-lam}
		k=1+\frac{u\cdot (p^{2\ell}+p^{\ell}+1)}{3} \quad \text{ and }  \quad 9\lambda=u^2\cdot(p^{\ell}+2)+\frac{3u^2+3u}{f_{\ell}(p)}.
	\end{align}
	There is a positive integer $m$ such that $m\lambda=p^j-1$. Since $\lambda$ is an odd prime, we conclude that $m$ is even. Then by \eqref{eq:C2-i3-k-lam}, we conclude that
	\begin{align*}
		u^2\cdot(p^{\ell}+2)+\frac{3u^2+3u}{f_{\ell}(p)}=\frac{9(p^j-1)}{m}.
	\end{align*}
	Therefore, $9(p^j-1)-mu^2\cdot(p^{\ell}+2)>0$, and so $u^2<9(p^{\ell}-1)/[m\cdot(p^{\ell}+2)]$. This together with the fact that $m$ is even implies that $u\leq 2$. Combining this with \eqref{eq:C2-i3-k-lam} and the fact that $\lambda$ is an integer, we conclude that $p^{\ell}-1$ must divide $6$ or $18$ when $u=1$ or $2$, respectively. Then $(p,\ell,u)\in\{(3,1,1), (3,1,2), (7,1,1), (7,1,2), (19,1,1)\}$, but each of these triples dose not give rise to a parameter $\lambda$ by \eqref{eq:C2-i3-k-lam}, which is a contradiction.
	\smallskip
	
	\noindent \textbf{(iii)} Let $t\in\{4,5,6,7,8\}$ and $\ell=1$. Then  \eqref{eq:C2-k} implies that $k$ divides $\lambda t\cdot (p-1)$, where $t\in\{4,5,6,7,8\}$. Set $f(p):=p-1$. Then  $uk=\lambda t\cdot f(p)$, for some positive integer $u$. Since $v-1=p^{t}-1$, it follows from \eqref{eq:k-lam} that
	\begin{align}\label{eq:C2-i48-k-lam}
		k=1+\frac{u\cdot \sum_{j=0}^{t-1}p^j}{t} \quad \text{ and }  \quad t^2\lambda=u^2\cdot\sum_{j=0}^{t-2}p^j+\frac{tu^2+tu}{p-1}.
	\end{align}
	Here $\lambda$ is an odd prime divisor of $p-1$. Then $m\lambda=p-1$, for some even number $m$, and so by \eqref{eq:C2-i48-k-lam}, we conclude that
	\begin{align*}
		u^2\cdot\sum_{n=0}^{t-2}p^t+\frac{tu^2+tu}{p-1}=\frac{t^2\cdot(p-1)}{m}.
	\end{align*}
	Therefore, $t^2(p-1)-mu^2\cdot\sum_{n=0}^{t-2}p^t>0$. Since $t\in\{4,5,6,7,8\}$ and $m$ is even implies, it follows that $u=1$. By \eqref{eq:C2-i48-k-lam}, we conclude that $p-1$ divides $2t$, and so $p\leq 17$. For each such $p$, it is easy to check that the parameter $\lambda$ obtained by \eqref{eq:C2-i48-k-lam} is not an integer, which is a contradiction.
	\smallskip
	
	\noindent \textbf{(iv)} In this case,we have that $(\ell,t,p,\lambda)\in\{(2,5,3,5), (3,4,3,13)\}$. For such $v=p^{t\ell}$ and $\lambda$, the equation $k(k-1)=\lambda(v-1)$ has no integer solutions for $k$.
	
	\smallskip
	
	\noindent \textbf{(2)} Suppose now that $H$ lies in a member of $\Cmc_{4}$. Then $H$ is a subgroup of the stabilizer of a tensor product of two nonsingular subspaces of dimensions $1<i<\ell$ with $i\ell=d$ and $i^2<d$. So $V_d(p)=V_{i} \bigotimes V_{\ell}$ and $H \leq N_{\GL_{d}(p)}(\GL_{i}(p)\circ \GL_{\ell}(p))$ in its natural action on $V$, where $V_{i}$ and $V_{\ell}$ are spaces over $\mathbb{F}_{p}$ of dimension $t$ and $m$, respectively.  Here the nonzero vectors $v_1\otimes v_2$ form a union of $H$-orbits which has size $(p^i-1)(p^{\ell}-1)/(p-1)$.  Hence by Lemma~\ref{lem:six}(d), we conclude that
	\begin{equation}\label{eq:C4-k}
		k \mid \frac{\lambda\cdot (p^{i}-1)(p^{\ell}-1)}{p-1}.
	\end{equation}
	Note here that $v=p^{i\ell}$. So Lemma \ref{lem:six}(c) implies that
	\begin{equation}\label{eq:C4-1}
		p^{i\ell}<\lambda\cdot (p^{i}-1)^2(p^{\ell}-1)^2.
	\end{equation}
	Note that $k$ divides $|H|$ and $\lambda$ is an odd prime divisor of $k$. By Proposition~\ref{prop:lam-p}, we exclude the case where $\lambda=p$. Thus $\lambda$ divides $p^{j}-1$, for some $1\leq j \leq \ell$. Therefore, $\lambda\leq p^{\ell}$, and hence it follows from \eqref{eq:C4-1} that $i\ell<3\ell+2i$, that is to say, $\ell\cdot(i-3)<2i$, and since  $i<\ell$, one of the following holds:
	\begin{enumerate}[\rm (i)]
		\item $i=2$,
		\item $i=3$,
		\item $i=4$ and $\ell\in \{5,6,7\}$.
	\end{enumerate}
	In what follows, we analyze each of these three possible cases. Moreover, since by Lemma~\ref{lem:six}, $k$ divides $\lambda(v-1)$,  it follows from \eqref{eq:C4-k} that
	\begin{equation}\label{eq:C4-k-h}
		k \quad \text{divides} \quad \lambda h\cdot f_{\ell}(p),
	\end{equation}
	where $f_{\ell}(p)=p^{\ell}-1$ and $h=\gcd((p^{i}-1)/(p-1),p^{\ell}+1)$ in case (i) or $h=(p^{i}-1)/(p-1)$ in cases (ii) and (iii).
	
	\noindent \textbf{(i)} Let $i=2$. Then $h=\gcd(p+1,p^{\ell}+1)$, and so $h=2$ or $h=p+1$ when $\ell$ is even or odd, respectively.
	Moreover, by \eqref{eq:C4-k-h}, we have that $uk=\lambda h\cdot f_{\ell}(p)$, for some positive integer $u$. Since $v-1=p^{2\ell}-1$, it follows from \eqref{eq:k-lam}  that
	\begin{align}\label{eq:C4-i2-k-lam}
		k=1+\frac{u\cdot (p^{\ell}+1)}{h} \quad \text{ and }  \quad h^2\lambda=u^2+\frac{2u^2+uh}{f_{\ell}(p)}.
	\end{align}
	Therefore, $h^2\cdot \lambda-u^2>0$, and so
	\begin{align}\label{eq:C4-i2-u1}
		u< h\cdot \sqrt{\lambda}.
	\end{align}
	
	We first show that $u>h$. Assume to contrary that $u\leq h$. Since $\lambda$ is an odd prime, by \eqref{eq:C4-i2-k-lam}, we conclude that $f_{\ell}(p)$ divides $2u^2+uh$. Therefore, $p^{\ell}-1\leq 3h^2$. Since $\ell>2$ and $h\in\{2,p+1\}$, the only pairs satisfying this inequality is $(p,\ell)=(3,3)$, and so $v=3^{6}$, but we find no possible parameter set by \cite{a:Braic-2500-power}. Therefore, by \eqref{eq:C4-i2-u1}, we have that
	\begin{align}\label{eq:C4-i2-u}
		h<u< h\cdot \sqrt{\lambda}.
	\end{align}
	
	Suppose first that $j$ is even, where $\lambda$ divides $p^{j}-1$ for some  $1\leq j\leq \ell$. Then there exists an even number $m$ such that $m\lambda=p^{j/2}-\e1$ with $\e=+$ or $-$. As $\lambda$ is an integer, \eqref{eq:C4-i2-k-lam} implies that $f_{\ell}(p)$ divides $2u^2+uh$. Then by \eqref{eq:C4-i2-u}, we have that $f_{\ell}(p)\leq 2u^2+uh<3u^2<3h^{2}\lambda$, and so
	\begin{align}\label{eq:C4-i2-J-even}
		m\cdot (p^{\ell}-1)<3h^2\cdot(p^{j/2}-\e 1),
	\end{align}
	where $m$ and $j$ are even and $2\leq j\leq \ell$. Then $2(p^{\ell}-1)<3(p+1)^{2}(p^{j/2}+1)$, and so $p^{(\ell-4)/2}\leq 12$, which is true when $3\leq \ell\leq 8$. Considering the fact that $h=2$ or $h=p+1$ when $\ell$ is even or odd, respectively, it follows from \eqref{eq:C4-i2-J-even} that $(\ell,j)=(3,2)$, and so $h=p+1$. In which case, by \eqref{eq:C4-i2-k-lam} and the fact that $k$ divides $\lambda h f_{\ell}(p)$, we conclude that $u\cdot (p^{3}+1)+h$ divides  $h^2\cdot (p-\e1)(p^{3}-1)/m$, where $\e\in\{+,-\}$. Since $\gcd(u\cdot (p^{3}+1)+2, p^{3}-1)$ divides $2u+h$, it follows that $u\cdot (p^{3}+1)+h$ divides $h^2\cdot (2u+h)(p-\e1)/m$. Let $u_1$ be a positive integer such that
	\begin{align}\label{eq:C4-i2-J-even-2}
		mu_1\cdot[u\cdot (p^{3}+1)+h]=h^2\cdot (2u+h)(p-\e1)
	\end{align}
	By \eqref{eq:C4-i2-u}, we have that $2u+h<3u$, and so  $mu_1\cdot (p^{3}+1)<3h^3=3(p+1)^{3}$. Then Straightforward computations shows that $(m,u_1)=(2,1)$ or $(m,u_{1},p)\in\{(2,2,3),(2,2,5),(2,2,7),(2,3,3)\}$. In the latter case, as $\lambda=(p-\e 1)/m$, we have that $\lambda\in \{1,2,3,4\}$, which is impossible as we assume that $\lambda\geq 5$. In the former case, by \eqref{eq:C4-i2-J-even-2}, we conclude that $2u\cdot[(p^{3}+1)-h^2\cdot(p-\e1)]=h^3\cdot (p-\e1)-2h>0$, and so $p^{3}+1>h^2\cdot(p-\e1)$, and hence $p^{3}+1> (p+1)^{2}(p-1)$,  which is impossible.\smallskip
	
	Suppose now that $j$ is odd. Recall that $\lambda$ is an odd prime divisor of $p^j-1$ where $1\leq j\leq \ell$. Let $m$ be a positive integer such that $m\lambda=p^j-1$. Since $\lambda$ is an odd prime, we conclude that $m$ is even. Then by \eqref{eq:C4-i2-k-lam} and this fact that $\lambda$ is an integer, we conclude that $f_{\ell}(p)$ divides $2u^2+uh$. By \eqref{eq:C4-i2-u}, we have that $2u^2+uh<3u^2$, and again by \eqref{eq:C4-i2-u}, we conclude that $m\cdot (p^{\ell}-1)<3h^2\cdot(p^j-1)$. It is easy to see that the last inequality holds only when
	\begin{align}\label{eq:C4-i2-J-odd}
		j\in\{\ell, \ell-1, \ell-2\}.
	\end{align}
	We now consider two following cases:\smallskip
	
	\noindent\textbf{(i.1)}  Let $\ell$ be even. As $j$ is odd, we have that $j=\ell-1$.
	By \eqref{eq:C4-i2-k-lam} and the fact that $k$ divides $2\lambda f_{\ell}(p)$, we conclude that $u\cdot (p^{\ell}+1)+2$ divides $4(p^{j}-1)(p^{\ell}-1)/m=4(p^{\ell-1}-1)(p^{\ell}-1)/m$.
	Since $\gcd(u\cdot (p^{\ell}+1)+2, (p^{\ell-1}-1)(p^{\ell}-1))$ divides $(2u+2)[u\cdot(p+1)+2]$, it follows that $u\cdot (p^{\ell}+1)+2$ divides $4 (2u+2)[u\cdot(p+1)+2]/m$, then $m[u(p^{\ell}+1)+2]\leq 4(2u+2)[u\cdot(p+1)+2]<36u^2p$, and so $p^{\ell-1}<18u$. So by \eqref{eq:C4-i2-u} and the fact that $m\geq 2$, we have that $p^{\ell-1}<36\sqrt{(p^{\ell}-1)/m}\leq 36\sqrt{(p^{\ell}-1)/2}$, then $p^{\ell-1}<18\cdot36$. Thus $(\ell,p)\in\{(4,3),(4,5),(4,7),(6,3)\}$, and hence
	$ (\ell,p,\lambda)\in\{(4, 3, 13),
	(4, 5, 31),
	(4, 7, 19),
	(6, 3, 11)\}$, but then we cannot find a possible value for $k$ satisfying Lemma \ref{lem:six}(a), which is a contradiction. \smallskip
	
	\noindent\textbf{(i.2)}  Let $\ell$ be odd. Then $h=p+1$. Moreover, by \eqref{eq:C4-i2-J-odd}, we conclude that $j\in\{\ell, \ell-2\}$. We will analyze each of these cases separately. \smallskip
	
	\noindent\textbf{(i.2.1)} Let $j=\ell$. Note that $\lambda=(p^{\ell}-1)/m$ is an odd prime divisor of $k$. Then, by \eqref{eq:C4-i2-k-lam}, we have that $p^{\ell}-1$ divides $mu\cdot(p^{\ell}+1)+mh$. Thus $p^{\ell}-1$ must divide $2mu+mh$, where $h=p+1$. Let $u_1$ be a positive integer such that $2mu+mh=u_1\cdot (p^{\ell}-1)$. Then
	\begin{align}\label{eq:C4-u1}
		2u=u_1\cdot\frac{p^{\ell}-1}{m}-h,
	\end{align}
	where $h=p+1$. Note by \eqref{eq:C4-i2-k-lam} and \eqref{eq:C4-u1} that $2hk=2u(p^{\ell}+1)+2h=\lambda[u_{1}(p^{\ell}+1)-mh]$. Since
	$k$ divides $\lambda h\cdot f_{\ell}(p)$, it follows that
	$u_1\cdot(p^{\ell}+1)-mh$ divides $2h^2\cdot f_{\ell}(p)=2h^2(p^{\ell}-1)$.  Since also $\gcd(u_1\cdot (p^{\ell}+1)-mh, p^{\ell}-1) $ divides $|2u_1-mh|$, we conclude that $u_1\cdot (p^{\ell}+1)-mh$ divides $2h^2\cdot |2u_1-mh|$. If $2u_1-mh\geq 0$, then \eqref{eq:C4-u1} implies that $4u=[2u_1(p^\ell-1)-2mh]/m\geq [mh(p^\ell-1)-2mh]/m= h\cdot(p^{\ell}-3)$. Recall by \eqref{eq:C4-i2-u} that $u<h\sqrt{(p^{\ell}-1)/m}$. Thus $p^{\ell}-3<4\sqrt{(p^{\ell}-1)/m}$, and so $m\cdot(p^{\ell}-3)^2<16(p^{\ell}-1)$, which is impossible for even $m$. If $2u_1-mh<0$, then $u_1\cdot (p^{\ell}+1)-mh$ divides $2h^2\cdot (mh-2u_{1})$, and so
	\begin{align}\label{eq:C4-i-2-fin}
		u_1\cdot [p^{\ell}+4h^2+1]\leq 2mh^3+mh.
	\end{align}
	On the other hand by \eqref{eq:C4-i2-k-lam} and the fact that $k$ divides $\lambda h f_{\ell}(p)=h^2\cdot(p^{\ell}-1)^2/m$, we conclude that $u\cdot(p^{\ell}+1)+h$ divides $h^2\cdot(p^{\ell}-1)^2/m$. Since $\gcd(u\cdot(p^{\ell}+1)+h, p^{\ell}-1)$ divides $2u+h$, it follows that  $u\cdot(p^{\ell}+1)+h\leq h^2\cdot(2u+h)^2/m$. Note by \eqref{eq:C4-i2-u} that $2u+h<3u$. Therefore, $m\cdot p^{\ell}< 9h^2\cdot u$. Hence by \eqref{eq:C4-i2-u}, we conclude that $m^3\cdot p^{\ell}<81h^6$, and so
	\begin{align}\label{eq:C4-i2-m}
		m<6h.
	\end{align}
	This together with \eqref{eq:C4-i-2-fin} implies that $u_1\cdot (p^{\ell}+4h^2+1)<13h^4$, and since $h=p+1$, we have that $p^{\ell}+4(p+1)^2+1<13(p+1)^4$, which is true when $\ell=3$ or $(p,\ell)\in\{(3,5), (3,7), (5,5), (7,5), (11,5),(13,5)\}$. For even $m$ satisfying \eqref{eq:C4-i2-m}, we obtain
	\[(p,\ell,\lambda)\in\{(3,5,11),(3,7,1093),(7,5,2801), (11,5,3221),(13,5,30941)\}.\]
	As $v=p^{2\ell}$, for each such $(p,\ell,\lambda)$, we obtain no possible parameter $k$ satisfying Lemma \ref{lem:six}(a).
	
	Suppose now that $\ell=3$. Recall that $\lambda=(p^3-1)/m$ is prime. Let $c:=\gcd(3,p-1)$. Then $(p-1)c$ and $(p^2+p+1)/c$ are coprime, and since $\lambda$ is prime and
	\begin{align*}
		\lambda=\frac{(p-1)c}{\gcd(m,(p-1)c)}\cdot\frac{(p^2+p+1)}{c\cdot \gcd(m,(p^2+p+1)/c)}
	\end{align*}
	we conclude that $(p-1)c=\gcd(m,(p-1)c)$ or $(p^2+p+1)/c=\gcd(m,(p^2+p+1)/c)$, and hence $(p-1)c\mid m$ or $(p^2+p+1)\mid cm$.
	
	Let first that $p^2+p+1$ divides $cm$. Then by \eqref{eq:C4-i2-m}, we have that $m<6h=6(p+1)$, and so $p^2+p+1< 18(p+1)$. This inequality holds only for $p=3,5,7,11,13,17$, and so, $\lambda=13,31,19,19,61,307$, respectively. But this leads to no possible parameter sets satisfying Lemma \ref{lem:six}(a).
	Let now $(p-1)c$ divides $m$. Then there exists positive integer $u_2$ such that $m=u_2\cdot(p-1)$, and so $\lambda=(p^2+p+1)/u_2$, and hence $u_{2}$ is odd. Since $m<6h=6(p+1)$, we conclude that $u_{2}(p-1)<6(p+1)$, and since $u_{2}$ is odd, we have that $u_2\in\{1,3,5,7,9,11\}$.
	
	Note that $\lambda=(p^2+p+1)/u_{2}$ is a prime divisor of $k$. By \eqref{eq:C4-i2-k-lam}, we conclude that $p^{2}+p+1$ divides $u_{2}u(p^{2}-p+1)+u_{2}=u_{2}u(p^{2}+p+1)-u_{2}(2up-1)$, and so $p^{2}+p+1$ divides $u_{2}(2up-1)$. Thus $2u_2 up=n_{1}(p^{2}+p+1)+u_{2}$ for some positive integer $n_{1}$, and this implies that $p$ divides $u_{2}+n_{1}$. Thus $n_{1}=n_{2}p-u_{2}$ for some positive integer $n_{2}$. Then $2u_2 up=n_{1}(p^{2}+p+1)+u_{2}=(n_{2}p-u_{2})(p^{2}+p+1)+u_{2}$, and so
	\begin{align}\label{eq:n2u2}
		2u_2 u=n_{2}(p^{2}+p+1)-u_{2}(p+1).
	\end{align}
	Since $u<h\sqrt{\lambda}$, it follows that $n_{2}\lambda =n_{2}(p^{2}+p+1)/u_{2}<(p+1)(2\sqrt{\lambda}+1)$, and so $n_{2}\sqrt{\lambda}<3h$, that is to say, $n_{2}\sqrt{p^{2}+p+1}<3(p+1)\sqrt{u_{2}}$, where  $u_2\in\{1,3,5,7,9,11\}$. This requires $n_{2}<12$. Note by \eqref{eq:n2u2} that
	\begin{align*}
		u=\frac{n_{2}(p^2+p+1)}{2u_2}+\frac{p+1}{2},
	\end{align*}
	and so by \eqref{eq:C4-i2-k-lam}, we have that
	\begin{align*}
		\lambda=\frac{[n_{2}p^2+(n_{2}+u_{2})p+n_{2}+u_{2}][n_{2}p^4+p^3u_{2}+n_{2}p^2+n_{2}+3u_{2}]}{4(p+1)(p^3-1)u_{2}^2}.
	\end{align*}
	On the other hand, we have that $\lambda=(p^{2}+p+1)/u_{2}$. Therefore,
	\begin{align}\label{eq:lam}
		\nonumber(n_{2}^2-4u_{2})p^6+
		(2n_{2}u_{2}-8u_{2}+n_{2}^2)p^5+
		(u_{2}^2+2n_{2}u_{2}-8u_{2}+2n_{2}^2)p^4+\\
		(n_{2}+u_{2})^2p^3+
		(4n_{2}u_{2}+8u_{2}+2n_{2}^2)p^2+
		(3u_{2}^2+4n_{2}u_{2}+8u_{2}+n_{2}^2)p+\\
		\nonumber(3u_{2}^2+4n_{2}u_{2}+4u_{2}+n_{2}^2)&=0,
	\end{align}
	where $1\leq n_{2}<12$ and $u_2\in\{1,3,5,7,9,11\}$. But for each such a pair $(u_{2},n_{2})$, the equation \eqref{eq:lam} has no integer solution for $p$, which is a contradiction.\smallskip
	
	\noindent\textbf{(i.2.2)} Let $j=\ell-2$. Then  $\lambda=(p^{\ell-2}-1)/m$ is an odd prime divisor of $k$. By \eqref{eq:C4-i2-k-lam} and the fact that $k$ divides $\lambda hf_{\ell}(p)$, we conclude that $u\cdot (p^{\ell}+1)+h$ divides $\lambda h^2f_{\ell}(p)=h^2\cdot (p^{\ell-2}-1)(p^{\ell}-1)/m$. Note that $\gcd(u\cdot (p^{\ell}+1)+h, p^{\ell}-1)$ divides $2u+h$. Therefore, $u\cdot (p^{\ell}+1)+h$ divides $h^2\cdot (2u+h)(p^{\ell-2}-1)/m$. Let $u_1$ be a positive integer such that
	\begin{align}\label{eq:C4-i2-l2}
		mu_1\cdot[u\cdot (p^{\ell}+1)+h]=h^2\cdot (2u+h)(p^{\ell-2}-1)
	\end{align}
	Note by \eqref{eq:C4-i2-u} that $2u+h<3u$. Then $mu_1\cdot (p^{\ell}+1)<3h^2(p^{\ell-2}-1)$, and so $mu_1\cdot (p^{\ell}+1)<3(p+1)^2(p^{\ell-2}-1)$. This inequality holds only when $mu_1<6$. Since $m$ is even, we conclude that
	$mu_{1}=2$ or $mu_1=4$. If $mu_{1}=2$, then by \eqref{eq:C4-i2-l2}, we have that $2u[p^{\ell}+1-(p+1)^2(p^{\ell-2}-1)]=(p+1)[(p+1)^2(p^{\ell-2}-1)-2]>0$, and so $p^{\ell}+1>(p+1)^2(p^{\ell-2}-1)$, which is impossible. Therefore, $mu_{1}=4$, and hence \eqref{eq:C4-i2-l2} implies that $2u[2(p^{\ell}+1)-(p+1)^2(p^{\ell-2}-1)]=(p+1)[(p+1)^2(p^{\ell-2}-1)-4]<8h[2(p^{\ell}+1)-(p+1)^2(p^{\ell-2}-1)]$, and so $u<4h$. Since $f_{\ell}(p)$ divides $2u^2+uh$ by \eqref{eq:C4-i2-k-lam}, it follows that $2u^2+uh<36h^2$, and hence $p^{\ell}-1<36(p+1)^2$. Then  $(p,\ell)=(3,5)$, or $p\leq 37$ and $\ell=3$.  Note here that $mu_{1}=4$. Then $m=2,4$, and since $\lambda=(p^{\ell-2}-1)/m$, we obtain $\lambda=5,7,11,13$ when $(p,\ell,m)=(11,3,2), (29,3,4),(23,3,2),(3,5,2)$, respectively. But this leads to no possible parameter $k$ satisfying Lemma \ref{lem:six}(a). \smallskip
	
	\noindent \textbf{(ii)} Let $i=3$. Note that $\ell>i=3$. By \eqref{eq:C4-k-h}, we know that $k$ divides $\lambda h\cdot f_{\ell}(p)$, where $h=p^2+p+1$ and $f_{\ell}(p)=p^{\ell}-1$. Then there exists a positive integer  $u$ such that $uk=\lambda h f_{\ell}(p)$. Since $v-1=p^{3\ell}-1$, by \eqref{eq:k-lam}, we have that
	\begin{align}\label{eq:C4-i3-k-lam}
		k=1+\frac{u\cdot (p^{2\ell}+p^{\ell}+1)}{h} \quad \text{ and }  \quad h^2\lambda=u^2\cdot(p^{\ell}+2)+\frac{3u^2+uh}{f_{\ell}(p)}.
	\end{align}
	
	Suppose first that $j$ is even. Note that $\lambda$ is an odd prime divisor of $p^j-1$ where $1\leq j\leq \ell$. Then there exists an even positive integer such that $m\lambda=p^{j/2}-\e1$, where $\e=\pm$. By \eqref{eq:C4-i3-k-lam}, we have that $h^2\cdot(p^{j/2}-\e1)-mu^2\cdot(p^{\ell}+2)>0$, and so
	\begin{align*}
		u^2<\frac{h^2\cdot (p^{j/2}-\e1)}{m\cdot(p^{\ell}+2)}.
	\end{align*}
	We know that $8(p^{\ell/2}+1)<(p^{\ell}+2)$ for $\ell\geq 4$, and so as $j\leq \ell$ and $m$ is even, the latter inequality yields $u< h/4$. This together with that fact that $f_{\ell}(p)$ divides $3u^2+uh$ implies that $16(p^{\ell}-1)<7h^2$, or equivalently, $16(p^{\ell}-1)<7(p^2+p+1)^2$, which is impossible for odd prime $p$.\smallskip

	Suppose now that $j$ is odd. Then $m\lambda=p^j-1$ with  $1\leq j\leq \ell$, for some even positive integer $m$. Again, \eqref{eq:C4-i3-k-lam} implies that $u^2<h^2\cdot(p^{j}-1)/[m(p^{\ell}+2)]\leq h^2\cdot(p^{\ell}-1)/[m(p^{\ell}+2)]<h^2/2$, and so $u<h/\sqrt{2}$. This together with the fact that $f_{\ell}(p)=p^{\ell}-1$ divides $3u^2+uh$, we conclude that $p^{\ell}-1<3h^2$. This inequality holds only when $(p,\ell)=(3,5)$ or $\ell=4$. If $(p,\ell)=(3,5)$, then since $\lambda=(3^j-1)/m$ with $1\leq j\leq 5$ odd and $m$ even, we obtain $\lambda=11,13$ when $(m,j)=(2,3),(22,5)$, respectively. However, Lemma \ref{lem:six}(a) gives no possible parameter $k$. If $\ell=4$, then by \eqref{eq:C4-i3-k-lam} and the fact that $k$ divides $\lambda hf_{\ell}(p)$, we conclude that $u\cdot (p^{8}+p^{4}+1)+h $ divides $h^2\cdot (p^{j}-1)(p^{4}-1)/m$, and since $\gcd(u\cdot (p^{8}+p^{4}+1)+h, p^{4}-1)$ divides $3u+h$, it follows that $u\cdot (p^{8}+p^{4}+1)+h$ divides $h^2\cdot (3u+h)(p^{j}-1)/m$. Since also $u<h/\sqrt{2}$, we conclude that $m\cdot (p^{8}+p^4+1)<4h^3(p^{j}-1)$. If $j=1$, then this inequality holds for $(p,m)=(3,2)$ which implies that $\lambda=(p-1)/2=1$, which is not our case. Then since $\ell=4$, we have that $j=3$. We repeat our argument above to find a better bound for $u$ in terms of $h$. Since $u^2<h^2\cdot(p^{j}-1)/[m(p^{\ell}+2)]=h^2\cdot(p^{3}-1)/[m(p^{4}+2)]<h^2/(2p)$, and hence $u<h/\sqrt{2p}$. Since $p^{\ell}-1=p^4-1$ divides $3u^{2}+uh$, it follows that $p^{4}-1<(3+\sqrt{2p})h^2/(2p)$, or equivalently, $[2p(p^4-1)-3h^2]^2<2ph^4$, where $h=p^2+p+1$. This inequality is true only when $p=3$, and so $\lambda=(3^3-1)/m$ has to be $13$, and hence $v=3^{9}$. But this leads to no possible parameter $k$ satisfying Lemma \ref{lem:six}(a).\smallskip

	\noindent \textbf{(iii)} Let $i=4$ and $\ell\in \{5,6,7\}$. By \eqref{eq:C4-k-h}, there exists a positive integer $u$
	such that $uk=\lambda h\cdot f_{\ell}(p)$, where $f_{\ell}(p)=p^{\ell}-1$ and $h=(p^{4}-1)/(p-1)$.
	Since $v-1=p^{4\ell}-1$, by \eqref{eq:k-lam}, we have that
	\begin{align}\label{eq:k-lam-C4-i4}
		k=1+\frac{u\cdot (p^{3\ell}+p^{2\ell}+p^{\ell}+1)}{h} \quad \text{ and }  \quad h^2\lambda=u^2\cdot(p^{2\ell}+2p^{\ell}+3)+\frac{4u^2+uh}{f_{\ell}(p)}.
	\end{align}
	Moreover, $m\lambda=p^{j}-1$, for some even positive integer $m$. By the same argument as in the previous cases, using  \eqref{eq:k-lam-C4-i4}, we conclude that $u^2<h^2(p^j-1)/[m\cdot(p^{2\ell}+2p^{\ell}+3)]$, and since $\ell\in\{5,6,7\}$ and $m$ is even, it follows that $u< p$. Then by \eqref{eq:k-lam-C4-i4} and the fact that $\lambda$ is an integer, we conclude that $p^{\ell}-1$ divides $hp+4p^2$, and hence $p^{5}-1<p^{4}+p^{3}+5p^{2}+p+1$, which has no suitable solution for $p$.
	\smallskip
	
	\noindent \textbf{(3)}  Suppose now that $H$ lies in a member of $\Cmc_{7}$. Then $H$ stabilises a tensor product of spaces of $V$, so $V_{d}(p)=V_{t} \bigotimes \cdots \bigotimes V_{t}$ with $d=t^{\ell}$ and  $\ell,t>1$. Moreover, $H\leq N_{\GL_{d}(p)}(\GL_{t}(p) \circ \cdots \circ \GL_{t}(p))$. Here the nonzero vectors $v_1\otimes \cdots \otimes v_{\ell}$ form a union of $H$-orbits of length $h_{t,\ell}(p):=(p^t-1)^{\ell}/(p-1)^{\ell-1}$. Then by Lemma~\ref{lem:six}, we conclude that
	\begin{align}\label{eq:C7-k}
		k \quad \text{divides} \quad \lambda\cdot h_{t,\ell}(p).
	\end{align}
	Note by Lemma~\ref{lem:six}(c) and Proposition~\ref{prop:lam-p} that $\lambda$ is an odd prime divisor of $p^{j}-1$, for some $1\leq j \leq t$.
	Since $v=p^{t^{\ell}}$, it follows that from \eqref{eq:u} that $p^{t^{\ell}}<\lambda\cdot (p^{t}-1)^{2\ell}/(p-1)^{2\ell-2}$. Since $\lambda<p^{t}$,  we have that $t^{\ell}<t\cdot(2\ell+1)$, or equivalently, $t^{\ell-1}<2\ell+1$ which is true when $(t,\ell)\in\{(2,2), (2,3), (2,4), (3,2), (4,2)\}$. For each pair $(t,\ell)$, let $g_{t,\ell}(p):=\gcd((v-1)/(p^{t}-1),h_{t,\ell}(p)/(p^{t}-1))$. Then $g_{t,\ell}(p)$ divides $2$, $2^{4}$, $2^{9}$, $3$ or $4$ if $(t,\ell)=(2,2)$, $(2,3)$, $(2,4)$, $(3,2)$ or $(4,2)$, respectively. Then by \eqref{eq:C7-k} and Lemma \ref{lem:six}(a), we have that
	\begin{equation}\label{eq:C7-2-1}
		k \mid \lambda\cdot g_{t,\ell}(p)\cdot(p^{t}-1).
	\end{equation}
	Let $(t,\ell)=(2,2)$. Then by \eqref{eq:C7-2-1},
	$k$ divides $2\lambda\cdot (p^{2}-1)$, and so $uk=2\lambda (p^{2}-1)$ for some positive integer $u$, and since  $v-1=p^{4}-1$, by \eqref{eq:k-lam}, we have that
	\begin{align}\label{eq:C7-k-lam}
		2k=u\cdot (p^{2}+1)+2 \quad \text{ and }  \quad 4\lambda=u^2+\frac{2u^2+2u}{p^{2}-1}.
	\end{align}
	Since $\lambda$ is an odd prime divisor of $p^j-1$ with $j\in\{1,2\}$, there exists an even positive integer $m$ such that $m\lambda=p-\e1$, where $\e=\pm$. Then by \eqref{eq:C7-k-lam} and the fact that $k$ divides $2\lambda (p^{2}-1)$, we conclude that
	$u\cdot (p^{2}+1)+2$  divides $2u(p-\e1)(p^2-1)=2(p-\e1)[u\cdot(p^2+1)+2]-4\cdot(p-\e1)(u+1)$,  then $u\cdot(p^2+1)+2$ divides $4\cdot(p-\e1)(u+1)$, and so $u\cdot(p^2+1)<8u(p-\e1)\leq 8u(p+1)$. Therefore, $p=3,5,7$, but for such a $p$, the prime divisors of $p-\e1$ are $2$ and $3$,  and hence $\lambda=2$ or $3$, which is not our case. For the remaining cases, that is to say, when  $(t,\ell)=(2,3), (2,4), (3,2), (4,2)$, by \eqref{eq:k-lam}, there exists a positive integer $u$ such that
	\begin{align*}
		g_{t,\ell}(p)k=\frac{ u\cdot(p^{t^{\ell}}-1)}{p^{t}-1}+g_{t,\ell}(p),
	\end{align*}
	where $g_{t,\ell}(p)$ divides $2^{4}$, $2^{9}$, $3$ or $4$, respectively. Moreover, $\lambda$ divides $p^{j}-1$, with $1\leq j\leq t$. Then by \eqref{eq:C7-2-1}, we conclude that $u\cdot(p^{t^{\ell}}-1)+g_{t,\ell}(p)(p^{t}-1)\leq g_{t,\ell}(p)^{2}(p^{j}-1)(p^{t}-1)^{2}$. Since $\lambda$ is a prime divisor of $p^{j}-1$ with $j\leq t$, it follows that $\lambda\leq p^{2}+p+1$, and hence
	\begin{align}\label{eq:C7-3}
		(p^{t^{\ell}}-1)+g_{t,\ell}(p)(p^{t}-1)\leq g_{t,\ell}(p)^{2}\cdot (p^{2}+p+1)(p^{t}-1)^{2}.
	\end{align}
	Thus
	\[
	\begin{array}{ll}
		p=3,5,7,11,13&\text{if $(t,\ell)=(2,3)$,} \\
		p=3&\text{if $(t,\ell)=(2,4)$,} \\
		p=3,5,7&\text{if $(t,\ell)=(3,2)$.} \\
	\end{array}
	\]
	Again since $\lambda$ is a prime divisor of $p^{j}-1$ with $j\leq t$, we conclude that $\lambda=2,3,5,7$ if $(t,\ell)=(2,3)$, $\lambda=2$ if $(t,\ell)=(2,4)$, $\lambda=2,3,13,19,31$ if $(t,\ell)=(3,2)$, however, we cannot find any possible parameter $k$ satisfying Lemma \ref{lem:six}(a).\smallskip
	
	\noindent \textbf{(4)} Suppose now that $H$ lies in a member of $\Cmc_{5}$. Then $H\leq \N_{\GL_{d}(p)}(\GL_{n}(q_{0}))$ with $q=q_{0}^{t}$
	and $t>1$, but this normaliser lies in a subgroup of $\GL_{t}(q_{0})\circ
	\GL_{n}(q_{0})$ of $\GL_{tn}(q_{0})\leq \GL_{d}(p)$, and hence $H$ lies in a
	maximal member of type $\Cmc_{4}$ or $\Cmc_{7}$ of $\GL_{d}(p)$, which have been ruled out in cases (2) and (3), respectively.\smallskip

	\noindent \textbf{(5)} Suppose finally that $H$ lies in a member of $\Cmc_{6}$. Then $H$ lies in the normaliser of an irreducible symplectic type $s$-group $R$, where $s$ is prime and $s\neq p$. Note in this case that Aschbacher's proof in \cite[Section 11]{a:Aschbacher-84} shows that we may assume that $G_{0}$ contains $R$, otherwise, $G_{0}$ lies in some other family $\mathcal{C}_{i}$, but these possibilities have already been ruled out. Therefore, it follows from Lemma 3.7, Lemma 3.10 and its proof and Corollary 3.12 in \cite{a:Foulser-69} that $\Zbb_{s} \cong \Z(R) \leq \Z(\GL_{n}(q))$, and hence $s>2$ by Lemma~\ref{lem:diff}.
	In this case, by Lemma \ref{lem:six}(c), we have that
	\begin{align}\label{eq:C6-k-1}
		k \quad \text{divides} \quad \log_{p}(q)\cdot (q-1)\cdot s^{m^{2}+2m} \cdot \prod_{j=1}^{m}(s^{2j}-1)).
	\end{align}
	By the proof of \cite[Lemma 3.7]{a:Liebeck-98-Affine}, we also have
	\begin{align}\label{eq:C6-k-0}
		k \quad \text{divides} \quad \lambda \cdot  (q-1)\cdot \gcd(\frac{q^{s^m}-1}{q-1},  s^{m^2+2m}\cdot\log_{p}(q) \cdot \prod_{j=1}^{m}(s^{2j}-1)),
	\end{align}
	where $s$ divides $q-1$. Since $\gcd((q^{s^m}-1)/(q-1),s^{m^2+2m})$ divides $s^{m}$, we have that
	\begin{align}\label{eq:C6-k}
		k \quad \text{divides} \quad \lambda \cdot  s^{m}\cdot  (q-1)\cdot \gcd(\frac{q^{s^m}-1}{q-1},  \log_{p}(q) \cdot \prod_{j=1}^{m}(s^{2j}-1)),
	\end{align}
	Here, $v=q^{s^{m}}$, where $s$ is an odd prime and $q=p^a$ with $p\neq 2$ and $a\geq 1$. Since $\lambda$ is an odd prime divisor of $k$, it follows from \eqref{eq:C6-k}, we conclude that $\lambda$ divides $a$, $q-1$, $s$, $s^{j}-1$ or $s^{j}+1$, where $j \in \{1,\ldots ,m\}$. Since $s$ is odd and $p\neq 2$, we get $\lambda<(q-1)\cdot s^{m}$. Then by \eqref{eq:u} and \eqref{eq:C6-k-0}, we have
	\begin{align}\label{eq:C6-i-odd}
		q^{s^m-3}< s^m\cdot \gcd(\frac{q^{s^m}-1}{q-1}, a\cdot s^{m^2+2m}\cdot\prod_{j=1}^{m}(s^{2j}-1))^2,
	\end{align}
	and so
	\begin{align*}
		q^{s^m-3}< s^{3m}\cdot\gcd(\frac{q^{s^m}-1}{q-1}, a\prod_{j=1}^{m}(s^{2j}-1))^2.
	\end{align*}
	Since also $s\leq q-1$ and $a^{2}\leq q-1$, it follows that $s^m<2m^2+5m+5$ which is true when $(s, m)\in\{(3,1), (3,2), (3,3), (5,1), (7,1), (11,1) \}$. If $(s,m)=(3,1)$, then $q\equiv 1\mod{3}$ and by \eqref{eq:C6-k}, we know that $k$ divides $3a\lambda (q-1)$. Then there exists a positive integer $u$  such that $uk=3a\lambda (q-1)$. Since $v-1=q^{3}-1$, by \eqref{eq:k-lam}, we have that
	\begin{align}\label{eq:C6-k-lam}
		3ak=u\cdot (q^2+q+1)+3a \quad \text{ and }  \quad 9a^2\lambda=u^2(q+2)+\frac{3u^2+3au}{q-1}.
	\end{align}
	Note that here $\lambda$ divides $a$ or $q-1$. Then by \eqref{eq:u}, we conclude that
	\begin{align}\label{eq:C6-u}
		u<3a.
	\end{align}
	As $\lambda$ is an integer, by \eqref{eq:C6-k-lam} and \eqref{eq:C6-u}, we get $p^a\leq 36a^2$. Since $3\mid q-1$, we conclude that $q\in \{7,13,19,25,31,49,121\}$. These values of $q$ do not lead to any possible parameter set. If $(s,m)\neq (3,1)$, then it is easy to check that the inequality \eqref{eq:C6-i-odd} does not hold unless $(s,m,p,a)=(5,1,11,1)$, for which we obtain no possible parameter set by Lemma \ref{lem:six}(a).
\end{proof}

\subsection{Quasisimple case}

In this section, we deal with the case where  $H^{(\infty)}$, the last term in the derived series of $H:=G_{0}$, is quasisimple, and its action on $V=V_{n}(q)$ is absolutely irreducible and not realizable over any proper subfield of $\Fbb_q$ as in Proposition~\ref{pro:Aschbacher}(b). Thus  $L=H^{(\infty)}/Z(H^{(\infty)})$ is a finite non-abelian simple group, and so it is an alternating group, a sporadic group, a group of Lie type in characteristic $p$, or a group of Lie type in characteristic $p'$. We follow the method given in  \cite{a:Liebeck-98-Affine}.
Moreover, by the argument given in \cite[p. 210]{a:Liebeck-98-Affine} for symmetric designs we frequently use the facts in Lemma \ref{lem:affine-elem}.

\begin{lemma}\label{lem:affine-elem}
	Let $\Dmc$ be a nontrivial symmetric $(v,k,\lambda)$ design with $\lambda\geq 5$. Let also $G$ be a flag-transitive automorphism group of $\Dmc$ of affine type satisfying {\rm Proposition~\ref{pro:Aschbacher}(b)}, then
	\begin{enumerate}[\rm \quad (a)]
		\item $k$ divides $\lambda\gcd(q^{n}-1,(q-1)|\Aut(L)|)$;
		\item $q^{n}<\lambda\gcd(q^{n}-1,(q-1)|\Aut(L)|)^{2}$.
	\end{enumerate}
	In particular, if $G$ is not a subgroup of $\AGaL_{1}(q)$ and $\lambda$ is an odd prime, then $\lambda$ divides $(q-1)_{2'}$ or $|\Aut(L)|_{2'}$, and so $\lambda \leq \max\{(q-1)_{2'},|\Aut(L)|_{2'}\}$.
\end{lemma}
\begin{proof}
	Since $v=q^n$ and $|G_{0}|$ is a divisor of $(q-1)|\Aut(L)|$, it follows from Lemma~\ref{lem:six}(a) and (c) that $k$ divides $\lambda \gcd(q^{n}-1,(q-1)|\Aut(L)|)$. Since also $\lambda v<k^{2}$, we have that $q^{n}<\lambda \gcd (q^{n}-1,(q-1)|\Aut(L)|)^{2}$.
	In particular, if $G$ is not a subgroup of $\AGaL_{1}(q)$ and $\lambda$ is an odd prime, then \cite{a:Biliotti-CP-sym-affine} implies that $\lambda$ is a prime divisor of $k$, and so by part (a), $\lambda$ divides $(q-1)|\Aut(L)|$. Thus $\lambda$ divides $(q-1)_{2'}$ or $|\Aut(L)|_{2'}$, and consequently, $\lambda \leq \max \{(q-1)_{2'},|\Aut(L)|_{2'}\}$.
\end{proof}

\begin{lemma}\label{lem:affine-Alt}
	The group $L$ cannot be an alternating group $\A_c$ with $c\geq 5$.
\end{lemma}

\begin{proof}
	Let $L$ be an alternating group $\A_c$ of degree $c\geq 5$. Since $\lambda$ is an odd prime divisor of $k$, it follows from Lemma~\ref{lem:affine-elem} that $\lambda \leq \max \{c,(q-1)/2\}$, and so
	\begin{equation}\label{eq:A-2,lam}
		\lambda <c\cdot(q-1)/2.
	\end{equation}
	We first show that $V=V_n(q)$ is not the fully deleted permutation module for $\A_c$. Assume to the contrary that $V=V_n(q)$ is the fully deleted permutation module. Then $q=p$, $n=c-1$ if $p\nmid c$ and $n=c-2$ if $p\mid c$. Also by the proof of \cite[Lemma 4.1]{a:Liebeck-98-Affine}, we conclude that  $k$ divides $\lambda f(p)$, where
	\begin{align*}
		f(p)=\left\{
		\begin{array}{ll}
			c\cdot (p-1), & \hbox{if $p\nmid c$,}\\
			c(c-1)\cdot(p-1)/2, & \hbox{if $p\mid c$.}
		\end{array}
		\right.
	\end{align*}
	If $p\nmid c$, then by inequalities \eqref{eq:u} and \eqref{eq:A-2,lam}, we conclude that $2p^{c-1} < c^3(p-1)^3$ which is true only when $c=5$ and $p=3$ or $7\leq p\leq 59$, $c=6$ and $p=5,7$, or $c=7$ and $p=3$. By Lemma \ref{lem:six},  we obtain $(v,k,\lambda)$ is $(81, 16, 3)$ or $(81, 80, 79)$ when $(c,p)=(5,3)$, however, both possibilities are not our case.
	If $p\mid c$, then \eqref{eq:u} and \eqref{eq:A-2,lam} imply that $8p^{c-2} < c^3(c-1)^2(p-1)^3$ which is true only when $(c,p)=(5,5), (6,3), (7,7), (9,3), (12,3)$, but for each case, we obtain no possible parameter by Lemma \ref{lem:six}.\smallskip

	Therefore, $V=V_n(q)$ is not the fully deleted permutation module for $\A_c$. Let $q=p^a$, where $p$ is an odd prime. In this case, $|\Aut(L)|$ is $c!$ or $2(c!)$ respectively for $c\neq 6$ or $c=6$. By Lemma~\ref{lem:affine-elem}, we have that
	\begin{align}\label{eq:Ai-q-odd-0}
		q^{n}<2c(q-1)^{3}\cdot (c!)_{p'}^2,
	\end{align}
	and so $q^{n-3}<2c(c!)^{2}$, and since $c!<c^{c-1}$ for $c\geq 5$, we conclude that $q^{n-3}<2c^{2c}$. Since also $p$ is odd, we have that $\log_{2}(p)>3/2$, and so
	\begin{align}\label{eq:Ai-q-odd-2}
		3a\cdot(n-3)<4c\cdot \log_2(c).
	\end{align}
	We now consider the following cases:\smallskip
	
	\noindent \textbf{(a)} Suppose first that $c>15$. Then $H^{(\infty)}=\A_c$ by Lemma \ref{lem:diff}, and hence \cite[Theorem 7]{a:James-mini-dim-1983} implies that $n\geq c(c-5)/4$. By \eqref{eq:Ai-q-odd-2}, we have that
	\begin{align}\label{eq:Ai-q-odd-1}
		3a\cdot(c^2-5c-12)<16c\cdot\log_2c,
	\end{align}
	which is true for $c\leq 32$. By \eqref{eq:Ai-q-odd-0}, we conclude that $q=3$ for $c\in \{16,\ldots,22\}$, which leads to no possible parameters by Lemma \ref{lem:six}.\smallskip

	\noindent\textbf{(b)} Suppose now that $12\leq c\leq 15$. Then $H^{(\infty)}=\A_c$ by Lemma \ref{lem:diff}, and by \cite{b:Atlas-Brauer,a:LP-1992}, we conclude that $n\geq 43$. If $q\geq 5$, then $q^{n-3}\geq 5^{40}>2\cdot15\cdot(15!)^2\geq 2c(c!)^2$, which is a contradiction. Thus $q=3$, and again $3^{n-3}=q^{n-3}<2c(c!)^2\leq 2\cdot 15\cdot(15!)^2$ implies that $n\leq 56$. Therefore, $v=3^n$ with $43\leq n\leq 56$. But for these possible values of $v$, we cannot find a possible parameter set satisfying Lemma \ref{lem:six}.\smallskip
	
	\noindent\textbf{(c)} Suppose now that $5\leq c\leq 11$. If $n\geq 4$, then by \cite{b:Atlas-Brauer} and \eqref{eq:Ai-q-odd-0}, we have one of the possibilities as below:
	\begin{enumerate}[\rm (1)]
		\item $5\leq c\leq 7$, $4\leq n\leq 20$ and $q\leq 355622400$;
		\item $8\leq c\leq 11$, $8\leq n\leq 37$ and $q\leq 2029$.
	\end{enumerate}
	For these values of $q$, by Lemma \ref{lem:six}, as $\lambda\geq 5$, we obtain $(v,k,\lambda)=(15625, 280, 5)$ when $c=7$, $n=6$ and $q=p=5$, but this case is ruled out in Proposition \ref{prop:lam-p}.
	
	We now deal with the remaining cases where $c=5,6,7$ and $n=2,3$, and so by inspecting the ordinary and modular characters of $\A_c$ and their covering groups from \cite{b:Atlas,b:Atlas-Brauer}, we observe that $(c,n)$ is $(7,3)$, $(6,2)$, $(6,3)$, $(5,2)$ or $(5,3)$.
	
	Let $(c,n)=(7,3)$. By \cite{b:Atlas,b:Atlas-Brauer}, we conclude that $p=5$ and $H^{(\infty)}=3\cdot \A_7$. Then by Lemma \ref{lem:six}(c), as $k$ is a divisor $7!\cdot (q-1)$ and $v<k^{2}$, we have that $q<(7!)^2$, and so $q=5^{a}$ with $1\leq a\leq \lfloor 2\log_{5}(7!) \rfloor$. Again by Lemma~\ref{lem:six}, we obtain the parameter set $(v,k,\lambda)=(15625, 280, 5)$ when $q=5^{2}$, which is a contradiction by Proposition \ref{prop:lam-p}. \smallskip
	
	Let $(c,n)=(6,3)$. Then since $p$ is an odd prime, by \cite{b:Atlas,b:Atlas-Brauer}, we conclude that $p\geq 3$. So by Lemma \ref{lem:six}(c), the parameter $k$ divides $2\cdot 6!\cdot (q-1)$. The fact that $v<k^{2}$ implies that $q<2073600$, however for these values of $q$, we obtain no possible parameters. By a similar argument, the case where $(c,n)=(5,3)$ can be ruled out.  \smallskip
	
	Let now $(c,n)=(6,2)$. Then it follows from \cite{b:Atlas,b:Atlas-Brauer} that $p=3$ and $H^{(\infty)}=2\cdot \A_6$, which is impossible by Lemma \ref{lem:diff}.
	
	Let finally $(c,n)=(5,2)$. By \cite{b:Atlas,b:Atlas-Brauer}, we have that $p\geq 3$. By Lemma \ref{lem:six}(c), there exists a positive integer $u$ such that $uk=5!\cdot(q-1)$, and so $v<k^{2}$ implies that  $u<120$, and again by Lemma~\ref{lem:six}(a), we conclude that $q+1$ divides $120(240+u)$, where $u<120$. Therefore, $q<3\cdot 120^2=43200$. We now apply Lemma~\ref{lem:six}, and obtain the parameter sets $(v,k,\lambda)=(49, 16, 5)$ or $(121, 25, 5)$ respectively when $q=7$ or $11$. Both cases are ruled out since $-1$ lies in $H^{(\infty)}$, and hence in $H$, by \cite{b:Atlas}.
\end{proof}

\begin{lemma}\label{lem:sporadic}
	The group $L$ cannot be a sporadic simple group or Tits group $^{2}\F_{4}(2)'$.
\end{lemma}

\begin{proof}
	Suppose that $L$ is a sporadic group or Tits group $^{2}\F_{4}(2)'$. By Lemma \ref{lem:affine-elem}(b), we have
	\begin{equation}\label{eq:sporadic}
		q^{n}<\lambda(q-1)^2\cdot |\Aut(L)|^2,
	\end{equation}
	where $\lambda\neq p$ is an odd prime divisor of $(q-1)|\Aut(L)|$ greater than $3$.
	A lower bound $\l_{n}(L)$ for the minimal degree $R(L)$ of nontrivial projective representations of $L$ for all primes $p$ can be read off from  \cite[2.3.2]{a:LPS90} and \cite{b:Atlas,b:Atlas-Brauer}. Therefore, $n\geq \l_{n}(L)$ for each $L$, and hence, by \eqref{eq:sporadic}, we observe that $L$ cannot be one of the groups $\J_4$, $\mathrm{\Ly}$, $\Fi_{23}$, $\mathrm{Fi}^{'}_{24}$, $\BM$ and $\M$. For the remaining groups which are listed in Table~\ref{tbl:sporadic}, as $n\geq \l_n(L)$, by \eqref{eq:sporadic}, we can find an upper bound $\u_n(L)$ of $n$ and
	an upper bound $\u_{q}(L)$ of $q$ as in the fourth and fifth columns of Table~\ref{tbl:sporadic}, respectively. For these values of $n$ and $q$, we use Lemma~\ref{lem:six}, and obtain the parameter set $(15625, 280, 5)$ when $L=\M_{22}$ or $\J_{2}$ for $(n,q)=(6,5)$, but it is impossible by Proposition \ref{prop:lam-p}.
\end{proof}
\begin{table}
	\centering
	\small
	\caption{Some bounds for $n$ and $q$ when $L$ is a sporadic simple group of the Tits group $^{2}\F_4(2)'$.}\label{tbl:sporadic}
		\begin{tabular}{lllll}
			\noalign{\smallskip}\hline\noalign{\smallskip}
			$L$ & $|\Aut(L)|$ & $\l_n(L)$ & $\u_n(L)$ & $\u_q(L)$ \\\noalign{\smallskip}\hline \noalign{\smallskip}
			$\M_{11}$ & $2^4\cdot 3^2\cdot 5\cdot 11$ & $5$ & 
			$18$ & $397$ \\
			$\M_{12}$ & $2^7\cdot 3^3\cdot 5\cdot 11$ & $6$ & 
			$24$ & $435$\\
			$\M_{22}$ & $2^8\cdot 3^2\cdot 5\cdot 7\cdot 11$ & $6$ & 
			$26$ & $941$ \\
			$\M_{23}$ & $2^7\cdot 3^2\cdot 5\cdot 7 \cdot 11\cdot 23$ & $11$ & 
			$31$ & $36$\\
			$\M_{24}$ & $2^{10}\cdot 3^3\cdot 5\cdot 7\cdot 11\cdot 23$ & $11$ & 
			$37$ & $73$\\
			$\J_{1}$ & $2^{3}\cdot 3\cdot 5\cdot 7\cdot 11\cdot 19$ & $7$ & 
			$23$ & $125$ \\
			$\J_{2}$ & $2^{8}\cdot 3^3\cdot 5^2\cdot 7$ & $6$ & 
			$27$ & $1099$ \\
			$\J_{3}$ & $2^{8}\cdot 3^5\cdot 5\cdot 17\cdot 19$ & $9$ & 
			$35$ & $193$\\
			$\HS$ & $2^{10}\cdot 3^2\cdot 5^3\cdot 7\cdot 11$ & $20$ & 
			$35$ & $7$\\
			$\McL$ & $2^8\cdot 3^6\cdot 5^3\cdot 7\cdot 11$ & $21$ & 
			$40$ & $9$ \\
			$\Suz$ & $2^{14}\cdot 3^7 \cdot 5^2 \cdot 7\cdot 11\cdot 13 $ & $12$ & $52$ & $245$ \\
			$\He$ & $2^{11} \cdot 3^3 \cdot 5^2 \cdot 7^3 \cdot 17 $ & $18$ & 
			$43$ & $17$ \\
			$\Ru$ & $2^{14} \cdot 3^3 \cdot 5^3 \cdot 7 \cdot 13\cdot 29 $ & $28$ & $48$ & $7$ \\
			$\ON$ & $2^{10} \cdot 3^4 \cdot 5 \cdot 7^3 \cdot 11\cdot 19\cdot 31 $ & $31$ &  $52$ & $6$ \\
			$\Co_1$ & $2^{21} \cdot 3^9 \cdot 5^4 \cdot 7^2 \cdot 11\cdot 13\cdot 23$ & $24$ & $80$ & $49$ \\
			$\Co_2$ & $2^{18} \cdot 3^6 \cdot 5^3 \cdot 7 \cdot 11\cdot 23$ & $22$ & $59$ & $23$ \\
			$\Co_3$ & $2^{10} \cdot 3^7 \cdot 5^3 \cdot 7 \cdot 11\cdot 23$ & $22$ & $51$ & $14$ \\
			$\Fi_{22}$ & $2^{18} \cdot 3^9 \cdot 5^2 \cdot 7 \cdot 11\cdot 13$ & $27$ & $61$ & $13$  \\
			$\HN$ & $2^{15} \cdot 3^6 \cdot 5^6 \cdot 7 \cdot 11\cdot 19$ & $56$ & $63$ & $3$ \\
			$\Th$ & $2^{15} \cdot 3^{10} \cdot 5^3 \cdot 7^2 \cdot 13\cdot 19\cdot 31$ & $48$ &  $73$ & $5$ \\
			$^{2}\F_{4}(2)'$ & $2^{11} \cdot 3^{3} \cdot 5^2\cdot 13$ & $26$ & $32$ & $4$ \\
			\noalign{\smallskip}\hline \noalign{\smallskip}
		\end{tabular}
\end{table}

\subsubsection{Lie type groups in defining characteristic}

In this section, we assume that $L$ is a finite simple group of Lie type in characteristic $p$, and show that Theorem \ref{thm:main} holds in this case. We use the same machinery developed by Liebeck in \cite{a:Liebeck-98-Affine}. Recall that $V=V_n(q)$ is an absolutely irreducible module for $H^{(\infty)}$ realized over no proper subfield of $\Fbb_q$, where $q=p^a$. Suppose that $L=H^{(\infty)}/Z(H^{(\infty)})=L(s)$ is a group of Lie type over $\mathbb{F}_s$, where $s$ is a power of $p$. Since $p$ is odd, it follows that $s$ is odd.

\begin{lemma}\label{lem:lie-6.1}
	There is a positive integer $t$, and a faithful irreducible projective $\Fbb_{p} L$-module of dimension $m$, such that at least one of the following holds:
	\begin{enumerate}[\rm \quad  (a)]
		\item $s=q^{t}$ and $\dim(V)=n=m^{t}$;
		\item $L$ is of type $^{2}\A_{l}$, $^{2}\D_{l}$ or $^{2}\E_{6}$, $s=q^{t/2}$, $t$ is odd, and $n=m^{t}$;
		\item $L$ is of type $^{3}\D_{4}$ , $s=q^{t/3}$, $t$ is not divisible by $3$, and $n=m^{t}$;
		\item $L$ is of type $^{2}\B_{2}$, $^{2}\G_{2}$ or $^{2}\F_{4}$, $s=q^{t}$, and $n\geq m^{t}$.
	\end{enumerate}
	In particular, $s\leq q^{t}$ and  $n\geq R_{p}(L)^{t}$, where $R_{p}(L)$ is the minimal dimension of a faithful projective representation of $L$ in characteristic $p$. Moreover, $k$  divides $\lambda (q-1)|L : \N_{L}(U)|$, where $U$ is a Sylow $p$-subgroup of $L$, and
	\begin{align}\label{eq:Chp-1}
		q^{n}<\lambda q^{2t\cdot(N+l)},
	\end{align}
	where $l$ is
	the rank of the simple algebraic group over $\bar{\Fbb}_p$ corresponding to $L$ and $N$ is
	the number of positive roots in the corresponding root system.
\end{lemma}
\begin{proof}
	The first part is  \cite[Lemma~6.1]{a:Liebeck-98-Affine}. It follows from the proof of \cite[Lemma~6.2]{a:Liebeck-98-Affine} that the group $H/(H\cap \mathbb{F}_{q}^{\ast})$ has an orbit on $P_{1}(V)$ of length dividing $|L : \N_{L}(U)|$, where $U$ is a Sylow $p$-subgroup of $L$. So by Lemma~\ref{lem:six}(d), we conclude that $k$  divides $\lambda (q-1)|L : \N_{L}(U)|$.
	By \cite[Lemma~6.3]{a:Liebeck-98-Affine}, we have  $k<\lambda q^{t(N+l)}$, and so by applying Lemma \ref{lem:six}(c), we obtain \eqref{eq:Chp-1}, as desired.
\end{proof}

\begin{table}
	\centering
	\small
	\caption{Some parameters in Lemma~\ref{lem:class-chp}.}\label{tbl:class-chp}
	\begin{tabular}{lllllll}
		\hline\noalign{\smallskip}
		$L$ &
		$\u_{0}(L)$ &
		$N+l$  &
		$R_p(L)$ &
		Comments \\
		\noalign{\smallskip}\hline\noalign{\smallskip}
		$\A_{l}(s)$ &
		$l+1$ &
		$l(l+3)/2$ &
		$l+1$ &
		$l\geq 1$, $s=q^{t}$ \\
		$\A_{l}^{-}(s)$ &
		$l+1$ &
		$l(l+3)/2$ &
		$l+1$ &
		$l\geq 2$, $s=q^{t}$ or $q^{t/2}$\\
		$\B_{l}(s)$ &
		$l$ &
		$l^2+l$  &
		$2l+1$ &
		$l\geq 3$, $s=q^{t}$\\
		$\C_{l}(s)$ &
		$l$ &
		$l^2+l$ &
		$2l$ &
		$l\geq 2$, $s=q^{t}$\\
		$\D_{l}(s)$ &
		$l$ &
		$l^2$ &
		$2l$ &
		$l\geq 4$, $s=q^{t}$\\
		$\D_{l}^{-}(s)$ &
		$l$ &
		$l^2$ &
		$2l$ &
		$l\geq 4$, $s=q^{t}$ or $q^{t/2}$\\
		\noalign{\smallskip}\hline\noalign{\smallskip}&
	\end{tabular}
\end{table}

\begin{lemma}\label{lem:class-chp}
	The group $L$ cannot be a finite simple classical group of Lie type in characteristic $p$.
\end{lemma}
\begin{proof}
	Let $L=L(s)$ be a finite simple classical group of Lie type with $s$ a power of $p$.
	Note by Lemma~\ref{lem:affine-elem} that $\lambda$ divides $(q-1)\cdot |\Aut(L)|$. Then as $\lambda\neq p$, by Proposition~\ref{prop:lam-p}, we conclude that $\lambda< s^{\u_{0}(L)}$, where $\u_{0}(L)$  is listed as in the second column of Table~\ref{tbl:class-chp}.  By Lemma \ref{lem:lie-6.1}, we have $s\leq q^t$ and  $n\geq R_{p}(L)^{t}$, where the value of $R_{p}(L)$ is recorded in Table~\ref{tbl:class-chp} with noting that $s$ is odd, see  \cite[Table 5.4.C]{b:KL-90}.
	Therefore, \eqref{eq:Chp-1} implies that
	\begin{align}\label{eq:Chp-2}
		n< t \cdot [\u_{0}(L) + 2(N+l)],
	\end{align}
	and hence $R_{p}(L)^{t}<t \cdot [\u_{0}(L) + 2(N+l)]$ as $n\geq R_{p}(L)^{t}$.
	This forces $t\in\{1,2\}$, or $L$ and $n$ are as one the rows of Table~\ref{tbl:class-chp-1}. We first deal with the possibilities in Table~\ref{tbl:class-chp-1}. For each group $L$ in this table, by Lemma~\ref{lem:lie-6.1}, we know that $k/\lambda$ divides $(q-1)\cdot |L:\N_{L}(U)|$ where $|L:\N_{L}(U)|$ is recorded as in the second column of the same table, and since $k/\lambda$ also divides $v-1=q^n-1$, we conclude that $k/\lambda$ divides $f(q)$ given in the forth column of Table~\ref{tbl:class-chp-1}. As $\lambda\neq p$ is a prime divisor of $k$ which is a divisor of $(q-1)\cdot|\Aut(L)|$, for each $L$, we obtain an upper bound $\u_{\lambda}(L)$ of $\lambda$ as in the last column of Table~\ref{tbl:class-chp-1}, and since $\lambda v< k^{2}$, it follows that $q^n<\u_{\lambda}(L)\cdot f(q)^{2}$, which is clearly impossible for each possibility recorded in Table~\ref{tbl:class-chp-1}. Therefore $t=1$ or $t=2$.
	\begin{table}
		\centering
		\caption{The pairs $(L,n)$ in Lemma~\ref{lem:class-chp}.}\label{tbl:class-chp-1}
		\small
		\begin{tabular}{lllllll}
			\hline\noalign{\smallskip}
			$L$ &
			$|L:\N_{L}(U)|$ &
			$n$ &
			$f(q)$ &
			$\u_{\lambda}(L)$ & \\
			\noalign{\smallskip}\hline\noalign{\smallskip}
			$\PSL_{2}(q^3)$ &
			$q^{3}+1$ &
			$8$,  $9$, $10$  &
			$q^2-1$ &
			$2q^{2}$ & \\
			$\PSL_{2}(q^3)$ &
			$q^{3}+1$ &
			$11,12,13, 14$  &
			$(q-1)(q^3+1)$ &
			$2q^{2}$ & \\
			$\PSL_{2}(q^4)$ &
			$q^{4}+1$ &
			$16,17,18, 19$  &
			$(q-1)(q^4+1)$ &
			$2q^{4}$ & \\
			$\PSL_{3}(q^3)$ &
			$(q^{3}+1)(q^{6}+q^{3}+1)$ &
			$27$   &
			$(q-1)(q^{6}+q^{3}+1)$ &
			$2q^6$ & \\
			$\PSL_{3}(q^3)$ &
			$(q^{3}+1)(q^{6}+q^{3}+1)$ &
			$28,\ldots,35$   &
			$(q-1)(q^3+1)$&
			$2q^6$ & \\
			$\PSU_{3}(q^{3/2})$ &
			$q^{9/2}+1$ &
			$27$, \ldots, $38$  &
			$(q-1)(q^{9/2}+1)$ &
			$q^3$ & \\
			$\PSU_{4}(q^{3})$ &
			$(q^{3}+1)(q^{6}+1)^2$ &
			$65$  &
			$(q-1)(q^{3}+1)(q^{6}+1)^2$ &
			$q^6$ & \\
			
			\noalign{\smallskip}\hline\noalign{\smallskip}&
		\end{tabular}
	\end{table}

	Suppose first that  $l=1$ and $L=\PSL_{2}(q^{t})$ with $t=1,2$. Since $\lambda\neq p$, we have that $\lambda<q^{t}+1$. We know that $|L:\N_{L}(U)|=q^{t}+1$. So by Lemma~\ref{lem:lie-6.1}, as $n=m^{t}$ and $k/\lambda$ divides $(q-1)\cdot|L:\N_{L}(U)|$, we conclude that $q^{m^{t}}<(q-1)^{2}(q^{t}+1)^{3}$ implying that $m^{t}<3t+4$ with $t=1,2$. Then $n=2,3,4,5$ if $t=1$, and $n=4$ if $t=2$. In the latter case, $L=\PSL_{2}(q^{2})\cong \POm^{-}_{4}(q)$ which has been treated in Lemma~\ref{lem:C8}, and so we have no example in this case. In the former case where $t=1$ and $n=2,3,4,5$, we have that $k/\lambda$ divides $q^2-1$. If $n=3$ or $5$, then since $\gcd(q^{n}-1,q^2-1)=q-1$ and $\lambda<q+1$, it follows that $q^n<(q+1)(q-1)^{2}$, which is impossible. Thus $L=\PSL_{n}(q)$ with $n=2,4$. If $n=4$, then by \eqref{eq:k-lam}, we have that $\lambda=u^2 + (2u^2+ u)/(q^2 - 1)$ for some positive integer $u$, and so $q^2-1$ divides $2u^2+u$ and $\lambda>u^2$. Since $\lambda<q+1$, we must have $q^2-1<3(q+1)$ which is true when $q=3$, and so $\lambda<4$, which is not the case. If $n=2$, then $n<2(N+l)$, and so by \cite[Lemma~6.6]{a:Liebeck-98-Affine}, $V=M(\mu)$, where $\mu$ is as in \cite[Table IV]{a:Liebeck-98-Affine} by replacing $\mu$ with $\lambda$ in this reference, see also \cite[Theorem 2.2]{a:Liebeck-HA-rank3}. Hence the dimension of $V$ must be at least $3$, which is a contradiction. Therefore, $l\geq 2$.
	
	We now improve the bound given in \eqref{eq:Chp-2}. Note by Lemma \ref{lem:lie-6.1} that $k/\lambda$ divides  $(q-1)\cdot|L:\N_{L}(U)|$, where $U$ is a Sylow $p$-subgroup of $L$, and we know that $|L:\N_{L}(U)|\leq \prod_{i\in I}[(s^i-1)/(s-1)]$, where $|I|=l$ and $I$ consists of positive integers with sum $N+l$. Since $s-1\geq 2s/3$, it follows that $|L:\N_{L}(U)|\leq (3/2)^{l}\cdot s^{-l}\cdot\prod_{i\in I}(s^i-1)$, and since $s= q^{ct}$ with $c\in \{1,1/2\}$, we have that  $k/\lambda<(3/2)^{l}q^{ctN}(q-1)$. If $\lambda\leq q^{ct\u_0(L)}$, then Lemma \ref{lem:six}(c) implies that $q^n<(9/4)^{l}q^{ct[\u_0(L)+2N]}(q-1)^2$, and since $\log_{q}(9/4)\leq \log_{3}(9/4)<3/4$, we conclude that
	\begin{align}\label{eq:chp-bouund}
		n<ct[\u_0(L)+2N]+(3/4)l+2,
	\end{align}
	where $c\in \{1,1/2\}$ and $\u_{0}(L)=l+1$ if $L$ is of type $\A_{l}^{\e}$, otherwise, $\u_{0}(L)=l$. In what follows, we discuss two possibilities $t=1$ or $t=2$ considering the fact that $l\geq 2$ and noting that in our argument, \eqref{eq:chp-bouund} is a useful tool for restricting parameters.\smallskip
	
	\noindent \textbf{(1)} Let $t=1$. We consider two cases where $n\geq 2(N+l)$ or $n<2(N+l)$, and in what follows, we discuss each case separately.\smallskip
	
	\noindent \textbf{(1.1)} If $n\geq 2(N+l)$, then we conclude from \eqref{eq:chp-bouund} that $c=1$, that is to say, $s=q$ for all simple groups $L$. Moreover,
	\begin{align}\label{eq:chp-nl}
		2(N+l)\leq n< \u_{0}(L)+2N+(3/4)l+2,
	\end{align}
	where $\u_{0}(L)$ is recorded in Table~\ref{tbl:class-chp} and the value of $N$ can be obtained from the same table for each $L$. In particular, $\u_0(L)+2N+(3/4)l+2>2(N+l)$ implies that $5l<4\u_{0}(L)+8$, and hence $l\leq 11$ if $L$ is of type $\A_{l}^{\e}$, otherwise, $l\leq 7$.
	
	Suppose that $L=\A_{l}^{\e}(q)$ with $l\leq 11$. It follows from \eqref{eq:chp-nl} that $(l,n)$ is one the pairs below:
	\begin{center}
		$(2, 10)$,  $(2, 11)$,  $(2, 12)$,  
		$(3, 18)$,  $(3, 19)$,  $(3, 20)$,  \\
		$(4, 28)$,  $(4, 29)$,  $(4, 30)$,  
		$(5, 40)$,  $(5, 41)$,  $(6, 54)$, $(6, 55)$,  \\
		$(7, 70)$,  $(7, 71)$,  
		$(8, 88)$,  $(8, 89)$,  $(9, 108)$,  
		$(10, 130)$,  $(11, 154)$.
	\end{center}
	Note by Lemma~\ref{lem:lie-6.1} that $q^n<(q^{l+1}-1) (q-1)^2|L:\N_{L}(U)|^2$, where $q$ is odd and $U$ is a Sylow $p$-subgroup of $L$. If $L=\A_{l}(q)$, then $|L:\N_{L}(U)|=\prod_{i=2}^{l+1}[(q^{i}-1)/(q-1)]$, and so this inequality holds only when $(l,n)=(2,10), (3,18)$. If $(l,n)=(2,10)$, then $k/\lambda$ divides $\gcd(3,q-1)\cdot(q^2-1)$ and $\lambda\leq q^{2}+q+1$, and so $q^{10}<3^2(q^{2}+q+1)(q^2-1)^2$, which is impossible. If $(l,n)=(3,18)$, then $k/\lambda$ divides $(q+1)(q^2+q+1)(q^{4}-1)$ and $\lambda\leq q^{2}+q+1$, and so $q^{12}<(q+1)^2(q^{4}-1)^2(q^{2}+q+1)^{3}$, which is true only for $q=2$, but it is not the case. If $L=\A_{l}^{-}(q)$, then $|L:\N_{L}(U)|\leq \prod_{i=2}^{l+1}[(q^{i}-(-1)^{i})/(q-(-1)^{i})]$, and since $q^n<(q^{l+1}-(-1)^{l+1}) (q-1)^2|L:\N_{L}(U)|^2$, we obtain $(l,n)=(2,10), (3,18)$. If $(l,n)=(2,10)$, then $k/\lambda$ divides $(q^2-1)(q^{2}-q+1)$ and $\lambda\leq q^{2}-q+1$, and so $q^{10}<(q^2-1)^2(q^{2}-q+1)^{3}$, which is impossible. If $(l,n)=(3,18)$, then $k/\lambda$ divides $(q+1)(q^2-q+1)(q^{4}-1)$ and $\lambda\leq q^{2}+1$, and so $q^{12}<(q+1)^2(q^2+1)(q^{4}-1)^2(q^{2}-q+1)^{2}$, which is impossible.
	
	Suppose that $L=\B_{l}(q)$ or $\C_{l}(q)$ with $l\leq 7$. Here, $N=l^{2}$ and $\u_{0}(L)=l$. By \eqref{eq:chp-nl}, we obtain $(l,n)\in \{(2,12)$, $(2,13)$, $(3,24)$, $(3,25)$, $(4,40)$, $(4,41)$, $(5,60)$, $(6,84)$, $(7,112)\}$ with noting that $l\geq 3$ for type $\B_{l}$. For these pairs of $(l,n)$, the inequality $q^n<(q^{l}+1) (q-1)^2|L:\N_{L}(U)|^2$ is true only when $(l,n,q)=(3,24,3)$, or $L=\PSp_{4}(q)$ and $n=12$. In the former case, since $\lambda\geq 5$ is an odd prime divisor of $2|\POm_{7}(3)|=2|\PSp_{6}(3)|$, we have $\lambda=5$, $7$ or $13$, but for none of these possibilities $k^{2}-k=\lambda(3^{24}-1)$ has an integer solution, which is a contradiction. Let $L=\PSp_{4}(q)$ and $n=12$. Then $k/\lambda$ divides both $q^{12}-1$ and $(q-1)|L:\N_{L}(U)|=(q-1)(q+1)^2(q^2+1)$, and so $k/\lambda$ divide $\gcd(6,q+1)\cdot (q+1)(q^2+1)$. Thus $q^n<\lambda(k/\lambda)^{2}$ implies that $q=3$ or $5$. By Proposition~\ref{prop:lam-p}, since $\lambda\geq 5$ is a prime divisor of $2\log_p(q)|\PSp_{4}(q)|$, we obtain $\lambda=5$ if  $q=3$, and $\lambda=13$ if $q=5$, and both cases lead to no possible parameter $k$ satisfying Lemma~\ref{lem:six}(a).
	
	Suppose finally that $L=\D_{l}^{\e}(q)$ with $l\leq 7$. We know that $N=l^{2}-l$,  $\u_{0}(L)=l$ and $l\geq 4$, it follows from \eqref{eq:chp-nl} that $(l,n)\in \{(4,32)$, $(4,33)$, $(5,50)$, $(6,72)$, $(7,98)\}$, and so by the same argument as in the previous cases, we obtain $(l,n)=(4,32)$. If $L=\D_{4}(q)$, then $k/\lambda$ divides $\gcd(q^{32}-1,(q + 1)^4(q^2 + 1)^2(q^4 + q^2 + 1))$ which is a divisor of $ 2^{16}\cdot 3\cdot(q + 1)(q^2 + 1)$, and since $\lambda$ is a prime divisor of $(q-1)|\Aut(L)|$, it follows that $\lambda\leq q^{2}+q+1$, and hence $q^{32}<2^{32}\cdot 3^{2}\cdot (q-1)^{2}(q + 1)^2(q^2 + 1)^2(q^{2}+q+1)$, which is true for $q=3$. In this case, $\lambda$ is $5$, $7$ or $13$, however, for each such a $\lambda$, the quadratic equation $k^{2}-k=\lambda(3^{32}-1)$ has no integer solution, which is a contradiction.  If $L=\D_{4}^{-}(q)$, then $k/\lambda$ divides $\gcd(q^{32}-1,(q + 1)^2(q^2 + 1)(q^4 + 1)(q^4 + q^2 + 1))$ which is a divisor of $2^{3}\cdot 3\cdot (q + 1)(q^2 + 1)$, and since $\lambda\leq q^{4}+1$, Lemma~\ref{lem:six}(c) implies that $q^{32}<2^{6}\cdot3^{2}\cdot (q-1)^{2}(q + 1)^2(q^2 + 1)^2(q^{4}+1)^{3}$, which does not hold for $q\geq 3$.\smallskip
	
	\noindent \textbf{(1.2)} If $n<2(N+l)$, then the proof of \cite[Lemma~6.7]{a:Liebeck-98-Affine} implies that $V=M(\mu)$, where $\mu$ is as in \cite[Table IV]{a:Liebeck-98-Affine} by replacing $\lambda$ with $\mu$ in the reference. For convenience, we list $\mu$ and the dimension of $M(\mu)$ in Table \ref{tbl:chp-mu} when $p$ is odd. Moreover, $k/\lambda$ divides $(q-1)|L:P_{\mu}|$, where $P_{\mu}$ is a parabolic subgroup corresponding to the set of fundamental roots on which $\mu$ does not vanish. Therefore, for each $L$, the parameter $k/\lambda$ divides the value in the fourth column of Table~\ref{tbl:chp-mu}. Since  $n=\dim(M(\mu))$, $\lambda<q^{c\,\u_0(L)}$ with $c\in\{1,1/2\}$ and $q^{n}<\lambda(k/\lambda)^{2}$, we will make a  similar argument as in the proof of \cite[Lemma~6.7]{a:Liebeck-98-Affine}.\smallskip
	
	\begin{table}
		\centering
		\caption{The dimension of $M(\mu)$ in Lemma~\ref{lem:class-chp}.}\label{tbl:chp-mu}
		\resizebox{\textwidth}{!}{
			\begin{tabular}{lllll}
				\noalign{\smallskip}\hline\noalign{\smallskip}
				$L$ & $\mu$ & $\dim M(\mu)$ & $k/\lambda$ divides & Comments\\
				\noalign{\smallskip}\hline\noalign{\smallskip}
				$\A_{l}^{\e}(q^{c})$ &
				$\lambda_{2}$ &
				$l(l+1)/2$ &
				$(q^{l}-1)(q^{l+1}-1)/(q^2-1)$ &
				$l \geq 3$
				\\
				&
				$\lambda_{3}$ &
				$l(l^2-1)/6$ &
				$(q^{l-1}-1)(q^{l}-1)(q^{l+1}-1)/(q^2-1)(q^3-1)$ &
				$l \geq 6$
				\\
				&
				$\lambda_{3}$ &
				$20$ &
				$(q^5-\e)(q^3+1)(q^2+1)$ &
				$l =5$
				\\
				&
				$2\lambda_{1}$ &
				$(l+1)(l+2)/2$ &
				$q^{l+1}-1$ &
				$l\geq 2$
				\\
				&
				$(1+p^i)\lambda_{1}$ &
				$(l+1)^2$ &
				$q^{l+1}-1$ &
				$l\geq 2$, $i>0$
				\\
				&
				$\lambda_{1}+p^i\lambda_{l}$ &
				$(l+1)^2$ &
				$(q^{l+1}-1)(q^{l}-1)/(q-1)$ &
				$l\geq 2$, $i>0$
				\\
				&
				$\lambda_{1}+\lambda_{l}$ &
				$l^2+2l-\delta$ &
				$(q^{l+1}-\e^{l+1}1)(q^{l}-\e^{l}1)(q-1)/(q-\e1)$ &
				$l\geq 2$, $\delta\in\{0,1\}$
				\\
				$\B_{l}(q)$ &
				$\lambda_{2}$ &
				$l(2l+1)$ &
				$(q^{2l-2}-1)(q^{2l}-1)/(q^2-1)$ &
				$l\geq 3$
				\\
				&
				$\lambda_{l}$ &
				$2^l$ &
				$(q-1)(q+1)\cdots (q^{l}+1)$ &
				$l\geq 3$
				\\
				$\C_{l}(q)$ &
				$\lambda_{2}$ &
				$l(2l-1)-\delta$ &
				$(q^{2l-2}-1)(q^{2l}-1)/(q^2-1)$ &
				$l\geq 2$, $\delta\in\{1,2\}$
				\\
				&
				$2\lambda_{1}$ &
				$2l^2+2,\ldots, 2l^2+l$ &
				$q^{l+1}-1$ &
				$l\geq 3$
				\\
				&
				$\lambda_{3}$ &
				$14$ &
				$(q^3+1)(q^4-1)$ &
				$l= 3$
				\\
				$\D^{\e}_{l}(q^{c})$ &
				$\lambda_{2}$ &
				$l(2l-1)-\delta$ &
				$(q^{l-2}+\e1)(q^{l}-\e1)(q^{2l-2}-1)/(q^2-1)$ &
				$l\geq 4$, $\delta\in\{0,1,2\}$
				\\
				&
				$\lambda_{l-1}$, $\lambda_{l}$ &
				$2^{l-1}$ &
				$(q-1)(q+1)\cdots (q^{l-1}+1)$ &
				$l\geq 4$
				\\
				\noalign{\smallskip}\hline\noalign{\smallskip}
			\end{tabular}
		}
	\end{table}

	\noindent \textbf{(i)} Let $\mu=\lambda_{2}$. Then $L$ is of type $\A_{l}^{\e}$ with $l\geq 3$, $\B_{l}$ with $l\geq 3$, $\C_{l}$ with $l\geq 2$ or $\D_{l}^{\e}$ with $l\geq 4$. \smallskip
	
	Suppose that $L=\A_{l}^{\e}(q^{c})$  with $l\geq 3$.  Here, the case where $c=1$ in the unitary case over $\Fbb_{s}$ occurs only when $l=3$, but in this case, $n=6$ and $L=\POm_{6}^{-}(q)$  which has already been treated in Lemma~\ref{lem:C8}. Since $k/\lambda$ divides $(q^{l}-1)(q^{l+1}-1)/(q^2-1)$ and $\lambda<q^{l+1}$, it follows from $q^{n}<\lambda(k/\lambda)^{2}$ that $l(l+1)/2<5(l+1)-5$ which is true for $l\leq 9$. But a direct computation shows that  $q^{l(l+1)/2}(q^{2}-1)^{2}<q^{l+1}(q^{l}-1)^{2}(q^{l+1}-1)^{2}$ has no possible solution for $l=9$.  Since  $\A_{3}^{\e}(q)\cong \POm_{6}^{\e}(q)$, it follows that $l\neq 3$ when  $n=6$ as it has been already discussed in Lemma~\ref{lem:C8}. Therefore, $l=4, \ldots 8$.
	
	If $l=4$, then $n=10$ and $k/\lambda$ divides $\gcd( (q^5-1)(q^2+1),q^{10}-1 )$ which is a divisor of $2(q^5-1)$. Then \eqref{eq:k-lam} implies that $4\lambda=u^{2}+2u(u+1)/(q^5-1)$, and so $u^{2}<4\lambda$ and $q^{5}-1$ divides $2u(u+1)$ for some positive integer $u$. Since $\lambda$ is a prime divisor of $(q-1)|\Aut(L)|$, it follows that $\lambda\leq q^{4}+q^{3}+q^{2}+q+1=(q^{5}-1)/(q-1)$. Since also $q^{5}-1\leq 2u(u+1)$, we have that $q^{5}-1\leq 4u^{2}<16\lambda\leq 16(q^{5}-1)/(q-1)$ implying that $q<17$, but for each such $q$, we find $\lambda$ as a prime divisor of $(q-1)|\Aut(L)|$, and none of these gives a possible parameter $k$ satisfying Lemma~\ref{lem:six}(a).
	
	If $l=5$, then $n=15$ and $k/\lambda$ divides $\gcd( (q^{4}+q^2+1)(q^5-1),q^{15}-1 )$ which is a divisor of $(q^2+q+1)(q^5-1)$. By \eqref{eq:k-lam}, there exists positive integer $u$ such that
	\begin{align}\label{eq:k-lam-chp-1}
		k=uf_{1}(q) + 1 \text{ and }
		\lambda=(q-2)u^{2} +\frac{u^{2}f_{2}(q)+ u}{(q^{5}-1)(q^{2}+q+1)},
	\end{align}
	where $f_{1}(q)=q^8 - q^7 + q^5 - q^4 + q^3 - q + 1$ and $f_{2}(q)=q^{6}+3 q^{5}-q^{4}+2 q^{3}-q^{2}-2 q-1$.
	Note that $\lambda$ is a divisor of $k$ and it is coprime to $u$. Moreover, $(q^{5}-1)(q^{2}+q+1)$ divides $u^2f_{2}(q)+ u$.  Note also that $\lambda$ is a prime divisor of $(q-1)|\Aut(L)|$. Then $\lambda$ is a prime divisor of $(q-1)/2$, $(q+1)/2$, $q^{2}+1$, $q^{2}-q+1$, $q^{2}+q+1$, $q^4 + q^3 + q^2 + q + 1$ or $q^4 - q^3 + q^2 - q + 1$. Since $v<\lambda(k/\lambda)^{2}$, we easily observe that $\lambda$ is not a divisor of $(q-1)/2$, and if $\lambda$ divides $(q+1)/2$, then $q=3$ for which $\lambda$ must be $2$, which is not the case.
	If $\lambda$ divides $g(q):= q^4 + q^3 + q^2 + q + 1=(q^5-1)/(q-1)$, which is a divisor of $u^2f_{2}(q)+ u=(3q^{3}+3)u^{2}g(q)+(q^{2}+2q-4)u^{2}+u$, the parameter $\lambda$ divides $(q^{2}+2q-4)u+1$. Since also $\lambda$ divides both $g(q)$ and $k=uf_{1}(q)+1= (q^4 - 2q^3 + q^2 + q - 2)g(q)u+(3q^3 + 3)u+1$, we have that $\lambda$ divides $(3q^3 + 3)u+1$, and as $\lambda$ is coprime to $u$ and it is a divisor of $(q^{2}+2q-4)u+1$, we conclude that $\lambda$ divides $3q^3-q^{2}-2q+7$. Thus $\lambda$ divides $\gcd(3q^3-q^{2}-2q+7, q^4 + q^3 + q^2 + q + 1)$, which is a divisor of $8455332=2^2\cdot 3\cdot 61\cdot 11551$, and hence $\lambda=61$ or $11551$. The fact that $q^{15}<\lambda(k/\lambda)^{2}$ implies that $q\leq 11551$, and for each such $q$, we obtain no parameter $k$ satisfying Lemma~\ref{lem:six}(a).
	If $\lambda$ divides $g(q):= q^4 - q^3 + q^2 - q + 1$ in the case where $L=\PSU_{6}(q^{1/2})$, then $k$ divides $(q-1)|\Aut(L)|$, and so $k/\lambda$ divides $\gcd((q^{5}-1)(q^{2}+q+1),(q-1)|\Aut(L)|)$. Thus $k/\lambda$ divides $12\log_{p}(q)\cdot 2\cdot 5^{3}(q-1)(q^2+q+1)$, and hence $q^{15}<(24\cdot 5^{3})^{2}\log_{p}^{2}(q)(q-1)^{2}(q^2+q+1)^{2}(q^4 - q^3 + q^2 - q + 1)$, which is true for $q<46 $, but these values of $q$ give no possible parameter set.
	If $\lambda$ divides $q^{2}+q+1$, then since $\lambda$ divides $k=uf_{1}(q) + 1=(q^{2}+q+1)(q^6 - 2q^5 + q^4 + 2q^3 - 4q^2 + 3q + 1)u-5qu+1$, the parameter $\lambda$ is a divisor of $5qu-1$, and since $(q-2)u^{2}< \lambda$ by \eqref{eq:k-lam-chp-1}, it follows that $u<5q/(q-2)$, and so $q=3$ or $u<10$ for $q\geq 3$. The case where $q=3$ is easily ruled out by Lemma~\ref{lem:six}(a). Thus $u<10$, and hence for each such $u$, $\lambda$ divides $\gcd(q^{2}+q+1, 5uq-1)$, which implies that $\lambda<2071$, but none of these possibilities leads to a possible $k$ satisfying Lemma~\ref{lem:six}(a). Similarly, if $\lambda$ divides $q^{2}-q+1$, then $\lambda$ divides $qu-1$, and so  $(q-2)u^{2}<qu$ implies that $u=1$, and so  $\lambda$ divides $\gcd(q^{2}-q+1,q-1)=1$, which is a contradiction. If $\lambda$ divides $q^{2}+1$, then $\lambda$ must divide $k=(q^6 - q^5 - q^4 + 2q^3 - q)(q^2+1)u+u+1$, and so $\lambda$ divides $u+1$, but then $(q-2)u^{2}<u+1$ for $q\geq 3$, which is a contradiction.
	
	If $l=6$, then $n=21$ and $k/\lambda$ divides $\gcd( (q^{4}+q^2+1)(q^7-1),q^{21}-1 )$ which divides $(q^2+q+1)(q^7-1)$. We first observe that $q\neq 3$ by Lemma~\ref{lem:six}(a) and direct computation. Thus we assume that $q\geq 5$. If $\lambda$ is a divisor of $q-1$, $q+1$, $q^{2}-q+1$, $q^{2}+q+1$ or $q^{2}+1$, we easily see that $q^{21}<\lambda(k/\lambda)^{2}$ fails for $q\geq 5$. For the remaining cases where $\lambda$ divides of $(q^5-\delta 1)/(q-\delta1)$ or $(q^7-\delta1)/(q-\delta1)$ with $\delta=\pm$, it follows from \eqref{eq:k-lam} that
	\begin{align}\label{eq:k-lam-chp-2}
		k=uf_{1}(q) + 1 \text{ and }
		\lambda=(q^3 - 2q^2 + q + 2)u^{2} -\frac{u^{2}f_{2}(q)+ u}{f(q)},
	\end{align}
	where $f(q)=(q^{7}-1)(q^{2}+q+1)$, $f_{1}(q) = q^{12} - q^{11} + q^9 - q^8 + q^6 - q^4 + q^3 - q + 1$ and $f_{2}(q)=4q^8 + 2q^7 - q^6 - q^5 + 2q^4 - q^3 - q^2 - 2q - 3$, for some positive integer $u$. Note that $f_{2}(q)>0$ and $f_{2}(q)<f(q)$ for $q\geq 5$. If $\lambda$ divides $g(q):=q^{4}-q^{3}+q^{2}-q+1$, then since $\lambda$ divides $k=uf_{1}(q) + 1=(q^8 - q^6 + q^5 - 2q^3 + q^2 + 2q - 2)g(q)u+(3q^2 - 5q + 3)u+1$, we conclude that $\lambda$ divides $(3q^2 - 5q + 3)u+1$. On the other hand by \eqref{eq:k-lam-chp-2}, as $f_{2}(q)>0$ and $f_{2}(q)<f(q)$ for $q\geq 5$, it follows that $(u^{2}f_{2}(q)+ u)/f(q)<u(u+1)$, and so $(q^3 - 2q^2 + q + 2)u^{2}-u(u+1)<\lambda$. Since $\lambda\leq (3q^2 - 5q + 3)u+1$, we conclude  that $(q^3 - 2q^2 + q)u<3q^2 - 5q + 4$, which is not true for $q\geq 5$. If $\lambda$ divides $q^{4}+q^{3}+q^{2}+q+1$, a same argument as above leads to  $(q^3 - 2q^2 + q)u<q^2 - q + 1$, which is also impossible.
	If $\lambda$ divides $g(q):=(q^{7}-1)/(q-1)$, then it is a divisor of $u^2f_{2}(q)+ u=(4q^2 - 2q - 3)u^{2}g(q)+(3q^4 + 3q)u^{2}+u$, and since $\lambda$ is coprime to $u$, the parameter $\lambda$ divides $(3q^4 + 3q)u+1$.
	Since also $\lambda$ divides both $g(q)$ and $k=uf_{1}(q)+1=(q^6 - 2q^5 + q^4 + q^3 - 2q^2 + q + 1)ug(q)-(3q^3 + 3)u+1$, we have that $\lambda$ divides $(3q^3 + 3)u-1$. Therefore, $\lambda$ divides both $(3q^3 + 3)u-1$ and $(3q^3 + 3)u+1$, which requires $\lambda=2$, which is a contradiction.
	If $\lambda$ divides $g(q):=(q^{7}+1)/(q+1)$, then  $L=\PSU_{7}(q^{1/2})$, and so $k$ divides $(q-1)|\Aut(L)|$, thus $k/\lambda$ divides $\gcd((q^{7}-1)(q^{2}+q+1),(q-1)|\Aut(L)|)$. Therefore, $k/\lambda$ divides $12\log_{p}(q)\cdot 7^{6}(q-1)(q^2+q+1)$, and
	hence $q^{21}<(24\cdot 7^{6})^{2}\log_{p}^{2}(q)\cdot (q-1)^{2}(q^2+q+1)^{2}(q^6 - q^5 + q^4 - q^{3} +q^{2}-q + 1)$, which is true for $q < 14$, however, we obtain no possible parameter sets for these values of $q$.
	
	If $l=7$, then $n=28$ and $k/\lambda$ divides $\gcd( (q^2+1)(q^{4}+1)(q^7-1),q^{28}-1 )$ which divides $2^3(q^2+1)(q^7-1)$, and since $\lambda\leq (q^7-1)/(q-1)$, we have that $q^{28}(q-1)<2^{6}(q^2+1)^2(q^7-1)^{3}$, which is true only for $q=3$, but this leads to no possible parameters satisfying Lemma~\ref{lem:six}(a).
	
	If $l=8$, then $n=36$ and $k/\lambda$ divides $\gcd( (q^2+1)(q^{4}+1)(q^9-1),q^{36}-1 )$ dividing $2^{4}(q^{2}+1)(q^{9}-1)$. Here, we again have  $\lambda\leq (q^7-1)/(q-1)$, and so $q^{28}(q-1)<2^{8}(q^2+1)^2(q^9-1)^{2}(q^7-1)$, which is impossible for $q\geq 3$.
	
	Suppose that $L=\B_{l}(q)$ with $l\geq 3$. Then $n=l(2l+1)$ and $k/\lambda$ divides $(q^{2l-2}-1)(q^{2l}-1)/(q^2-1)$, and so $q^{l(2l+1)}(q^{2}-1)^{2}<q^{l}(q^{2l-2}-1)^{2}(q^{2l}-1)^{2}$ which implies that $l(2l+1)+3<l+2(2l-2)+4l$, and so $l=2$, which is not the case.
	
	Suppose that $L=\C_{l}(q)$ with $l\geq 2$. Then $n=l(2l-1)-\delta$ with $\delta=1,2$, and $k/\lambda$ divides $(q^{2l-2}-1)(q^{2l}-1)/(q^2-1)$, and so $q^{l(2l-1)-2}(q^{2}-1)^{2}<q^{l}(q^{2l-2}-1)^{2}(q^{2l}-1)^{2}$ which implies that $[l(2l-1)-2]+3<l+2(2l-2)+4l$, and so $l=2,3,4$. If $l=2$, then $n=5$ and $L=\PSp_{4}(q)\cong \POm_{5}(q)$, but this case has been treated in Lemma~\ref{lem:C8}. If $l=3$, then $L=\PSp_{6}(q)$ and $n=13$ or $14$, and so $k/\lambda$ divides $\gcd((q^2+1)(q^6-1), q^n-1)$ which is a divisor of $2(q^{2}-1)$, and since $\lambda<q^l=q^{3}$, we have that  $q^{10}<2^2(q^2-1)^2$, but this inequality has no solution for $q\geq 3$. If $l=4$, then $L=\PSp_{8}(q)$ and $n=26$ or $27$, and so $k/\lambda$ divides $\gcd((q^2+1)(q^4+1)(q^6-1), q^n-1)$ which is a divisor of $2^2(q^{2}-1)(q^2+q+1)$, and hence   $q^{22}<2^4(q^2-1)^2(q^2+q+1)^{2}$, which has also no  positive integer solutions.
	
	Suppose that $L=\D_{l}^{\e}(q^{c})$ with $l\geq 4$, $c=1$ or $1/2$ when $\e=-$, and $c=1$ if $\e=+$. Then $n=l(2l-1)-\delta$ with $\delta=0,1,2$, and $k/\lambda$ divides $(q^{l-2}+\e1)(q^{l}-\e1)(q^{2l-2}-1)/(q^2-1)$. Since $\lambda<q^{l}$, we have that $l(2l-1)+1<l + 2[(l - 2) + 1] + 2(l + 1) + 2(2l - 2)$, and so $l=2$ or $3$, but this contradicts the assumption $l\geq 4$.\smallskip

	\noindent \textbf{(ii)} Let $\mu=\lambda_{3}$. Then $L$ is of type $\A_{l}^{\e}$ with $l\geq 5$, $\B_{3}$, $\C_{3}$, or $\D_{4}^{\e}$. \smallskip
	
	Suppose that $L=\A_{l}^{\e}(q^{c})$ with $l\geq 5$.
	If $l=5$, then $n=20$ and $L=\A_{5}^{\e}(q)$, and so $k/\lambda$ divides $\gcd(q^{20}-1, (q^5-\e)(q^3+1)(q^2+1))$.
	If $L=\PSL_{6}(q)$, then $k/\lambda$ divides $(q+1)(q^{2}+1)(q^{5}-1)$. If $\lambda$ is a divisor of $q\pm 1$, $q^{2}\pm q+1$ or $q^2+1$, then   $q^{20}(q-1)<(q^{2}+q+1)(q+1)^{2}(q^{2}+1)^{2}(q^{5}-1)^{2}$, and so $q=2$, which is not the case. If $\lambda$ is a divisor of $q^{4}+q^{3}+q^{2}+q+1$, then 
	by \eqref{eq:k-lam}, we have that 
	\begin{align}\label{eq:k-lam-chp-3}
		k=uf_{1}(q) + 1 \text{ and }
		\lambda=(q^4 - 2q^3 + q^2 + 2)u^{2} -\frac{u^{2}f_{2}(q)- u}{f(q)},
	\end{align}
	where $f(q)=(q+1)(q^{2}+1)(q^{5}-1)$, $f_{1}(q) = q^{12} - q^{11} + q^{8} - q^{6} + q^{4} - q + 1$ and $f_{2}(q)=2q^7 + 4q^6 + 2q^5 - q^4 - q^3 - 3q^2 - q - 3$, for some positive integer $u$. 
	Note that $f_{2}(q)>0$ and $f_{2}(q)<f(q)$ for $q\geq 3$. If $\lambda$ divides $g(q):=q^{4}+q^{3}+q^{2}+q+1$, then since $\lambda$ divides $k=uf_{1}(q) + 1=(q^8 - 2q^7 + q^6 + q^4 - 3q^2 + 2q + 1)g(q)u-4qu+1$, we conclude that $\lambda$ divides $4qu-1$. On the other hand by \eqref{eq:k-lam-chp-3}, as $f_{2}(q)>0$ and $f_{2}(q)<f(q)$ for $q\geq 5$, it follows that $(u^{2}f_{2}(q)- u)/f(q)<u(u+1)$, then  $(q^4 - 2q^3 + q^2 + 2)u^{2}-(u+1)u<\lambda$, and so $(q^4 - 2q^3 + q^2 + 1)u^2<4qu+u-1$ as  $\lambda\leq 4qu-1$. Thus $(q^4 - 2q^3 + q^2 + 1)u<4q+1$, which is false for $q\geq 3$. If $L=\PSU_{6}(q)$, then $k/\lambda$ divides $2^{2}(q^{2}+1)(q^{5}+1)$, and since $\lambda\leq (q^{5}+1)/(q+1)$, it follows that $q^{20}(q+1)<2^{4}(q^{2}+1)^{2}(q^{5}+1)^{3}$ implying that $q=3$, but again, we obtain no parameter sets satisfying Lemma~\ref{lem:six}(a).
	If $l\geq 6$, then  $n=l(l^{2}-1)/6$ and $k/\lambda$ divides $(q^{l-1}-1)(q^{l}-1)(q^{l+1}-1)/(q^2-1)(q^3-1)$, and since $\lambda<q^{l+1}$, it follows from $q^{n}<\lambda(k/\lambda)^{2}$ that
	$q^{l(l^2 - 1)/6}(q^2 - 1)^2(q^3 - 1)^2 < q^{l+1}(q^{l - 1} - 1)^2(q^l - 1)^2(q^{l + 1} - 1)^2$, and so $4+l(l^2 - 1)/6 < (l + 1) + 2(l - 1)) + 2l + 2(l + 1)$ implying that $l=6$, and so
	$q^{35}(q^2 - 1)^2(q^3 - 1)^2 < q^{7}(q^{5} - 1)^2(q^6 - 1)^2(q^{7} - 1)^2$, which has no integer solution for $q\geq 3$.
	
	Suppose that $L=\B_{3}(q)$ or $\D_{4}^{\e}(q^{c})$ with $c=1$ or $1/2$. Then $n=8$ and $k/\lambda$ divides $\gcd(q^{8}-1,(q^{3}+1)(q^{4}-1))$ which is a divisor of $2(q^{4}-1)$, and so \eqref{eq:k-lam} implies that
	\begin{align}\label{eq:k-lam-chp-4}
		2k=(q^{4}+1)u+2 \text{ and }
		4\lambda= u^{2}+\frac{2u(u+1)}{q^4-1},
	\end{align}
	for some positive integer $u$. This forces $u^{2}<4\lambda$ and $q^{4}-1\leq 2u(u+1)$, and since $\lambda\leq q^{2}+q+1$, it follows that
	$q^{4}-1\leq 2u(u+1)\leq 4u^{2}\leq 16\lambda\leq 16(q^{2}+q+1)$, and so $q^{4}-1\leq 16(q^{2}+q+1)$ implying that $q=3$, however, this gives no possible parameter set satisfying Lemma~\ref{lem:six}(a).
	
	Suppose that $L=\C_{3}(q)$.
	Then $n=14$ and $k/\lambda$ divides $\gcd(q^{14}-1,(q^{3}+1)(q^{4}-1))$ which is a divisor of $2^{3}\cdot7(q^{2}-1)$, and since $\lambda\leq q^{2}+q+1$, we must have $q^{14}<2^{6}\cdot 7^{2}(q^{2}+q+1)(q^{2}-1)^{2}$ implying that $q=2$, which is a contradiction.\smallskip
	
	\noindent \textbf{(iii)} Let $\mu=2\lambda_{1}$. Then $L$ is of type
	$\A_{l}^{\e}$ with $l\geq 2$, or
	$\C_{l}$ with $l\geq 2$.\smallskip
	
	Suppose that $L=\A_{l}^{\e}(q^{c})$ with $l\geq 2$. Then $n=(l+1)(l+2)/2$ and
	$k/\lambda$ divides
	$q^{l+1}-1$, and since $\lambda<q^{l+1}$, it follows from $q^{n}<\lambda(k/\lambda)^{2}$ that $l(l + 1)(l + 2)/2 < l + 1 + 2(l + 1)$ forcing $l=1$, which is a contradiction.
	
	Suppose that $L=\C_{l}(q)$ with $l\geq 2$. Then $n\geq 2l^{2}+2$ and
	$k/\lambda$ divides
	$q^{l+1}-1$, and since $\lambda<q^{l}$, it follows from $q^{n}<\lambda(k/\lambda)^{2}$ that $ 2l^{2}+2 < l  + 2(l + 1)$ which is true only for $l=1$, which is a contradiction.
	
	\noindent \textbf{(iv)} Let $\mu=(1+p^{i})\lambda_{1}$, $\lambda_{1}+p^{i}\lambda_{l}$ with $i>0$, or $\lambda_{1}+\lambda_{l}$. Then
	$L$ is of type
	$\A_{l}^{\e}$ with $l\geq 2$.
	
	In this case, $n\geq l^{2}+l$ and $\lambda<q^{l+1}$, and so by considering all cases, the inequality $q^{n}<\lambda(k/\lambda)^{2}$ forces $l=2,3$ if $\mu=\lambda_{1}+p^{i}\lambda_{l}$ or $\mu=\lambda_{1}+\lambda_{l}$.
	
	Let $l=2$.
	If  $\mu=(1+p^{i})\lambda_{2}$, then $n=9$ and $k/\lambda$ divides $\gcd(q^9-1,(q+1)(q^{3}-1))$, and so $k/\lambda$ divides $q^{3}-1$. Thus, \eqref{eq:k-lam} implies that $\lambda=(q^3 + 2)u^2 + (3u^2 + u)/(q^3 - 1)$, and so $(q^3 + 2)u^2<\lambda$. Since $\lambda$ divides $(q-1)|\Aut(L)|$, we have that $\lambda\leq q^{2}+q+1$, and hence $(q^3 + 2)u^2<q^{2}+q+1$, which is impossible.
	If $\mu=\lambda_{1}+\lambda_{2}$, then $n=7$ or $8$ and $k/\lambda$ divides $\gcd(q^n-1,(q^2-1)(q^{3}-\e1)/(q-\e1))$ which is a divisor of $7(q-1)$ if $n=7$, or $4(q^{2}-1)$ if $n=8$. Thus $q^{n}<(q^{2}+q+1)(k/\lambda)^{2}$ forces $q=3$, for which we find no possible parameter sets satisfying Lemma~\ref{lem:six}(a).
	
	Let $l=3$.
	If  $\mu=(1+p^{i})\lambda_{3}$, then $n=16$ and $k/\lambda$ divides $(q^{2}+q+1)(q^4-1)$, which fails to satisfy $q^{n}<\lambda(k/\lambda)^{2}$ for $q\geq 3$.
	If $\mu=\lambda_{1}+\lambda_{3}$, then $n=14$ or $15$. If $n=15$, then $k/\lambda$ divides $5(q^{3}-1)$ or $2\cdot3\cdot 5(q-1)$, respectively, for $\e=+$ or $\e=-$. Thus $k/\lambda\leq 5(q^{3}-1)$, and since $\lambda\leq q^{2}+q+1$, we must have $q^{15}<25(q^{2}+q+1)(q^{3}-1)^{2}$, which is not true for $q\geq 3$.
	If $n=14$, then $k/\lambda$ divides $2^{2}\cdot7(q^{2}-1)$ or $7(q^{2}-1)$, when $\e=+$ or $-$, respectively. Thus $k/\lambda\leq 2^{2}\cdot7(q^{2}-1)$, and so $q^{14}<2^{4}\cdot7^{2}(q^{2}+q+1)(q^{2}-1)^{2}$ which forces $q=2$, which is a contradiction.    \smallskip
	
	\noindent \textbf{(v)} Let $\mu=\lambda_{l}$ or $\lambda_{l-1}$. Then $L$ is of type $\D_{l}^{\e}$ with $l\geq 4$, or  it is of type
	$\B_{l}$ with $l\geq 4$ in the former case. Note that the cases where $\mu=\lambda_{3}$ and $L$ is of type $\B_{3}$ or $\D_{4}$ has been treated in (ii).\smallskip
	
	Suppose that $L=\B_{l}(q)$ with $l\geq 4$. Then $n=2^{l}$ and $k/\lambda$ divides $(q-1)(q+1)\cdots (q^{l}+1)$, and so $q^{n}<\lambda(k/\lambda)^{2}$ implies that $2^{l}<l(l+1)+3l+2$, which forces $l=4$ or $5$.
	If $l=4$, then $k/\lambda$ divides $\gcd(q^{16}-1,(q^3+1)(q^8-1))$ which is a divisor of $2(q^8-1)$, and so \eqref{eq:k-lam} implies that $4\lambda=u^2 + 2u(u + 1)/(q^8 - 1)$, and so $q^{8}-1$ divides $2u(u + 1)$, and since $\lambda\leq q^{4}+1$, we conclude that $q^{8}-1\leq 2u(u + 1)\leq 4u^{2}<16\lambda\leq 16(q^{4}+1)$, that is to say, $q^{8}-1< 16(q^{4}+1)$, which fails if $q\geq 3$.
	If $l=5$, then $k/\lambda$ divides $\gcd(q^{32}-1,(q^3+1)(q^{5}+1)(q^8-1))$ which is a divisor of $2^{4}(q^8-1)$, and since $\lambda\leq q^4 + q^3 + q^2 + q + 1$, we must have $q^{32}(q-1)<2^{8}(q^{5}-1)(q^8-1)^{2}$, which is not true for $q\geq 3$.
	
	Suppose that $L=\D_{l}^{\e}(q)$ with $l\geq 4$. Then $n=2^{l-1}$ and $k/\lambda$ divides $(q-1)(q+1)\cdots (q^{l-1}+1)$, and so $q^{n}<\lambda(k/\lambda)^{2}$ implies that $2^{l-1}<l(l-1)+3l$, and so $l=4,5$ or $6$.
	If $l=4$, then $k/\lambda$ divides $\gcd(q^{8}-1,(q^3+1)(q^4-1))$ which is a divisor of $2(q^4-1)$, and so \eqref{eq:k-lam} implies that
	$4\lambda=u^2 + 2u(u + 1)/(q^4 - 1)$, then $q^{4}-1$ divides $2u(u + 1)$. Since $2k=(q^{4}+1)u+1$ and $\lambda$ is a divisor of $k$, we conclude that $\lambda$ cannot divide $q_{4}+1$, and so $\lambda\leq q^{2}+q+1$. Thus, $q^{4}-1\leq 2u(u + 1)\leq 4u^{2}<16\lambda\leq 16(q^{2}+q+1)$, that is to say, $q^{4}-1< 16(q^{2}+q+1)$, forcing $q=3$, but this gives no possible parameter set.
	If $l=5$, then $k/\lambda$ divides $\gcd(q^{16}-1,(q^3+1)(q^8-1))$ which is a divisor of $2(q^8-1)$, and since $4\lambda=u^2 + 2u(u + 1)/(q^8 - 1)$ by \eqref{eq:k-lam}, we conclude that  $q^{8}-1$ divides $2u(u + 1)$, and since $\lambda\leq (q^{5}-1)/(q-1)$, we conclude that $q^{8}-1\leq 2u(u + 1)\leq 4u^{2}<16\lambda\leq 16(q^{5}-1)/(q-1)$, that is to say, $(q^{8}-1)< 16(q^4 + q^3 + q^2 + q + 1)$, which fails if $q\geq 3$.
	If $l=6$, then $k/\lambda$ divides $\gcd(q^{32}-1,(q^3+1)(q^{5}+1)(q^8-1))$ which is a divisor of $2^{2}(q^8-1)$, and since $\lambda\leq q^4 + q^3 + q^2 + q + 1$, we must have $q^{32}(q-1)<2^{4}(q^{5}-1)(q^8-1)^{2}$, which is not true for $q\geq 3$.
	\smallskip
	
	\noindent \textbf{(2)} Let $t=2$. If $L=\B_l(q^2)$, then $N=l^2$ and $\u_0(L)=l$, and so by \eqref{eq:chp-bouund}, we have that $n<2(l+2l^2)+(3/4)l+2$, and since $n\geq R_p(L)^2=(2l+1)^2$, we must have $(2l+1)^2<2(l+2l^2)+(3/4)l+2$, or equivalently, $5l < 4$, which is a contradiction. By the same argument, $L$ cannot be of type  $\D_{l}^{\e}$. Therefore, $L=\A_l^\e(q^2)$ or $\C_l(q^2)$.
	
	Suppose that $L=\A_{l}^{\e}(q^{2})$. Then by \eqref{eq:chp-bouund} we have that $n=m^{2}<2(l+1)^{2}+(3/4)l+2$ implying that $l=2$ or $m<(l+1)^2/2$. If $l=2$, then $L=\A_{2}^{\e}(q^2)$ and $m\geq l+1=3$, and since $n=m^{2}<2(l+4)^{2}$, we conclude that $m=3$ or $4$, or equivalently, $n=9$ or $16$. We know that $k/\lambda$ divides both $q^{n}-1$ and $(q-1)\cdot|L:\N_{L}(U)|=(q-1)(q^4-1)(q^{6}-\e1)/[(q^{2}-1)(q^{2}-\e1)]$. Then $k/\lambda$ divides $q^{3}-1$ or $3(q-1)(q^{2}+1)$ if $\e=+$ and $n=9$ or $16$, respectively, and it is a divisor of $q-1$ or $(q-1)(q^{2}+1)$ if $\e=-$ and $n=9$ or $16$, respectively. However, this violates Lemma~\ref{lem:six}(c). Therefore, $m<(l+1)^2/2$, and so by \cite[Theorem 1.1]{a:Liebeck1985}, we have $m=l+1$, and hence \cite[Proposition~5.4.6]{b:KL-90} implies that $L=\PSL_{l+1}(q^{2})$ and
	$V\otimes \Fbb_{q^{2}}=W\otimes W^{(q)}$,
	where $W=V_{l+1}(q^{2})$, the usual (projective) module for $L$. In this case, following the proof of \cite[Lemma~2.4]{a:Liebeck-HA-rank3}, $G_{0}$ has an orbit on the vectors of the form $v\otimes v$ of length $(q-1)(q^{2l+2}-1)/(q^2-1)$. Thus $k/\lambda$ divides $(q-1)(q^{2l+2}-1)/(q^2-1)$, and since $\lambda< q^{2(l+1)}$, Lemma~\ref{lem:six}(c) implies that $q^{(l+1)^{2}}=q^{n}<\lambda (k/\lambda)^{2}<q^{2(l+1)+4(l+1)}$, and so $(l+1)^{2}<6(l+1)$, which true when $l\leq 4$, and hence $l=2,3$ or $4$. If $l=2$, then $L=\PSL_{3}(q^{2})$ and $n=9$, and so $k/\lambda$ divides $\gcd(q^{6}-1,q^{9}-1)$, thus $k/\lambda\leq q^{3}-1$, and since $\lambda\neq p$ is a prime divisor of $(q-1)\cdot |\Aut(L)|$, it follows that $\lambda<q^2+q+1$, and hence $q^{9}<(q^{2}+q+1)(q^{3}-1)^{2}$, which contradicts Lemma~\ref{lem:six}(c). Similarly, if $l=4$, then $n=25$, $\lambda<q^{5}-1$ and $k/\lambda$ divides $\gcd(q^{25}-1,q^{10}-1)=q^{5}-1$, and so Lemma~\ref{lem:six}(c) implies that $q^{25}<(q^{5}-1)^{3}$, which is a contradiction. If $l=3$, then $n=16$, and so $k/\lambda$ divides $\gcd(q^{16}-1,q^{8}-1)=q^{8}-1$. Then there exists $u$ such that $ku=\lambda f(q)$, where $f(q)=q^{8}-1$, and so \eqref{eq:k-lam} implies that $\lambda=u^{2}+[(2u^{2}+u)/(q^{8}-1)]$, and hence $u^{2}<\lambda$ and $q^{8}-1$ divides $2u^{2}+u$. In this case, we know that $\lambda< q^{4}+1$, and so $q^{8}-1\leq 2u^{2}+u<3u^{2}<3\lambda< 3(q^{4}+1)$, which is impossible.
	
	Suppose now that $L=\C_{l}(q^{2})$. Then it follows from \eqref{eq:chp-bouund} that $m^{2}=n<2(l+2l^2)+(3/4)l+2$, and so  $m<(2l)^2/2$, and hence  by \cite[Theorem 1.1]{a:Liebeck1985}, we must have $m=2l$. Therefore, $L=\PSp_{2l}(q^{2})$ and $V\otimes \Fbb_{q^{2}}=W\otimes W^{(q)}$, where $W=V_{2l}(q^{2})$, the usual (projective) module for $L$. As in the  previous case, by the proof of \cite[Lemma~2.4]{a:Liebeck-HA-rank3}, the group $G_{0}$ has an orbit on the vectors of the form $v\otimes v$ of length $(q-1)(q^{2l}-1)/(q^2-1)$, and so $k/\lambda$ is divisor of $(q-1)(q^{2l}-1)/(q^2-1)$, and since $\lambda< q^{2l}$, Lemma~\ref{lem:six}(c) implies that $q^{(2l)^{2}}=q^{n}<\lambda (k/\lambda)^{2}<q^{6l}$, which is impossible.

\end{proof}

\begin{table}
	\centering
	\small
	\caption{Some parameters in Lemma~\ref{lem:exp-chp}.}\label{tbl:exp-chp}
	\begin{tabular}{lllllll}
		\hline\noalign{\smallskip}
		$L$ &
		$\u_{0}(L)$ &
		$N+l$  &
		$R_p(L)$ &
		$\u_{n}(L)$ &
		Comments \\
		\noalign{\smallskip}\hline\noalign{\smallskip}
		$\G_{2}(s)$ &
		$3$ &
		$8$ &
		$7$ &
		$19$ &
		$s=q^{t}$\\
		%
		%
		$\F_{4}(s)$ &
		$4$ &
		$28$ &
		$\geq 25$ &
		$60$ &
		$s=q^{t}$\\
		%
		%
		$\E_{6}(s)$ &
		$7$ &
		$42$  &
		$27$ &
		$91$ &
		$s=q^{t}$\\
		%
		%
		$\E_{6}^{-}(s)$ &
		$6$ &
		$42$  &
		$27$ &
		$90$ &
		$s=q^{t/2}$ or $q^{t}$
		\\
		%
		%
		$\E_{7}(s)$ &
		$7$ &
		$70$ &
		$56$ &
		$147$ &
		$s=q^{t}$\\
		%
		%
		$\E_{8}(s)$ &
		$9$ &
		$128$ &
		$248$ &
		$265$ &
		$s=q^{t}$\\
		${}^{2}\!\G_{2}(s)$ &
		${2}$ &
		$8$ &
		$7$ &
		$18$ &
		$s=q^{t}$\\
		%
		%
		${}^3\!\D_{4}(s)$ &
		$4$ &
		$16$ &
		$8$ &
		$36$ &
		$s=q^{t/3}$ or $q^{t}$
		\\
		%
		%
		\noalign{\smallskip}\hline\noalign{\smallskip}&
	\end{tabular}
\end{table}

\begin{lemma}\label{lem:exp-chp}
	The group $L$ cannot be a finite simple exceptional group of Lie type in characteristic $p$.
\end{lemma}
\begin{proof}
	Let $L=L(s)$ be a finite simple exceptional group of Lie type with $s$ a power of $p$. Since $p$ is odd, $s$ is also odd, and so $L$ cannot be $^{2}\B_{2}(s)$ or $^{2}\F_{4}(s)$. Note by Lemma~\ref{lem:affine-elem} and Proposition \ref{prop:lam-p} that $\lambda\neq p$ divides $(q-1)\cdot |\Aut(L)|$. By the same argument as in the proof of Lemma \ref{lem:class-chp}, the inequality \eqref{eq:Chp-2} holds, and so  $R_{p}(L)^{t}\leq n<t \cdot [\u_{0}(L) + 2(N+l)]$, where the value of $R_{p}(L)$, $N+l$ and $\u_{0}(L)$ are given in Table \ref{tbl:exp-chp} for each group $L$. It is easy to observe that this inequality holds only when $t=1$, or $t=2$ and $L= {}^3\!\D_4(s)$ with $s=q^{2/3}$. In the latter case, let $U$ be a Sylow $p$-subgroup of $L$. Then $\N_{L}(U)$ is a Borel subgroup of $L$, and so $|L:\N_{L}(U)|=(s^{8}+s^{4}+1)(s^{3}+1)(s+1)<s^{13}=q^{26/3}$. By Lemma \ref{lem:lie-6.1}, we have that $k/\lambda\leq (q-1)|L:\N_{L}(U)|<  q^{29/3}$, and since $\lambda<s^{\u_{0}(L)}=q^{8/3}$, it follows from Lemma \ref{lem:six}(c) that $n<22$, which contradicts the fact that $n\geq R_{p}(L)^{t}=8^{2}$. Therefore, $t=1$, and hence for each $L$, by \eqref{eq:Chp-2}, we can find an upper bound $\u_{n}(L)$ of $n$ with $n<\u_{n}(L)$ as in the fifth column of Table \ref{tbl:exp-chp}. In this case, we have $s=q$ for all $L$, with extra possibilities that $s=q^{1/2}$ or $q^{1/3}$ when $L=\,^{2}\E_{6}(s)$ or $^{3}\D_{4}(s)$, respectively. If $n=\dim(V)<2(N+l)$, then by Lemma~6.6 in \cite{a:Liebeck-98-Affine}, we conclude that $V$ is quasiequivalent to one of the modules $M(\lambda')$ given in \cite[Table IV]{a:Liebeck-98-Affine} by replacing $\lambda$ with $\lambda'$ in this reference, and so following a similar argument given in the proof of \cite[Lemma 6.8]{a:Liebeck-98-Affine} by replacing $r$ with $k/\lambda$, we observe that $L=\,^{3}\D_{4}(q^{1/3})$ when $n=8$ and $k/\lambda$ divides $2(q^4-1)$.  Then there exists a positive integer $u$ such that $uk=2\lambda f(q)$, where $f(q)=2(q^4-1)$. Since $v-1=q^{8}-1$, by \eqref{eq:k-lam}, we have that
	\begin{align}\label{eq:quasi-excp-k-lam}
		4\lambda=u^2+\frac{2u^2+2u}{q^4-1}.
	\end{align}
	This implies that $u^2\leq 4\lambda$, and since $\lambda<s^4$, it follows that $u^2<4q^{4/3}$. Moreover, \eqref{eq:quasi-excp-k-lam} implies that $q^4-1$ divides $2u^2+2u$, and since $u^2<4q^{4/3}$, we conclude that $q^4-1<4q^{4/3}$, which is impossible. Therefore, $n\geq 2(N+l)$, and so by \eqref{eq:Chp-2}, we have that $2(N+l)\leq n<\u_n(L)$, where $\u_n(L)$ is recorded in the fifth column of Table \ref{tbl:exp-chp} for each group $L$. Suppose that $L=\,^{3}\D_{4}(s)$ with $s=q$ or $q^{1/3}$. Since $32=2(N+l)\leq n<\u_n(L)=36$, we have that $n=32$, $33$, $34$ or $35$. Let $s=q$. For each $n$, we know by Lemma \ref{lem:affine-elem} that $k/\lambda$ divides $3a\cdot\gcd(q^n-1,(q^8+q^4+1)(q^6-1)(q^2-1)(q-1))$, where $q=p^a$, and so $k/\lambda$ is less than $3aq^{9}$, $3aq^{8}$, $3aq^{7}$ or $3aq^{4}$ respectively when  $n=32$, $33$, $34$ or $35$, but this violates the fact that $\lambda v<k^2$. By a similar argument, we observe that $L$ cannot be one of the remaining groups.
\end{proof}

\subsubsection{Lie type groups in cross-characteristic}

Suppose that $L=L(s)$ is a simple group of Lie type in characteristic $p'$, and is not isomorphic to an alternating group. Then Lemma~\ref{lem:diff} implies that
\begin{equation}\label{eq:Lietype-p'symm-n}
	\lambda\cdot q^{n}<(q-1)_{2'}^2\cdot |\Aut(L)|^2.
\end{equation}
If $R_{p'}(L)$ is the smallest degree of a faithful projective representation of $L$ over a field of characteristic $p'$, then $n\geq R_{p'}(L)$, and hence
\begin{align}\label{eq:Lietype-p'symm}
	\lambda \cdot q^{R_{p'}(L)}< (q-1)_{2'}^2\cdot|\Aut(L)|^2.
\end{align}
In what follows, we frequently use the lower bounds for $R_{p'}(L)$ given by \cite[Theorem 5.3.9]{b:KL-90} and they are recorded in \cite[Table 5.3.A]{b:KL-90}.

\begin{lemma}\label{lem:class-cross}
	$L$ cannot be a classical simple group of Lie type in characteristic $p'$.
\end{lemma}
\begin{proof}
	Suppose first that $L=\PSL_2(s)$ with $s\geq 4$. As we are assuming that $L$ is not isomorphic to an alternating group, $s\neq 4,5$ or $9$.
	Then by \eqref{eq:Lietype-p'symm} and \cite[Theorem 5.3.9]{b:KL-90}, we conclude that
	\begin{align}\label{eq:quasi-L2i}
		q^{(s-1)/\gcd(2,s-1)}< \log_{p'}^{2}(s)\cdot \frac{q^2}{4}\cdot s^2\cdot(s^2-1)^2.
	\end{align}
	Since $\log_{p'}^{2}(s)\leq 2s$, it follows that  $2q^{(s-1)/\gcd(2,s-1)}< q^2\cdot s^3\cdot(s^2-1)^2$, and so by taking logarithm to base $q$, we get
	\begin{align*}
		\frac{s-1}{\gcd(2,s-1)}-2<7\log_qs \leq 7\log_3s
	\end{align*}
	since $q$ is odd. Straightforward computation shows that the last inequality holds only when $s\leq 53$. Then for these possible values of $(s,n)$, we observe that \eqref{eq:Lietype-p'symm} holds only for $q$ as in Table~\ref{tbl:cross-psl2}. This table also contains the admissible $\lambda$ which is indeed a prime divisor of $(q-1) | \Aut(L)|$ greater than $3$ for each $q$. Now, filtering the triples $(n,q,\lambda)$ with respect to Lemma \ref{lem:six}(b), one see that only possibility is $L\cong \PSL_2(7)$ and $(n,q,\lambda)=(3,25,3),(3,61,7)$. However, both cases are excluded by Proposition~\ref{prop:lam-p}.

	\begin{table}
		\centering
		\caption{Some parameters for $L=\PSL_{2}(s)$ in Lemma~\ref{lem:class-cross}.}\label{tbl:cross-psl2}
		\begin{tabular}{llllllllllllll}
			\noalign{\smallskip}
			\cline{1-4} \cline{6-9} \cline{11-14}
			\noalign{\smallskip}
			$s$ & $n$   & $q$          & $\lambda$ &  & $s$  & $n$        & $q$       & $\lambda$ &  & $s$  & $n$  & $q$ & $\lambda$ \\
			\noalign{\smallskip}
			\cline{1-4} \cline{6-9} \cline{11-14}
			\noalign{\smallskip}
			$7$ & $3$   & $\leq 28224$ & -         &  & $11$ & $5$        & $\leq 71$ & -         &  & $19$ & $9$  & $3$ & $5,19$    \\
			$7$ & $6$   & $3$          & $7$       &  & $11$ & $10,11,12$ & $3$       & $5,11$    &  & $19$ & $9$  & $5$ & $19$      \\
			$7$ & $6$   & $5$          & $7$       &  & $13$ & $7$        & $3$       & $7,13$    &  & $19$ & $9$  & $7$ & $5,19$    \\
			$7$ & $6$   & $9$          & $7$       &  & $13$ & $7$        & $5$       & $7,13$    &  & $23$ & $11$ & $3$ & $11,23$   \\
			$7$ & $6$   & $11$         & $5,7$     &  & $13$ & $7$        & $7$       & $13$      &  & $23$ & $11$ & $5$ & $11,23$   \\
			$7$ & $7,8$ & $3$          & $7$       &  & $13$ & $7$        & $9$       & $7,13$    &  & $25$ & $13$ & $3$ & $5,13$    \\
			$7$ & $7,8$ & $5$          & $7$       &  & $13$ & $7$        & $11$      & $7,13$    &  & $27$ & $13$ & $3$ & $7,13$    \\
			$8$ & $7$   & $3$          & $7$       &  & $13$ & $12,13,14$ & $3$       & $7,13$    &  & $27$ & $13$ & $5$ & $7,13$    \\
			$8$ & $7$   & $5$          & $7$       &  & $16$ & $15,16,17$ & $3$       & $5,17$    &  & $29$ & $15$ & $3$ & $5,7,29$  \\
			$8$ & $7$   & $9$          & $7$       &  & $17$ & $9$        & $3$       & $17$      &  & $31$ & $15$ & $3$ & $11,31$   \\
			$8$ & $7$   & $11$         & $5,7$     &  & $17$ & $9$        & $5$       & $17$      &  & $37$ & $19$ & $3$ & $19,37$   \\
			$8$ & $9$   & $3$          & $7$       &  & $17$ & $9$        & $7$       & $17$      &  & $41$ & $21$ & $3$ & $5,7,41$  \\
			$8$ & $7$   & $5$          & $7$       &  & $17$ & $14$       & $3$       & $17$      &  & $43$ & $21$ & $3$ & $7,11,43$ \\
			\noalign{\smallskip}
			\cline{1-4} \cline{6-9} \cline{11-14}
			\noalign{\smallskip}
		\end{tabular}%
	\end{table}

	Suppose that $L=\PSL_m(s)$ with $m\geq 3$, $(m, s)\neq (4,2), (3,2)$. If $(m,s)\neq (3,4)$, then by \eqref{eq:Lietype-p'symm} and \cite[Theorem 5.3.9]{b:KL-90}, we conclude that
	\begin{align}\label{eq:quasi-Luis-ymm}
		(q-1)_{2}^{2}\cdot q^{s^{m-1}-3}<\log_{p'}^{2}(s)\cdot s^{m(m-1)}\cdot \prod_{i=2}^{m}(s^i-1)^2.
	\end{align}
	This together with facts that $\log_{p'}^{2}(s)\leq 2s$ and $\prod_{i=2}^{m}(s^i-1)^2<s^{m(m+1)-2}$ implies that $s^{m-1}-3<(2m^2-1)\cdot \log_q(s)$. This forces $s^{m-1}-3<(2m^2-1)s$. Easy computation shows that the last inequality holds only for pairs $(m,s)$ as in below:
	\begin{align*}
		\begin{array}{llll}
			m=3, & \quad 3 \leq s\leq 19;\\
			m=4, & \quad s=3,4, 5;\\
			m=5,6,7,8, & \quad s= 2.
		\end{array}
	\end{align*}
	It is easy to see that none of the values of $(m,s)$ as above fulfills \eqref{eq:quasi-Luis-ymm} with $n$ as in \cite{b:Atlas,b:Atlas-Brauer} in the role of $s^{m-1}-1$, and hence they are all excluded.
	If $(m,s)=(3,4)$, then $(q-1)_{2}^{2} \cdot q^{n-2}\leq 2^{16}\cdot 3^{6}\cdot 5^{2}\cdot 7^2$, where $n\geq 15$ by \cite{b:Atlas, b:Atlas-Brauer} since $-1 \notin G_{0}$ by Lemma \ref{lem:diff}.  Then the last inequality holds only for $q=3,5$ and $n\leq 21$, and so $H^{(\infty)} \cong L_{3}(4)$ and $(q,n)=(3,15),(3,19),(5,20)$ by \cite{b:Atlas-Brauer}. Since $\lambda\neq p$ is a prime divisor of $(q-1) |\Aut(L)|$, it follows that $\lambda =5,7$ in the first two cases and $\lambda =7$ in the remaining ones. However, all these cases violate Lemma \ref{lem:six}(b), and hence they are also excluded.

	Suppose that $L=\PSU_m(s)$ with $m\geq 3$ with $(m,s)\neq (3,2)$. Assume first that $(m,s)\neq (4,2)$ and $(4,3)$. Then, by \eqref{eq:Lietype-p'symm} and \cite[Theorem 5.3.9]{b:KL-90}, we conclude that
	\begin{align}\label{eq:quasi-Uui-symm}
		(q-1)_{2}^2\cdot q^{(s^{m}-s^{t})/(s+1)-2}<4\log_{p'}^{2}(s)\cdot  s^{m(m-1)}\cdot \prod_{i=2}^{m}(s^i-(-1)^i)^2,
	\end{align}
	where $t=\gcd(m,2)-1$. Since $\log_{p'}^{2}(s)\leq 2s$ and $\prod_{i=1}^{m}(s^i-(-1)^i)^2<s^{m^2+m-2}$,  we conclude that $s^{m}-s^{t}-2s-2<\log_q(2)+(2m^2-1)\cdot(s+1)\cdot \log_q(s)$. Since $\log_q(2)\leq 2/3$ and $\log_q(s)\leq s/2$, we have that $s^{m}-s^{t}-2s-8/3<(2m^2-1)(s+1)s/2$, and hence this inequality holds for pairs $(m,s)$ as in below:
	\begin{align*}
		\begin{array}{llll}
			m=3, & \quad s=3,4,5,7,8, 9;\\
			m=4, & \quad s= 4; \\
			m=5, & \quad s=2, 3;\\
			m=6,7,8, & \quad s= 2.
		\end{array}
	\end{align*}
	For these values of $(m,s)$, and the value of $n$ provided in \cite{b:Atlas,b:Atlas-Brauer,a:Hiss2002} in the role of $(s^{m}-s^{t})/(s+1)$, using we now employ \eqref{eq:quasi-Uui-symm}, we can find the possible values of $n$ and $q$. Further, for each such pair $(n,q)$, we determine the corresponding admissible values of $\lambda\neq p$ among the prime divisors of $(q-1)|\Aut(L)|$ greater than $3$. Therefore, we obtain the possibilities listed in Table~\ref{tbl:cross-psu}.  However, all of them violate Lemma \ref{lem:six}(b), and hence they are excluded.
	
	\begin{table}
		\centering
		\small
		\caption{Some parameters for $L=\PSU_{m}(s)$ in Lemma~\ref{lem:class-cross}.}\label{tbl:cross-psu}
		\begin{tabular}{llllllllllllll}
			\noalign{\smallskip}\cline{1-5}\cline{7-11}\noalign{\smallskip}
			$m$ & $s$ & $n$ & $q$ & $\lambda$  & &  $m$ & $s$ & $n$ & $q$ &  $\lambda$\\
			\noalign{\smallskip}\cline{1-5}\cline{7-11}\noalign{\smallskip}
			$3$ & $3$ & $6$     & $\leq71$  & $5,7,11,13,23,29$ &  & $5$ & $2$ & $20$    & $5$  & $11$        \\
			$3$ & $3$ & $6,7$   & $\leq 31$ & $5,7,11,13$       &  & $6$ & $2$ & $21,22$ & $11$ & $5$         \\
			$3$ & $4$ & $12,13$ & $7$       & $5,13$            &  & $6$ & $2$ & $21,22$ & $9$  & $5,7,11$    \\
			$3$ & $4$ & $12,13$ & $5$       & $13$              &  & $6$ & $2$ & $21,22$ & $7$  & $5,11$      \\
			$3$ & $4$ & $12,13$ & $3$       & $5,13$            &  & $6$ & $2$ & $21,22$ & $5$  & $7,11$      \\
			$3$ & $5$ & $20,21$ & $3$       & $5,7$             &  & $6$ & $2$ & $21,22$ & $3$  & $5,7,11$    \\
			$5$ & $2$ & $10$    & $\leq 59$ & $5,7,11,13,23,29$ &  & $7$ & $2$ & $42$    & $3$  & $5,7,11,43$ \\
			$5$ & $2$ & $11$    & $\leq 19$ & $5,11$            &  & $7$ & $2$ & $42,43$ & $5$  & $7,11,43$   \\
			$5$ & $2$ & $20$    & $3$       & $5,11$            &  &     &     &         &      &             \\
			\noalign{\smallskip}\cline{1-5}\cline{7-11}\noalign{\smallskip}
		\end{tabular}
	\end{table}
	
	If $(m,s)=(4,2)$ or $(4,3)$, then \eqref{eq:Lietype-p'symm} implies that $(q-1)_{2}^{2}\cdot q^{n-2}\leq 2^{14}\cdot 3^{8}\cdot 5^{2}$ or $(q-1)_{2}^{2}\cdot q^{n-2}\leq 2^{20}\cdot 3^{12}\cdot 5^2\cdot 7^2$, respectively, where $n$ is as in \cite{b:Atlas,b:Atlas-Brauer}.  We then observe that the last inequalities hold only for pairs $(n,q)$ as in Table~\ref{tbl:cross-psu42}. Further, for each such pair $(n,q)$, we can determine the corresponding admissible values of $\lambda$ as recorded in the same table by  considering the fact that $\lambda\neq p$ is a prime divisor of $(q-1)|\Aut(L)|$ greater than $3$. Again, none of these possibilities occurs by Lemma \ref{lem:six}(b).
	
	\begin{table}
		\centering
		\small
		\caption{Some parameters for $L=\PSU_{4}(s)$ with $s=2,3$ in Lemma~\ref{lem:class-cross}.}\label{tbl:cross-psu42}
		\begin{tabular}{llllllllllllll}
			\noalign{\smallskip}\cline{1-5}\cline{7-11}\noalign{\smallskip}
			$m$ &  $s$ & $n$ & $q$ & $\lambda$  &&  $m$ & $s$ & $n$ & $q$ &  $\lambda$\\
			\noalign{\smallskip}\cline{1-5}\cline{7-11}\noalign{\smallskip}
			$4$ & $2$ & $5$  & $\leq 503$ & $\leq 251$  &&  $4$ & $3$ & $15$ & $3$ &  $5,7$ \\
			$4$ & $2$ & $6$  & $\leq 107$ & $\leq 53$  &&   $4$ & $3$  &  $15$     & $7$   &   $5$ \\
			$4$ & $2$  & $10$  & $3,7$     & $5$  &&        $4$ & $3$     &  $21$ & $3$ &  $5,7$ \\
			$4$ & $2$  & $14$  & $3$       & $5$  &&      &     &      &      &      \\
			\noalign{\smallskip}\cline{1-5}\cline{7-11}\noalign{\smallskip}
		\end{tabular}
	\end{table}
	
	Suppose that $L=\PSp_{2m}(s)'$ with $m\geq 2$. The case where $\PSp_4(2)'\cong \A_{6}$ is excluded by Lemma \ref{lem:affine-Alt}, and $\PSp_4(3)\cong \PSU_4(2)$, and so we may assume that $L=\PSp_{2m}(s)$ with $(m,s)\neq (2,2),(2,3)$. Then by \eqref{eq:Lietype-p'symm}, we have that
	\begin{align}\label{eq:quasi-Spi}
		(q-1)_{2}^{2}\cdot q^{R_{p'}(L)-2}<\log_{p'}(s)^2\cdot s^{2m^2}\cdot \prod_{i=1}^{m}(s^{2i}-1)^2.
	\end{align}
	Since $\log_{p'}(s)^2\leq 2s$ and $\prod_{i=1}^{m}(s^{2i}-1)^2<s^{2m^2+2m}$, we conclude that  $R_{p'}(L)-2<(4m^2+2m+1)\cdot \log_q(s)$ with $\log_q(s) \leq s/2$ as $q$ is odd. Therefore,
	\begin{equation*}
		2R_{p'}(L)-4<(4m^2+2m+1)s.
	\end{equation*}
	Now, a lower bound for $2R_{p'}(L)$ is either $s^{m}-1$ or $s(s^{m}-1)(s^{m-1}-1)/(s+1)$ according as $s$ is odd or even, respectively, by \cite[Theorem 5.3.9]{b:KL-90} and \cite[Table 1]{a:SZ}. Thus $s^{m}-5<(4m^2+2m+1)s$ for $s$ odd, and $(s^{m}-1)(s^{m-1}-1)<(4m^2+2m+1)(s+1)$ for $s$ even. Hence, we obtain pairs $(m,s)$ as below:
	\begin{align*}
		\begin{array}{llll}
			m=2, & \quad s=4,5,7,11,13,17,19;\\
			m=3, & \quad s=2,3,5;\\
			m=4, & \quad s=2,3;\\
			m=5, & \quad s=3.
		\end{array}
	\end{align*}
	For these values of $(m,s)$, we use \eqref{eq:quasi-Spi} with $n$ provided in \cite{b:Atlas,b:Atlas-Brauer,a:Hiss2002} in the role of a lower bound for $R_{p^{\prime}(L)}$ given above, and we can find the possible values of $n$ and $q=p^a$, for which, we can determine the corresponding possible values of $\lambda\neq p$ dividing $(q-1) |\Aut(L)|$. This gives the possibilities recorded in Table~\ref{tbl:cross-psp}, however, none of these cases can occur by Lemma \ref{lem:six}(b).
	
	\begin{table}
		\centering
		\small
		\caption{Some parameters for $L=\PSp_{2m}(s)$ in Lemma~\ref{lem:class-cross}.}\label{tbl:cross-psp}
		\begin{tabular}{llllllllllllll}
			\noalign{\smallskip}
			\cline{1-5}\cline{7-11}
			\noalign{\smallskip}
			$m$ & $s$ & $n$ & $q$ & $\lambda$  &&  $m$ & $s$ & $n$ & $q$ &  $\lambda$\\
			\noalign{\smallskip}
			\cline{1-5}\cline{7-11}
			\noalign{\smallskip}
			$2$ & $4$ & $18$ & $3$        & $5,17$     &  & $3$ & $2$ & $15$    & $7$       & $5$            \\
			$2$ & $4$ & $18$ & $5$        & $17$       &  & $3$ & $2$ & $21,27$ & $3$       & $5,7$          \\
			$2$ & $5$ & $13$ & $\leq 19$  & $5,13$     &  & $3$ & $3$ & $13$    & $\leq 49$ & $5,7,11,13,23$ \\
			$2$ & $7$ & $25$ & $3$        & $5,7$      &  & $3$ & $5$ & $62$    & $3$       & $5,7,13$       \\
			$2$ & $7$ & $25$ & $5$        & $7$        &  & $8$ & $2$ & $35$    & $3$       & $5,7,17$       \\
			$3$ & $2$ & $7$  & $\leq 283$ & $\leq 131$ &  & $8$ & $2$ & $35$    & $5$       & $7,17$         \\
			$3$ & $2$ & $14$ & $3$        & $5,7$      &  & $8$ & $3$ & $41$    & $5$       & $7,17,41$      \\
			$3$ & $2$ & $15$ & $5$        & $7$        &  & $8$ & $3$ & $41$    & $5$       & $7,17$         \\
			\noalign{\smallskip}\cline{1-5}\cline{7-11}\noalign{\smallskip}
		\end{tabular}
	\end{table}
	
	Suppose that $L=\Omega_{2m+1}(s)$ with $s$ odd. Note that $\POm_{5}(s)\cong\PSp_4(s)$ and $\POm_{3}(s)\cong \PSL_2(s)$. Then we can assume that $m\geq 3$. If $(m,s)\neq (3,3)$, then by \eqref{eq:Lietype-p'symm} and \cite[Theorem 5.3.9]{b:KL-90}, we have that
	\begin{align*}
		(q-1)^{2}_{2}\cdot q^{s^{m-1}(s^{m-1}-1)-2}<\log_{p'}(s)^2\cdot s^{2m^2}\cdot \prod_{i=1}^{m}(s^{2i}-1)^2.
	\end{align*}
	Note that $\log_{p'}(s)^2\leq s^2$ and $\prod_{i=1}^{m}(s^{2i}-1)^2<s^{2m^2+2m}$, Then  $(s^{m-1}(s^{m-1}-1)-2)\cdot \log_2(q)<(4m^2+2m+2)\cdot \log_2(s)$. Since $\log_2(s)\leq 2s/3$, we conclude that
	\begin{align*}
		9s^{2m-2}-9s^{m-1}-18<16m^2s+8ms+8s,
	\end{align*}
	but we cannot find any pairs $(m,s)$ satisfying this inequality, and hence this case is ruled out. We now consider the case where $(m, s)=(3, 3)$. Then by \eqref{eq:Lietype-p'symm}, we get $(q-1)^{2}_{2}\cdot q^{n-2}<2^{20}\cdot 3^{18}\cdot 5 \cdot 7^2\cdot 13^2$, where either $n=27$, or $n\geq 78$ by \cite[Table 2]{a:Hiss2002}. It is easy to check that the last inequality holds only for pairs $(n,q)=(27,5)$, for which $\lambda=7,17$ as $\lambda\neq p$ is a prime divisor of $(q-1) |\Aut(L)|$. However, this case cannot also occur as these values  violate Lemma \ref{lem:six}(b).
	
	Suppose that $L=\POm_{2m}^{+}(s)$ with $m\geq 4$. If $(m,s)\neq (4,2)$, then \eqref{eq:Lietype-p'symm} and \cite[Theorem 5.3.9]{b:KL-90} imply that
	\begin{align*}
		(q-1)_{2}^{2}\cdot q^{s^{m-2}(s^{m-1}-1)-2}<8\log_{p'}(s)^2\cdot s^{2m(m-1)}(s^m-1)^2\cdot \prod_{i=1}^{m-1}(s^{2i}-1)^2.
	\end{align*}
	This together with facts that $8\log_{p'}(s)^2\leq s^3$ and $\prod_{i=1}^{m-1}(s^{2i}-1)^2<s^{2m^2-2m}$ implies that $[s^{m-2}(s^{m-1}-1)-2]\cdot \log_2(q)<(4m^2-2m+3)\cdot \log_2(s)$, and so
	\begin{align*}
		9s^{2m-3}-9s^{m-2}-18<16m^2s-8ms+12s.
	\end{align*}
	This inequality does not hold for any pairs $(m,s)$, which is a contradiction. Now if $(m,s)=(4,2)$, then by \eqref{eq:Lietype-p'symm}, we get $(q-1)_{2}^{2}\cdot q^{n-2}<2^{26}\cdot 3^{10}\cdot 5^4\cdot 7^2$, where $n=28,25$ or $n\geq 50$ by \cite[Table 2]{a:Hiss2002} and Lemma \ref{lem:diff}. Then, the last inequality forces $n=28,25$ and $q=3$, and so we find $\lambda=5,7$ since $\lambda\neq p$ is a prime divisor of $(q-1) |\Aut(L)|$ which is greater than $3$. These possibilities are also excluded by Lemma \ref{lem:six}(b).
	
	Suppose finally that $L=\POm_{2m}^{-}(s)$ with $m\geq 4$. In this case,  \eqref{eq:Lietype-p'symm} and \cite[Theorem 5.3.9]{b:KL-90} implies that
	\begin{align*}
		(q-1)_{2}^{2}\cdot q^{(s^{m-2}-1)(s^{m-1}+1)-2}<\log_{p'}(s)^2\cdot s^{2m(m-1)}(s^m+1)^2\cdot \prod_{i=1}^{m-1}(s^{2i}-1)^2.
	\end{align*}
	As $4\log_{p'}(s)^2\leq s^4$ and $\prod_{i=1}^{m-1}(s^{2i}-1)^2<s^{2m^2-2m}$, it follows that  $[(s^{m-2}-1)(s^{m-1}+1)-2]\cdot \log_2(q)<(4m^2-2m+4)\cdot \log_2(s)$, and hence
	\begin{align*}
		9(s^{m-2}-1)(s^{m-1}+1)-18<16m^2s-8ms+16s,
	\end{align*}
	however, it is easy to observe that this inequality does not hold for any pairs $(m,s)$, which is a contradiction.
\end{proof}

\begin{lemma}\label{lem:exp-cross}
	The group $L$ cannot be a finite simple exceptional group of Lie type in characteristic $p'$.
\end{lemma}
\begin{proof}
	Suppose that $L=L(s)$ is a finite simple exceptional group of Lie type in characteristic $p'$. We first note that the case where $L$ is $\G_2(2)'\cong \PSU_{3}(3)$ or ${}^2\G_{2}^{'}(3) \cong \PSL_{2}(8)$ has already been ruled out in Lemma \ref{lem:class-cross}, and $L$ cannot be the Tits group  $^{2}\F_{4}(2)'$ by Lemma \ref{lem:sporadic}. Note also that the inequalities \eqref{eq:Lietype-p'symm-n} and \eqref{eq:Lietype-p'symm} hold. Therefore, by considering the lower bounds $R_{p'}(L)$ recorded in \cite[Table~5.3A]{b:KL-90}, we conclude that $L$ cannot be of type $^{2}\B_{2}$, ${}^2\G_2$, ${}^3\D_{4}$, ${}^2\E_6$, $\E_6$, $\E_7$ or $\E_{8}$.
	In what follows, we consider the remaining cases, namely $L$ is of one of types $\F_{4}$, $^{2}\F_{2}$ and $\G_{2}$.
	
	Suppose that $L=\F_4(s)$. If $s$ is odd, then by \eqref{eq:Lietype-p'symm} and \cite[Theorem 5.3.9]{b:KL-90}, we have $q^{s^{6}(s^2-1)-2}<\log_{p'}(s)^2\cdot s^{108}$. Since $\log_{p'}(s)^2\leq s$, it follows that $[s^{6}(s^2-1)-2]\cdot \log_2(q)<109\cdot \log_2(s)$, and so  $s^{6}(s^2-1)<111$, which is impossible for $s\geq 3$. If $s\neq 2$ is even, then by \eqref{eq:Lietype-p'symm} and \cite[Theorem 5.3.9]{b:KL-90}, we must have  $q^{\frac{1}{2}s^{7}(s^3-1)(s-1)-2}<\log_{p'}(s)^2\cdot s^{108}$. Note that $\log_{p'}(s)^2\leq s^2$. Then  $[s^{7}(s^3-1)-2]\cdot \log_2(q)<110\cdot \log_2(s)$, and so $s^{6}(s^3-1)<112$, which is impossible as $s\geq 4$. If $s=2$, then \eqref{eq:Lietype-p'symm} implies that $(q-1)_{2}^{2}\cdot q^{n-2}< 2^{50}\cdot 3^{12}\cdot 5^4\cdot 7^4\cdot 13^2\cdot 17^2$, which requires $n<68$, however \cite[Table 2]{a:Hiss2002} and Lemma \ref{lem:diff} implies that $n>250$, which is a contradiction.
	
	Suppose finally that $L=\,{}^2\F_4(s)$ with $s\geq 8$.
	In this case, by \eqref{eq:Lietype-p'symm} and \cite[Theorem 5.3.9]{b:KL-90}, we have $q^{s^4(s-1)-2}<\log_{p'}(s)^2\cdot s^{26}$. Since $\log_{p'}(s)^2\leq s^2$, it follows that $[s^4(s-1)-2]\cdot \log_2(q)<28\cdot \log_2(s)$, and so $s^4(s-1)-2<13s$, which has no solutions for $s>2$.
	
	Suppose now that $L=G_2(s)$ with $s \geq 3$. If $s\neq 3,4$, then by \eqref{eq:Lietype-p'symm} and \cite[Theorem 5.3.9]{b:KL-90}, we have $q^{s(s^2-1)-2}<\log_{p'}(s)^2\cdot s^{28}$.
	Note that $\log_{p'}(s)^2\leq s$. Then $[s(s^2-1)-2]<29\cdot \log_q(s)$, and since $\log_q(s)\leq 2s/3$, we must have $s^2-1-1/3<29/3$, which is impossible since $s > 4$.
	If $s=4$, then $(q-1)_{2}^{2} \cdot q^{n-2}< 2^{26}\cdot 3^{6} \cdot 5^4 \cdot 7^2\cdot 13^2$, and so $n\leq 38$, however, \cite[Table 2]{a:Hiss2002} and Lemma \ref{lem:diff} require $n \geq 64$, which is a contradiction. Thus, we have $s=3$. In this case, by \cite[Table 2]{a:Hiss2002}, we know that $n=14$ or $n \geq 27$. On the other hand, \eqref{eq:Lietype-p'symm} implies that $(q-1)_{2}^{2} \cdot q^{n-2}< 2^{14}\cdot 3^{12}\cdot 7^2\cdot 13^2$, and so $n\leq 31$, and hence $n=14$, for which one obtains $q=5,7,11$ since $s \neq p$. The fact that $\lambda\neq p$ is an odd prime divisor of $(q-1) |\Aut(L)|$ implying that $n=14$ and $(q,\lambda)=(5,7)$, $(5,13)$, $(7,13)$, $(11,5)$, $(11,7)$, $(11,13)$. However, these cases are ruled out by Lemma \ref{lem:six}(b).
\end{proof}

\noindent \textbf{Proof of Theorem~\ref{thm:main}.}
Suppose that $\Dmc$ is a nontrivial symmetric $(v,k,\lambda)$ design with $\lambda$ prime admitting a flag-transitive and point-primitive automorphism group $G=TH \leq \AGL_d(p)=\AGL(V)$, where $T \cong (\mathbb{Z}_p)^d$ is the translation group. Then by Corollary~\ref{cor:irr} and Lemma~\ref{lem:C8}, the point-stabilizer $G_{0}$ of the point $0$) is an irreducible subgroup of $\GL_d( p)$ not containing a classical group, and so Proposition~\ref{pro:Aschbacher} gives the possibilities for $G_{0}$. If $\lambda=p$, then Proposition~\ref{prop:lam-p} implies that $\Dmc$ is the unique symmetric $(16,6,2)$ design. By the main results in \cite{a:Zhou-lam3-affine,a:Regueiro-reduction}, we can assume that $\lambda\neq 2,3$. Moreover, if $\lambda$ is coprime to $k$, then $G\leq \AGaL_{1}(p)$ by \cite{a:Biliotti-CP-sym-affine}. Therefore, we assume that $\lambda\neq p$ is at least $5$ and it is an odd prime divisor of $k$, and hence, by excluding the possibility $(16,6,2)$, we conclude form Lemmas~\ref{lem:affine-geom}-\ref{lem:exp-cross} that $G$ must be a subgroup of $1$-dimensional affine group.\medskip

\noindent \textbf{Proof of Corollary~\ref{cor:main}.} The proof follows immediate from Theorem \ref{thm:main} and the main results of \cite{a:ABD-PrimeLam,a:ABD-Exp,a:ADM-PrimeLam-An,a:Biliotti-CP-sym-affine,a:Zhou-lam3-affine,a:Regueiro-reduction}.

\section*{Acknowledgments}

The second author (Mohsen Bayat) is supported by INSF (Iran National Science Foundation) and IPM (Institute for Research in Fundamental Sciences) with grant number 4013853.


\begin{thebibliography}{10}
	
	\bibitem{a:ABD-PSL2}
	S.~H. Alavi, M.~Bayat, and A.~Daneshkhah.
	\newblock Symmetric designs admitting flag-transitive and point-primitive
	automorphism groups associated to two dimensional projective special groups.
	\newblock {\em Designs, Codes and Cryptography}, 79(2):337--351, 2016.
	
	\bibitem{a:ABD-PrimeLam}
	S.~H. Alavi, M.~Bayat, and A.~Daneshkhah.
	\newblock Almost simple groups of {L}ie type and symmetric designs with
	$\lambda$ prime.
	\newblock {\em The Electronic Journal of Combinatorics}, 28(2):P2.13, apr 2021.
	
	\bibitem{a:ABD-Exp}
	S.~H. Alavi, M.~Bayat, and A.~Daneshkhah.
	\newblock Finite exceptional groups of {L}ie type and symmetric designs.
	\newblock {\em Discrete Mathematics}, 345(8):112894, 2022.
	
	\bibitem{a:ADM-PrimeLam-An}
	S.~H. Alavi, A.~Daneshkhah, and F.~Mouseli.
	\newblock Almost simple groups as flag-transitive automorphism groups of
	symmetric designs with {$\lambda$} prime.
	\newblock {\em Ars Math. Contemp.}, 23(4):Paper No. 3, 10, 2023.
	
	\bibitem{a:ADP-bi}
	S.~H. Alavi, A.~Daneshkhah, and C.~E. Praeger.
	\newblock Symmetries of biplanes.
	\newblock {\em Designs, Codes and Cryptography}, 88(11):2337--2359, aug 2020.
	
	\bibitem{a:Aschbacher-84}
	M.~Aschbacher.
	\newblock On the maximal subgroups of the finite classical groups.
	\newblock {\em Invent. Math.}, 76(3):469--514, 1984.
	
	\bibitem{a:BambergPenttila-2008}
	J.~Bamberg and T.~Penttila.
	\newblock Overgroups of cyclic {S}ylow subgroups of linear groups.
	\newblock {\em Comm. Algebra}, 36(7):2503--2543, 2008.
	
	\bibitem{b:Beth-I}
	T.~Beth, D.~Jungnickel, and H.~Lenz.
	\newblock {\em Design theory. {V}ol. {I}}, volume~69 of {\em Encyclopedia of
		Mathematics and its Applications}.
	\newblock Cambridge University Press, Cambridge, second edition, 1999.
	
	\bibitem{a:Biliotti-CP-sym-affine}
	M.~Biliotti and A.~Montinaro.
	\newblock On flag-transitive symmetric designs of affine type.
	\newblock {\em J. Combin. Des.}, 25(2):85--97, 2017.
	
	\bibitem{a:Braic-255}
	S.~Brai{\'c}.
	\newblock Primitive symmetric designs with at most 255 points.
	\newblock {\em Glas. Mat. Ser. III}, 45(65)(2):291--305, 2010.
	
	\bibitem{a:Braic-2500-power}
	S.~Brai{\'c}, A.~Golemac, J.~Mandi{\'c}, and T.~Vu{\v{c}}i{\v{c}}i{\'c}.
	\newblock Primitive symmetric designs with prime power number of points.
	\newblock {\em J. Combin. Des.}, 18(2):141--154, 2010.
	
	\bibitem{a:Braic-2500-nopower}
	S.~Brai{\'c}, A.~Golemac, J.~Mandi{\'c}, and T.~Vu{\v{c}}i{\v{c}}i{\'c}.
	\newblock Primitive symmetric designs with up to 2500 points.
	\newblock {\em J. Combin. Des.}, 19(6):463--474, 2011.
	
	\bibitem{b:BHR-Max-Low}
	J.~N. Bray, D.~F. Holt, and C.~M. Roney-Dougal.
	\newblock {\em The maximal subgroups of the low-dimensional finite classical
		groups}, volume 407 of {\em London Mathematical Society Lecture Note Series}.
	\newblock Cambridge University Press, Cambridge, 2013.
	\newblock With a foreword by Martin Liebeck.
	
	\bibitem{a:BDD-1988}
	F.~Buekenhout, A.~Delandtsheer, and J.~Doyen.
	\newblock Finite linear spaces with flag-transitive groups.
	\newblock {\em Journal of Combinatorial Theory, Series A}, 49(2):268 -- 293,
	1988.
	
	\bibitem{a:BDDKLS90}
	F.~Buekenhout, A.~Delandtsheer, J.~Doyen, P.~B. Kleidman, M.~W. Liebeck, and
	J.~Saxl.
	\newblock Linear spaces with flag-transitive automorphism groups.
	\newblock {\em Geom. Dedicata}, 36(1):89--94, 1990.
	
	\bibitem{b:Handbook}
	C.~J. Colbourn and J.~H. Dinitz, editors.
	\newblock {\em Handbook of combinatorial designs}.
	\newblock Discrete Mathematics and its Applications (Boca Raton). Chapman \&
	Hall/CRC, Boca Raton, FL, second edition, 2007.
	
	\bibitem{b:Atlas}
	J.~H. Conway, R.~T. Curtis, S.~P. Norton, R.~A. Parker, and R.~A. Wilson.
	\newblock {\em Atlas of finite groups}.
	\newblock Oxford University Press, Eynsham, 1985.
	\newblock Maximal subgroups and ordinary characters for simple groups, With
	computational assistance from J. G. Thackray.
	
	\bibitem{a:Delandtsheer-86}
	A.~Delandtsheer.
	\newblock Flag-transitive finite simple groups.
	\newblock {\em Arch. Math. (Basel)}, 47(5):395--400, 1986.
	
	\bibitem{a:Delandsheer-Lin-An-2001}
	A.~Delandtsheer.
	\newblock Finite flag-transitive linear spaces with alternating socle.
	\newblock In {\em Algebraic combinatorics and applications ({G}\"{o}\ss
		weinstein, 1999)}, pages 79--88. Springer, Berlin, 2001.
	
	\bibitem{a:Dempwolff2001}
	U.~Dempwolff.
	\newblock Primitive rank 3 groups on symmetric designs.
	\newblock {\em Designs, Codes and Cryptography}, 22(2):191--207, 2001.
	
	\bibitem{b:Dixon}
	J.~D. Dixon and B.~Mortimer.
	\newblock {\em Permutation groups}, volume 163 of {\em Graduate Texts in
		Mathematics}.
	\newblock Springer-Verlag, New York, 1996.
	
	\bibitem{a:Zhou-lam3-affine}
	H.~Dong and S.~Zhou.
	\newblock Affine groups and flag-transitive triplanes.
	\newblock {\em Sci. China Math.}, 55(12):2557--2578, 2012.
	
	\bibitem{a:Foulser-64}
	D.~A. Foulser.
	\newblock The flag-transitive collineation groups of the finite {D}esarguesian
	affine planes.
	\newblock {\em Canadian J. Math.}, 16:443--472, 1964.
	
	\bibitem{a:Foulser-64-sol}
	D.~A. Foulser.
	\newblock Solvable flag transitive affine groups.
	\newblock {\em Math. Z.}, 86:191--204, 1964.
	
	\bibitem{a:Foulser-69}
	D.~A. Foulser.
	\newblock Solvable primitive permutation groups of low rank.
	\newblock {\em Trans. Amer. Math. Soc.}, 143:1--54, 1969.
	
	\bibitem{GAP4}
	The GAP~Group.
	\newblock {\em {GAP -- Groups, Algorithms, and Programming, Version 4.12.2}},
	2022.
	
	\bibitem{a:HigMcL-61}
	D.~G. Higman and J.~E. McLaughlin.
	\newblock Geometric {$ABA$}-groups.
	\newblock {\em Illinois J. Math.}, 5:382--397, 1961.
	
	\bibitem{a:Hiss2002}
	G.~Hiss and G.~Malle.
	\newblock Corrigenda: ``{L}ow-dimensional representations of quasi-simple
	groups'' [{LMS} {J}. {C}omput.\ {M}ath.\ {\bf 4} (2001), 22--63; {MR}1835851
	(2002b:20015)].
	\newblock {\em LMS J. Comput. Math.}, 5:95--126 (electronic), 2002.
	
	\bibitem{a:Hussain-1945}
	Q.~M. Hussain.
	\newblock Symmetrical incomplete block designs with {$\lambda=2$}, {$k=8$} or
	{$9$}.
	\newblock {\em Bull. Calcutta Math. Soc.}, 37:115--123, 1945.
	
	\bibitem{a:James-mini-dim-1983}
	G.~D. James.
	\newblock On the minimal dimensions of irreducible representations of symmetric
	groups.
	\newblock {\em Math. Proc. Cambridge Philos. Soc.}, 94(3):417--424, 1983.
	
	\bibitem{b:Atlas-Brauer}
	C.~Jansen, K.~Lux, R.~Parker, and R.~Wilson.
	\newblock {\em An atlas of {B}rauer characters}, volume~11 of {\em London
		Mathematical Society Monographs. New Series}.
	\newblock The Clarendon Press, Oxford University Press, New York, 1995.
	\newblock Appendix 2 by T. Breuer and S. Norton, Oxford Science Publications.
	
	\bibitem{a:Kantor-85-2-trans}
	W.~M. Kantor.
	\newblock Classification of {$2$}-transitive symmetric designs.
	\newblock {\em Graphs Combin.}, 1(2):165--166, 1985.
	
	\bibitem{a:Kantor-85-Homgeneous}
	W.~M. Kantor.
	\newblock Homogeneous designs and geometric lattices.
	\newblock {\em J. Combin. Theory Ser. A}, 38(1):66--74, 1985.
	
	\bibitem{a:Kantor-87-Odd}
	W.~M. Kantor.
	\newblock Primitive permutation groups of odd degree, and an application to
	finite projective planes.
	\newblock {\em J. Algebra}, 106(1):15--45, 1987.
	
	\bibitem{a:Kantor-rank3}
	W.~M. Kantor and R.~A. Liebler.
	\newblock The rank {$3$} permutation representations of the finite classical
	groups.
	\newblock {\em Trans. Amer. Math. Soc.}, 271(1):1--71, 1982.
	
	\bibitem{b:KL-90}
	P.~Kleidman and M.~Liebeck.
	\newblock {\em The subgroup structure of the finite classical groups}, volume
	129 of {\em London Mathematical Society Lecture Note Series}.
	\newblock Cambridge University Press, Cambridge, 1990.
	
	\bibitem{a:Kleidman-Exp}
	P.~B. Kleidman.
	\newblock The finite flag-transitive linear spaces with an exceptional
	automorphism group.
	\newblock In {\em Finite geometries and combinatorial designs ({L}incoln, {NE},
		1987)}, volume 111 of {\em Contemp. Math.}, pages 117--136. Amer. Math. Soc.,
	Providence, RI, 1990.
	
	\bibitem{b:Lander}
	E.~S. Lander.
	\newblock {\em Symmetric designs: an algebraic approach}, volume~74 of {\em
		London Mathematical Society Lecture Note Series}.
	\newblock Cambridge University Press, Cambridge, 1983.
	
	\bibitem{a:Liebeck1985}
	M.~W. Liebeck.
	\newblock On the orders of maximal subgroups of the finite classical groups.
	\newblock {\em Proc. London Math. Soc. (3)}, 50(3):426--446, 1985.
	
	\bibitem{a:Liebeck-HA-rank3}
	M.~W. Liebeck.
	\newblock The affine permutation groups of rank three.
	\newblock {\em Proceedings of the London Mathematical Society},
	s3-54(3):477--516, may 1987.
	
	\bibitem{a:Liebeck-98-Affine}
	M.~W. Liebeck.
	\newblock The classification of finite linear spaces with flag-transitive
	automorphism groups of affine type.
	\newblock {\em J. Combin. Theory Ser. A}, 84(2):196--235, 1998.
	
	\bibitem{a:LP-1992}
	M.~W. Liebeck and C.~E. Praeger.
	\newblock Affine distance-transitive groups with alternating or symmetric point
	stabiliser.
	\newblock {\em European J. Combin.}, 13(6):489--501, 1992.
	
	\bibitem{a:LPS90}
	M.~W. Liebeck, C.~E. Praeger, and J.~Saxl.
	\newblock The maximal factorizations of the finite simple groups and their
	automorphism groups.
	\newblock {\em Mem. Amer. Math. Soc.}, 86(432):iv+151, 1990.
	
	\bibitem{a:Mandic-Sym-Imp-lam10}
	J.~Mandi\'{c} and A.~\v{S}uba\v{s}i\'{c}.
	\newblock Flag-transitive and point-imprimitive symmetric designs with
	{$\lambda\leq10$}.
	\newblock {\em J. Combin. Theory Ser. A}, 189:Paper No. 105620, 2022.
	
	\bibitem{a:Montinaro-sym-imp-ellgeq3}
	A.~Montinaro.
	\newblock Flag-transitive, point-imprimitive symmetric 2-{$(v,k,\lambda)$}
	designs with {$k>\lambda\,(\lambda-3)/2$}.
	\newblock {\em Discrete Math.}, 347(9):Paper No. 114070, 30, 2024.
	
	\bibitem{a:Montinaro-sym-imp-ell2}
	A.~Montinaro.
	\newblock On the symmetric 2-{$(v,k,\lambda)$} designs with a flag-transitive
	point-imprimitive automorphism group.
	\newblock {\em J. Algebra}, 653:54--101, 2024.
	
	\bibitem{a:Regueiro-alt-spor}
	E.~O'Reilly-Regueiro.
	\newblock Biplanes with flag-transitive automorphism groups of almost simple
	type, with alternating or sporadic socle.
	\newblock {\em European J. Combin.}, 26(5):577--584, 2005.
	
	\bibitem{a:Regueiro-reduction}
	E.~O'Reilly-Regueiro.
	\newblock On primitivity and reduction for flag-transitive symmetric designs.
	\newblock {\em J. Combin. Theory Ser. A}, 109(1):135--148, 2005.
	
	\bibitem{a:Regueiro-classical}
	E.~O'Reilly-Regueiro.
	\newblock Biplanes with flag-transitive automorphism groups of almost simple
	type, with classical socle.
	\newblock {\em J. Algebraic Combin.}, 26(4):529--552, 2007.
	
	\bibitem{a:Regueiro-Exp}
	E.~O'Reilly-Regueiro.
	\newblock Biplanes with flag-transitive automorphism groups of almost simple
	type, with exceptional socle of {L}ie type.
	\newblock {\em J. Algebraic Combin.}, 27(4):479--491, 2008.
	
	\bibitem{a:Praeger-45-12-3}
	C.~E. Praeger.
	\newblock The flag-transitive symmetric designs with 45 points, blocks of size
	12, and 3 blocks on every point pair.
	\newblock {\em Des. Codes Cryptogr.}, 44(1-3):115--132, 2007.
	
	\bibitem{a:Praeger-imprimitive}
	C.~E. Praeger and S.~Zhou.
	\newblock Imprimitive flag-transitive symmetric designs.
	\newblock {\em J. Combin. Theory Ser. A}, 113(7):1381--1395, 2006.
	
	\bibitem{a:Salwach-Mezzaroba-1978}
	C.~J. Salwach and J.~A. Mezzaroba.
	\newblock The four biplanes with {$k=9$}.
	\newblock {\em J. Combinatorial Theory Ser. A}, 24(2):141--145, 1978.
	
	\bibitem{a:Saxl2002}
	J.~Saxl.
	\newblock On finite linear spaces with almost simple flag-transitive
	automorphism groups.
	\newblock {\em J. Combin. Theory Ser. A}, 100(2):322--348, 2002.
	
	\bibitem{a:SZ}
	G.~M. Seitz and A.~E. Zalesskii.
	\newblock On the minimal degrees of projective representations of the finite
	{C}hevalley groups. {II}.
	\newblock {\em J. Algebra}, 158(1):233--243, 1993.
	
	\bibitem{a:Skinner-dioph}
	C.~Skinner.
	\newblock The {D}iophantine equation {$x^2=4q^n-4q+1$}.
	\newblock {\em Pacific J. Math.}, 139(2):303--309, 1989.
	
	\bibitem{b:Taylor92}
	D.~E. Taylor.
	\newblock {\em The geometry of the classical groups}, volume~9 of {\em Sigma
		Series in Pure Mathematics}.
	\newblock Heldermann Verlag, Berlin, 1992.
	
	\bibitem{a:Wagner-1965-affine}
	A.~Wagner.
	\newblock On finite affine line transitive planes.
	\newblock {\em Math. Z.}, 87:1--11, 1965.
	
	\bibitem{a:Zhou-sym-lam-prime}
	Y.~Zhang, Z.~Zhang, and S.~Zhou.
	\newblock Reduction for primitive flag-transitive symmetric 2-
	{$(v,k,\lambda)$} designs with {$\lambda$} prime.
	\newblock {\em Discrete Math.}, 343(6):111843, 4, 2020.
	
	\bibitem{a:Zhou-lam3-spor}
	S.~Zhou and H.~Dong.
	\newblock Sporadic groups and flag-transitive triplanes.
	\newblock {\em Sci. China Ser. A}, 52(2):394--400, 2009.
	
	\bibitem{a:Zhou-lam3-alt}
	S.~Zhou and H.~Dong.
	\newblock Alternating groups and flag-transitive triplanes.
	\newblock {\em Des. Codes Cryptogr.}, 57(2):117--126, 2010.
	
	\bibitem{a:Zhou-lam3-excep}
	S.~Zhou and H.~Dong.
	\newblock Exceptional groups of {L}ie type and flag-transitive triplanes.
	\newblock {\em Sci. China Math.}, 53(2):447--456, 2010.
	
	\bibitem{a:Zhou-lam3-classical}
	S.~Zhou, H.~Dong, and W.~Fang.
	\newblock Finite classical groups and flag-transitive triplanes.
	\newblock {\em Discrete Math.}, 309(16):5183--5195, 2009.
	
\end{thebibliography}

\end{document}